\documentclass[11pt]{amsart}

\usepackage[margin=3.5cm]{geometry}
\usepackage{amssymb, amsmath}
\usepackage{amscd}
\usepackage[alwaysadjust]{paralist}
\usepackage{hyperref}
\usepackage[capitalize]{cleveref}
\usepackage{enumitem}
\usepackage{esint}
\usepackage{graphicx} 
\usepackage[rgb]{xcolor} 
\usepackage{verbatim}

\numberwithin{equation}{section}

\newtheorem{cor}[equation]{Corollary}
\newtheorem{lemma}[equation]{Lemma}
\newtheorem{prop}[equation]{Proposition}
\newtheorem{theorem}[equation]{Theorem}
\theoremstyle{definition}

\newtheorem{rem}[equation]{Remark}

\def\IN{\mathbb N}
\def\IR{\mathbb R}

\def\IZ{\mathbb Z}

\def\eps{\varepsilon}

\def\IRP{\mathbb RP}

\newcommand{\mult}{\operatorname{mult}}

\newcommand{\supp}{\operatorname{supp}}

\newcommand{\area}{\operatorname{area}}

\newcommand{\Sph}{\mathbb{S}}
\newcommand{\real}{\operatorname{Re}}

\begin{document}

\title{Handle attachment and the normalized first eigenvalue}
\author{Henrik Matthiesen}
\address
{HM: Department of Mathematics, University of Chicago,
5734 S. University Ave, Chicago, Illinois 60637}
\email{hmatthiesen@math.uchicago.edu}
\author{Anna Siffert}
\address
{AS: Max Planck Institute for Mathematics,
Vivatsgasse 7, 53111 Bonn}
\email{siffert@mpim-bonn.mpg.de}
\date{\today}

\begin{abstract}
We show that the first eigenvalue of a closed Riemannian surface normalized by the area can be strictly increased
by attaching a cylinder or a cross cap.
As a consequence we obtain the existence of maximizing metrics for the normalized first eigenvalue on any closed surface
of fixed topological type.
Since these metrics are induced by (possibly branched) minimal immersions into spheres, we find new examples of immersed minimal surfaces
in spheres.
\end{abstract}

\maketitle

\tableofcontents

\section{Introduction} \label{sec_intro}
For a closed Riemannian surface $(\Sigma,g)$ the (positive) Laplace operator acting on functions has discrete spectrum. 
We list its eigenvalues (counted with multiplicity) as
$$
0=\lambda_0(\Sigma,g) < \lambda_1(\Sigma,g) \leq \lambda_2(\Sigma,g) \dots \to \infty.
$$
The connection between the eigenvalues and the geometry of $(\Sigma,g)$ has drawn a lot of attention in modern differential geometry.
It turns out that there are some fundamental differences between the case of surfaces and higher dimensional manifolds.
Our interest here is in surfaces only and we restrict our discussion to this case.
There is a number of particularly beautiful results connecting the first eigenvalue and the topology of the underlying surface.
This starts with Hersch's celebrated sharp bound for
the first non-trivial eigenvalue on the two-sphere stating that
$$
\lambda_1(\Sph^2,g) \area(\Sph^2,g) \leq 8 \pi
$$
with equality only for round metrics \cite{hersch}.
The Hersch bound was then generalized by Yang and Yau in \cite{yang_yau} 
to orientable surfaces of higher genus, in which case one has that
$$
\lambda_1(\Sigma,g) \area(\Sigma,g) \leq 8 \pi (\gamma+1).
$$ 
A few years laters, El Soufi and Ilias \cite{EI} improved this bound to
\begin{equation} \label{eq_intro_upper}
\lambda_1(\Sigma,g) \area(\Sigma,g) \leq 8 \pi \left \lfloor \frac{\gamma+3}{2}\right \rfloor.
\end{equation}
Later, Li and Yau gave another proof of this bound using the conformal volume \cite{li_yau} which extends to non-orientable surfaces as well \cite{karpukhin}.
Typically, this bound is not sharp -- see e.g. \cite{esgj} and the discussion therein and also the recent preprint by Karpukhin \cite{karpukhin-2}, in which it is proved that
\eqref{eq_intro_upper} is strict if the genus of $\Sigma$ is at least three.
A natural question arising from these results is if there is a sharp upper bound for the scale invariant quantity
\begin{equation} \label{eq_intro_1}
\bar \lambda_1(\Sigma) = \lambda_1(\Sigma) \area(\Sigma),
\end{equation}
that depends only on the topology of the surface $\Sigma$.
Petrides proved in \cite{petrides} that, in order to rule out degenerations of the conformal class of a carefully chosen
maximizing sequence in moduli space, one has to show that the sharp constant for  \eqref{eq_intro_1} strictly increases in terms of the topology 
of the underlying surface, see also \cite{fs,MS1}.
A first example of this type of phenomena was already present in Nadirashvili's work \cite{nadirashvili}, where he deals with the moduli space of flat tori.
It was widely believed that this type of monotonicity should hold for surfaces of any genus.
The goal of this paper is to settle this conjecture by proving the following result.

\begin{theorem} \label{thm_main}
Let $(\Sigma,g)$ be a closed Riemannian surface with at most finitely many conical singularities.
If $\Sigma'$ is obtained from $\Sigma$ by attaching a cylinder or a cross cap, there is a smooth metric $g'$ on $\Sigma'$, such that
\begin{equation*}
\lambda_1(\Sigma,g') \area(\Sigma',g')> \lambda_1(\Sigma,g) \area(\Sigma,g).
\end{equation*}
\end{theorem}

Here, by an \emph{isolated conical singularity} we mean a point $x \in \Sigma$ near which the metric in appropriate coordinates is given by
$f(z)(|z|^k + O(|z|^{k+1}))|dz|^2$ as $z \to 0$ with $f$ a smooth positive function.
We want to point out that this is far more restrictive than the standard definition.
Although we do not think that this restriction is necessary in \cref{thm_main} it is sufficient in order to obtain \cref{thm_max} below and the argument gets a bit simpler.
We would like to point out that that the non-strict version of \cref{thm_main} has been obtained by Colbois and El Soufi in \cite{ces} using a result of Ann{\'e} \cite{anne}.

As indicated above, when combined with work of Petrides \cite{petrides} and work by the authors \cite{MS1}, this in particular implies the existence of maximizing metrics for the normalized first eigenvalue on any closed surface.

\begin{theorem} \label{thm_max}
Let $\Sigma$ be a closed surface. Then there is a metric $g$ on $\Sigma$, which is smooth away from at most
finitely many conical singularities, such that
$$
\lambda_1(\Sigma,h) \area(\Sigma,h) \leq \lambda_1(\Sigma,g) \area(\Sigma,g)
$$
for any smooth metric $h$ on $\Sigma$.
\end{theorem}

It is worth pointing out that these metrics are always induced by a branched minimal immersion $\Phi \colon \Sigma \to \mathbb{S}^N$ for some $N \geq 2$, \cite{Ilias_ElSoufi, nadirashvili}.
In particular, \cref{thm_max} immediately implies the following existence result for branched minimal immersions.

\begin{theorem} \label{thm_min}
Let $\Sigma$ be a closed surface.
Then there is a branched minimal immersion $\Phi \colon \Sigma \to \mathbb{S}^N$ for some $N \geq 2$,
which is
induced by first eigenfunctions of the metric $\Phi^* g_{\it{can.}}$.
\end{theorem}

By cominbining \cref{thm_main} and \cref{thm_max} we can obtain some additional information on the structure of these minimal immersions.
For a closed surface $\Sigma$, we write
$$
\Lambda_1(\Sigma)= \sup_{g} \bar \lambda_1(\Sigma,g),
$$
where the supremum is taken over all smooth metrics.
It is not hard to show that we can equivalently take the supremum over all metrics with isolated conical singularities.
Note that \cref{thm_main} and \cref{thm_max} imply that
\begin{equation} \label{eq_intro_mono}
\Lambda_1(\Sigma_{\gamma}) < \Lambda_1(\Sigma_{\gamma+1})
\end{equation}
for any $\gamma \in \IN$, where $\Sigma_{\gamma}$ denotes
the orientable surface of genus $\gamma$.
It is also known from \cite{bm, bbd} (see \cite{fs14}) that we have
\begin{equation} \label{eq_intro_asymp_lower}
\Lambda_1(\Sigma_\gamma) \geq \frac{3}{4} \pi (\gamma-1)
\end{equation}
for  $\gamma \in \IN$ sufficiently large.\footnote{The constant here can be improved according to the current state of the art concerning Selberg's conjecture.}
Let us moreover denote for $B \in \IN$ given by
$
N(B)
$
the number of natural numbers $\gamma \leq B$ such that any maximzing metric for $\lambda_1 \cdot \area$ on $\Sigma_\gamma$
is induced by a branched, full, minimal immersion $\Sigma_\gamma \to \Sph^N$ with $N \geq 3$.

\begin{cor} \label{cor_branched}
We have for the function $N(B)$ defined above that
$$
N(B) \gtrsim  \log(B).
$$
In particular, there are infinitely many values of $\gamma \in \IN$ such that no maximizing metric for the normalized first eigenvalue on $\Sigma_{\gamma}$ can be induced by a branched cover over $\Sph^2$.
\end{cor}

For $\gamma=3,4$ this was proved in \cite{karpukhin-2} (using \cref{thm_main}, more precisely the version already contained in \cite{MS2}).

Interestingly, our argument does not work for non-orientable surface since the best known constant in the analogue of \eqref{eq_intro_upper} for non-orientable surfaces
is $16 \pi$ \cite{karpukhin}.

\begin{proof}
Let $\gamma_* \in \IN$ be chose such that \eqref{eq_intro_asymp_lower} holds for any $\gamma \geq \gamma_*$.
We show that for any $\gamma_0 \geq \gamma_*$ any interval
 $[\gamma_0,\gamma_0 + \lceil (13 \gamma_0  + 51)/16\rceil +1 ]$
 contains some $\gamma \in \IN$ such no maximizing metric on $\Sigma_{\gamma}$ can be induced by a branched cover over $\Sph^2$.
This implies the assertion by an easy counting argument.
 
Assume that there is a  maximizing metric on $\Sigma_{\gamma_0}$
induced by a branched minimal immersion $\Sigma_{\gamma_0} \to \Sph^2$ for some $\gamma_0 \geq \gamma_*$.
Then, by the area formula, there is $l_0 \in \IN$ such that
$$
\Lambda_1(\Sigma_{\gamma_0}) = 8 \pi l_0 \geq \frac{3}{4} \pi (\gamma_0 - 1),
$$
where the last inequality is precisely \eqref{eq_intro_asymp_lower}.
If there are also maximizing metrics on $\Sigma_{\gamma_0+1} , \dots, \Sigma_{\gamma_0+k}$ induced by branched minimal immersions $\Sigma_{\gamma_0+j} \to \Sph^2$, we can combine this with \eqref{eq_intro_mono} to find that we need to have
$$
\Lambda_1(\Sigma_{\gamma_0+k}) \geq 8 \pi (l_0+k) \geq \frac{3}{4} \pi (\gamma_0 - 1) + 8 \pi k.
$$
For $ k \geq \lceil (13 \gamma_0 + 51)/16\rceil+1$ this contradicts the upper bound \eqref{eq_intro_upper} implying our claim.
\end{proof}

Some special cases of \cref{thm_max} were known for some time.
For the two-sphere this follows immediately from Hersch's sharp bound given above.
Li and Yau obtained the analogue of the Hersch bound for the real projective plane \cite{li_yau}.
The case of the torus was settled by Nadirashvili in his groundbreaking paper \cite{nadirashvili}.
His arguments are fundamentally different from all previous work.
Instead of proving a bound for all metrics that can be verified to be sharp for some metrics, he first proves the existence of a maximizer and then uses a result of Montiel and Ros \cite{montiel_ros} (see also \cite{ckm}) to conclude that the maximizing metric has to be flat.
This it then combined with earlier work of Berger that gives the sharp bound among flat metrics.
The case of the Klein Bottle follows from the work in \cite{ckm, esgj, JNP, nadirashvili}.

Very recently, the conjecture that the Yang--Yau inequality is sharp for a family of metrics on 
surfaces of genus two (see \cite{JLNN}),
was confirmed by Nayatani and Shoda in \cite{nayatani_shoda}. 
When combined with our work in \cite{MS2} this in particular already implies the existence of a maximizer on the surface of genus three.
Although we are only concerned with the first eigenvalue, let us mention that
there is also a number of results on higher eigenvalues -- see e.g.\
\cite{nadirashvili-sire, petrides-2}, and
 \cite{knpp, nadirashvili-2, nadirashvili_sire_3}, 
for the case of  $\Sph^2$,
and \cite{karpukhin-3,nadirashvili_penskoi} for the case of $\IRP^2$.
These results are somewhat different from those for the first eigenvalue since one can not rule out bubbling for higher eigenvalues.

Very similar results are known also for the corresponding problem in the case of the Steklov operator on compact surfaces with boundary.
This was pioneered by Fraser and Schoen in \cite{fs} culminating in their celebrated existence result for maximizers in the case of genus zero
surfaces.
A key step in their proof is to obtain a monotonicity result similar to \cref{thm_main}.
Although the overall strategy (in particular the idea to use a two parameter family of surfaces) and also some technical steps (e.g.\ the proof of \cref{thm_lower_neumann}) of our proof are strongly inspired by their ideas, 
our arguments are much longer and much more technical.
We want to emphasize that this is not of purely technical nature but that there is serious obstruction to obtain a convergence rate for the
first eigenvalue that is better than the gain in area for closed surfaces.
This is explained in \cref{sec_ideas} and follows from the results in \cite{MS2}.

Before we move on to the more technical aspects of the paper, we want to describe the na{\"i}ve idea of our strategy
to prove \cref{thm_main}. 
Starting from $(\Sigma,g)$ we attach a degenerating cylinder or cross cap, respectively, to $\Sigma$ and 
study the asymptotic behavior of the first eigenvalue.
As the cylinder or cross cap, respectively, becomes very small one hopes that the loss in the first eigenvalue is small compared to the gain in area coming from the attached part.

In the glueing construction that we use we have to deal with several new difficulties.
Some of them are imminent from our construction.
The degenerating handle that we attach is given by a carefully truncated negatively curved cusp.
The non-truncated cusp has continuous spectrum, which introduces an \emph{a priori} unbounded number of approximate eigenvalues on the glued surface that could be relevant to locate the first eigenvalue.
On the other hand, the gain in scaling, that originates precisely in the non-discreteness of the spectrum, for the first Dirichlet eigenfunction of these truncated cusp is one of the driving forces of our argument.
In view of the sharp asymptotics obtained by the authors in \cite{MS2}, some of these difficulties seem to be unavoidable if one tries to follow our approach.
Another key problem is to understand the interaction of the spectrum of the truncated cusp with the spectrum of $\Sigma$.
While the arguments in \cite{MS2} also rely on the same type of phenomenon, the arguments there do not give any control on a sufficiently small scale.

Finally, let us mention some interesting open problems.
By \eqref{eq_intro_upper} and \eqref{eq_intro_asymp_lower} we know that
$$
A=\limsup_{\gamma \to \infty} \frac{\Lambda_1(\Sigma_\gamma)}{\gamma} \in \left[ \frac{3}{4} \pi , 4 \pi \right],
$$
which is the asymptotically sharp value of the linear bound for the normalized first eigenvalue.
It would be interesting to know the exact value of $A$.
In a related direction, one can ask if the full limit of $\Lambda_1(\Sigma_\gamma)/\gamma$ exists for $\gamma \to \infty$.

A natural question arising from \cref{cor_branched} and \cite{karpukhin-2} is if there are examples of surfaces with genus at least five, or even with arbitrarily large genus,
for which a maximizing metric is induced by a branched cover of $\Sph^2$.
In a similar direction, one might also ask if those maximizing metrics that are not obtained as branched covers of $\Sph^2$ have a distinguished geometry, e.g.\ if they can be hyperbolic.

\smallskip

\textbf{Outline.}
In \cref{sec_attach} we explain the exact construction of the two parameter family of surfaces $\Sigma_{\eps,\kappa}$ used in the proof of \cref{thm_main}
and explain some of our main ideas. 
We start in \cref{sec_cyl} with a discussion of the spectrum of the attached cylinders and cross caps.
In \cref{sec_pt_bd} we prove pointwise estimates for eigenfunctions with uniformly bounded eigenvalue 
near the points, where we attach the cylinder.
In the case of Neumann eigenfunctions on $\Sigma$ with a tiny ball removed one can also prove an $L^2$-gradient estimate along the boundary of that ball, this is done in
\cref{sec_neu}.
As an application we prove that the first eigenvalue has multiplicity one in many cases.
Next, as a warmup, in \cref{sec_eig_lim} we use the pointwise estimates from \cref{sec_pt_bd} to compute the spectrum as the cylinder or cross cap collapses, but without any good control on the rate of convergence.
In \cref{sec_quasimodes} we construct several good approximate solutions to the eigenvalue equation.
These are used in \cref{sec_conc} to get a good description of the first two eigenfunctions.
Finally, in \cref{sec_asymp} we combine the results from \cref{sec_neu} and \cref{sec_conc} to get a comparison result for the first eigenvalue of $\Sigma_{\eps,\kappa}$ and the first eigenvalue of $\Sigma$, which implies the the main technical result by an iteration argument (across scales) along carefully selected parameters $\kappa$. Finally, we establish the main result in \cref{sec_proofs_main_2} by a smoothing argument.

\smallskip

\textbf{Acknowledgements.}
The first named author would like to thank his former advisor Werner Ballmann for a helpful discussion on Green's functions and Andr{\'e} Neves for asking him the question leading to \cref{cor_branched}.
The second named author would like to thank the Max Planck Institute for Mathematics in Bonn for financial support and excellent working conditions.

\section{The surfaces $\Sigma_{\eps,\kappa}$ and main ideas of the proof} \label{sec_attach}

The proof of \cref{thm_main} consists of two separate but very similar parts.
Namely, we need to show the same type of result for attaching cylinders and for attaching cross caps. 
In the following two subsections we describe for these two cases the construction of two-parameter families of surfaces $\Sigma_{\eps,\kappa}$ that we will use as competitors for the problem.
In Subsection\,\ref{main_technical} we state and discuss the main technical result needed to prove
\cref{thm_main}. 
Finally, in the last subsection we give some ideas of the proof of the main technical result.

\subsection{Attaching a cylinder} \label{attach_cylinder}
Let $(\Sigma,g)$ be a closed Riemannian surface with at most finitely many conical singularities.
We start with the case of attaching a cylinder to $\Sigma$.

\smallskip

Given $\eps>0$ and $\alpha \in (1/3,\infty)$, we let 
$$
C_{\eps,\kappa} = \Sph^1 \times [1,R_{\eps,\alpha}],
$$
where $R_{\eps,\alpha} = \exp(1/\eps^{\alpha})$, endowed with the metric 
$$
g_{\eps,\kappa}=
\frac{1}{(\kappa y)^2} (\eps^2 d\theta^2 + dy^2),
$$
with $\theta \in [0,2\pi), y \in [1,R_{\eps,\alpha}]$, and $\kappa\in (0,\infty)$.
Note that this is a truncation of the standard constant curvature cusp of curvature $-\kappa^2$. 
Later on, the parameter $\alpha$ will be fixed, whereas $\eps$ and $\kappa$ will change.
In order to keep the notation a bit simpler we will not indicate the dependence on $\alpha$, unless it is of any particular importance. 
\smallskip

Let $x_0,x_1 \in \Sigma$ be two distinct points, such that $g$ is smooth near both of them.
Since $\Sigma$ is a surface it is locally conformally flat. In other words, we
can pick coordinate 
neighbourhoods $U_i$ containing $x_i,$
such that $g$ is conformal to the Euclidean metric in $U_i,$ that is $g=f g_e$ with $f$ a smooth, positive function and $g_e$ the Euclidean metric.
By rescaling these coordinates if necessary, we may assume that $f(x_i)=1$.
For given $k\in\mathbb{N}$, let $B_{\eps^k}(x_i)= B_{g_e}(x_i,\eps^k)$ be a ball centered at $x_i$ with radius equals $\eps^k$ with respect to $g_e.$
We then let 
$$
\Sigma_{\eps,\kappa} = 
\Sigma \setminus (B_{\eps^k}(x_0) \cup B_{\eps^k}(x_1)) \cup_{\partial C_{\eps,\kappa}} C_{\eps,\kappa},
$$
where we attach the component $\Sph^1 \times \{1\}$ of $\partial C_{\eps,\kappa}$ to $\partial B_{\eps^k}(x_0)$ and the component
$\Sph^1 \times \{R_{\eps,\alpha}\}$ of $\partial C_{\eps,\kappa}$ to $\partial B_{\eps^k}(x_1)$.
This surface comes naturally endowed with a singular metric given by the metric $g$ on $\Sigma \setminus (B_{\eps^k}(x_0) \cup B_{\eps^k}(x_1))$
and by the metric $g_{\eps,\kappa}$ on $C_{\eps,\kappa}$.
It will be important later, that, for any $\eps,\kappa$ fixed, this metric is the limit of smooth metrics in an appropriate sense (see \cref{sec_proofs_main_2} for details).

\begin{figure}[ht]
	\centering
  \includegraphics[width=12cm]{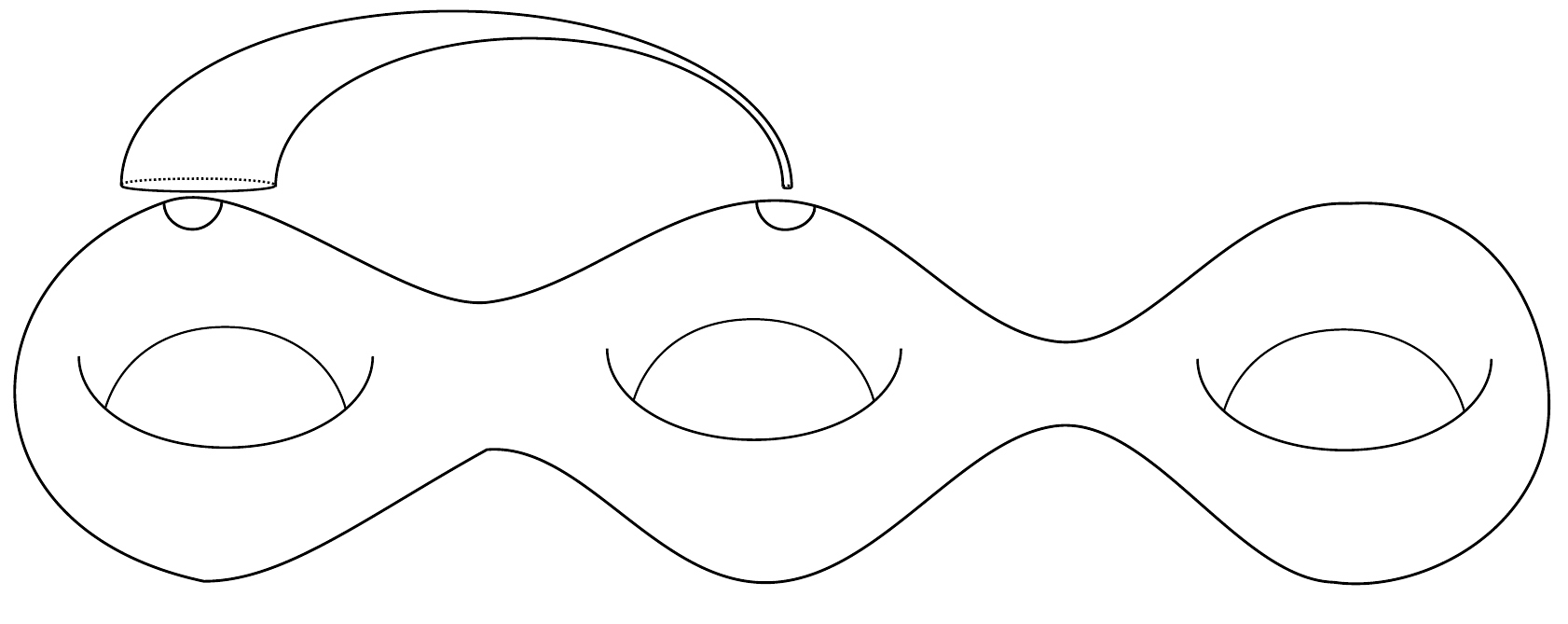}
	\caption{The construction of $\Sigma_{\eps,\kappa}$ involving the three different scales: the long boundary component of length $\sim\eps$, the short boundary component of length $\sim e^{-1/\eps^\alpha}$, the removed balls of radius $\sim \eps^k$. }
	\label{fig}
\end{figure}

\subsection{Attaching a cross cap} \label{attach_cross_cap}

The construction for attaching cross caps is similar to that of attaching a cylinder.

\smallskip

We start with $D_{\eps,\kappa}$ which is given by two copies $C_1$ and $C_2$ of $C_{\eps,\kappa}$ glued together along the boundary components $\Sph^1 \times \{R_{\eps,\alpha}\}$.
The group $\IZ/2\IZ$ acts freely and isometrically on $D_{\eps,\kappa}$ by $ C_1 \ni (\theta , y) \mapsto (\theta + \pi, y) \in C_2$, where we identify $\Sph^1 = \IR/ 2 \pi \IZ$.
The quotient 
$$
M_{\eps,\kappa} = D_{\eps,\kappa} / (\IZ / 2 \IZ)
$$
has the topology of a cross cap (also called a M{\"o}bius strip).

Let $x_0 \in \Sigma$ such that $g$ is smooth near $x_0.$
Let $U$ be a coordinate neighborhood containing $x_0,$
such that $g$ is conformal to the Euclidean metric in $U,$ that is $g=f g_e$ with $f$ a smooth, positive function and $g_e$ the Euclidean metric,
where we may assume also here that $f(x_i)=1$.
Let $B_{\eps^k}(x_0)=B_{g_e}(x_0,\eps^k)$ be a ball centered at $x_0$ with radius $\eps^k$ with respect to $g_e.$
We let
\[\
\Sigma_{\eps,\kappa}=(\Sigma \setminus B_{\eps^k}(x_0)) \cup_{\partial M_{\eps,\kappa}} M_{\eps,\kappa}.
\]
As for the case of cylinders this comes with the metric given by $g$ on $\Sigma \setminus B_{\eps^k}(x_0)$ and the metric induced by $g_{\eps,\kappa}$ on $M_{\eps,\kappa}$.
Also in this case, for $\eps$ and $\kappa$ fixed, the metric $g_{\eps,\kappa}$ can be obtained as the limit of smooth metrics.

\subsection{The main technical result}\label{main_technical}

In both constructions, we restrict to parameters $\kappa \in [\kappa_0,\kappa_1]$,
where $\kappa_i$ are chosen such that
$$
\kappa_0^2/4 < \lambda_1(\Sigma) < \kappa_1^2/4 < \lambda_{K+1}(\Sigma),
$$ 
where $K=\mult(\lambda_1(\Sigma))$ denotes the multiplicity of $\lambda_1(\Sigma)$.
We can now state the precise technical result that will easily imply \cref{thm_main}.

\begin{theorem} \label{thm_main_technical}
For $k \geq 9,$ there is $\alpha_0>1/3$ such that in both constructions of the surfaces $\Sigma_{\eps,\kappa}$  above with fixed $\alpha \in (1/3,\alpha_0)$, 
there are parameters $\kappa_\eps \in [\kappa_0,\kappa_1]$,
such that
\begin{equation} \label{eq_main_technical}
\lambda_1(\Sigma_{\eps,\kappa_{\eps}}) \geq \lambda_1(\Sigma) - o(\eps)
\end{equation}
as $\eps \to 0$.
\end{theorem}

In view of our results in \cite{MS2} we expect the restriction $\alpha>1/3$ to be unavoidable to get the bound \eqref{eq_main_technical}
for the specific parameters $\kappa_\eps$ that we use.
We do not know, how relevant the restriction $\alpha<\alpha_0$ really is.
In our arguments it is used to approximate eigenvalues and eigenfunctions efficiently. 
It seems that our proof should work with $\alpha_0=1/2$. 
We decided not to keep track of this too carefully anymore towards the end of the proof to make the argument not even more technical than it already is.
One way in which the assumption on $\alpha < \alpha_0$ arises is that we can locate eigenvalues on $\Sigma_{\eps,\kappa}$ corresponding to the Dirichlet spectrum of $C_{\eps,\kappa}$ up to scale $\eps^{3\alpha/2 +1/2}$.
On the other hand, the gap between two consecutive Dirichlet eigenvalues of $C_{\eps,\kappa}$ is on scale $\eps^{2 \alpha}$.
Therefore, it is useful to assume at least $\alpha<1$.
This allows us to treat things very much as if the spectrum of the hyperbolic cusp were discrete, although it is crucial for our construction that the spectrum is in fact continuous.
In a similar fashion the assumption on $k$ is very non-sharp. 
It should be easy to replace it at least by $k>1$.

\subsection{Some ideas in the proof of \cref{thm_main_technical}}  \label{sec_ideas}
The heart of our proof consists of two parts that we try to illuminate a bit in this section.
For simplicity, let us assume for the moment that we attach a cross cap to $\Sigma$ close to a point $x_0 \in \Sigma$, and that $\lambda_1(\Sigma)$ is simple.
Moreover, we also assume that the up to sign unique normalized first $\lambda_1(\Sigma)$-eigenfunction $\phi_0$ satisfies
\begin{equation} \label{eq_obs}
\phi_0(x_0)\neq 0.
\end{equation}
Of course, since we assume $\lambda_1(\Sigma)$ to be simple we could always choose $x_0$ such that $\phi_0(x_0)=0$, which makes it much easier to obtain the conclusion from \cref{thm_main} (see e.g.\ \cite[Theorem 1.1]{MS2}).
However, most of the main problems and ideas needed for the general case are well illustrated in the situation described above.

A rough summary of the two main steps is as follows.
The first one is to locate eigenvalues on $\Sigma_{\eps,\kappa}$, that are contained in $[\delta,\Lambda_{K+1}(\Sigma)-\delta]$ for some small $\delta>0$, relative to $\lambda_1(\Sigma)$ or a Dirichlet eigenvalue of $M_{\eps,\kappa}$.
At the same time, we obtain good quantitative control on how the corresponding eigenfunctions on $\Sigma_{\eps,\kappa}$ can be decomposed into linear combinations of  $\phi_0$ and Dirichlet eigenfunctions on $M_{\eps,\kappa}$.
We will explain below, that the choice of the metrics on $M_{\eps,\kappa}$ is exactly in such a way that we can locate those eigenvalues of $\Sigma_{\eps,\kappa}$ 
corresponding to Dirichlet eigenvalues of $M_{\eps,\kappa}$ very efficiently on a scale below $\eps$.
On the other hand, there is a serious obstacle originating from \eqref{eq_obs}.
In general, because of \eqref{eq_obs}, we can locate the corresponding eigenvalue on $\Sigma_{\eps,\kappa}$ only up to scale $\eps \log(1/\eps)$.
The main idea then, which is strongly inspired by \cite{fs}, is that there has to be a choice of $\kappa$ depending on $\eps$ for which there is some interaction of the two different types of eigenvalues.\footnote{It is easy to arrange this in such a way that we can observe this for the first eigenvalue of $\Sigma_{\eps,\kappa}$.}
Here, by interaction we mean that the corresponding eigenfunction has some amount of its $L^2$-norm concentrated on $\Sigma \setminus B_{\eps^k}$, but also some amount on $M_{\eps,\kappa}$.
In other words, the picture of having eigenvalues of two different types is not very accurate in some regime of parameters $\kappa$, while it is actually very accurate near $\kappa_0$ and $\kappa_1$.
In particular, at $\kappa_0$ the first eigenfunction has to be very concentrated on $M_{\eps,\kappa_0}$, and
at $\kappa_1$ the first eigenfunction has to be very concentrated on $\Sigma \setminus B_{\eps^k}$.
For this parameter it is then easy to locate the first eigenvalue very precisely in relation to $\lambda_0(M_{\eps,\kappa})$
thanks to the good convergence rate of these types of eigenvalues.
This part of the argument is quite lengthy and technical.
The reason for this is twofold.
On the one hand, we need to obtain quite strong quantitative control of the decomposition of eigenfunctions for the second step.
On the other hand, we have to deal with an unbounded number of Dirichlet-eigenvalues on $M_{\eps,\kappa}$ close to $\lambda_0(M_{\eps,h})$.
However, the problem that remains to be overcome turns out to be much more subtle.
Namely, in a second step we still have to show that the interaction phenomenon can only occur when $\lambda_0(M_{\eps,\kappa})$ and $\lambda_1(\Sigma)$ are very close to each other.
This is one of the key difficulties that is not present in \cite{fs}.
We obtain this by means of a subtle iteration argument (across scales) successively improving our understanding on smaller scales.
One of the key computations here is already contained in the proof of \cref{prop_first_simple}.
This computation, which heavily relies on some special properties of the cross caps $M_{\eps,\kappa}$,
suggests
 that if the gap $\lambda_2(\Sigma_{\eps,\kappa}) - \lambda_1(\Sigma_{\eps,\kappa})$ is small one can hope\footnote{This gap is only the enumerator of the error term not present in \cref{prop_first_simple}, which is why the actual argument is much more difficult than the proof of \cref{prop_first_simple}.}
 to show that also $\lambda_1(\Sigma)-\lambda_1(\Sigma_{\eps,\kappa})$ is very small.
We achieve this final step by building a test function for the first Neumann eigenvalue
of $\Sigma \setminus B_{\eps^k}(x_0)$\footnote{This is very close to $\lambda_1(\Sigma)$, see \cref{sec_neu}.}
out of the first two eigenfunctions on $\Sigma_{\eps,\kappa}$.
Controlling the error terms in this computation heavily relies on the good approximate solutions to the eigenvalue equation constructed in \cref{sec_quasimodes} and the quantitative decomposition of eigenfunctions, that we obtain in \cref{sec_conc}.
Because of the specific form of the error terms arising the estimate that we obtain this way turns out to be useful only on very specific scales of the interaction phenomenon.
By carefully adjusting $\kappa$, adopted to different ratios of concentration, this allows us to iteratively improve the size of the gap $\lambda_1(\Sigma)-\lambda_0(M_{\eps,\kappa})$ across scales until it is of size $o(\eps)$, while at the same time we keep control on the distance of the first eigenvalue to $\lambda_0(M_{\eps,\kappa})$.

The motivation for our construction stems from \cite[Theorem 1.6]{MS2}, which gives strong information what properties the metric on the attached 
cross cap should have, if one wants to use parameters $\kappa$ for which $\lambda_0(M_{\eps,\kappa})$ and $\lambda_1(\Sigma)$ are close to each other.
In general the rate of convergence in this range can in general not be expected to be any better than
\begin{equation} \label{eq_ideas_approx}
a_{\eps,\kappa,0}= - \int_{\partial M_{\eps,\kappa}} \partial_\nu \psi_{\eps,\kappa,0} \, d \mathcal{H}^1,
\end{equation}
where $\psi_{\eps,\kappa,0}$ is a non-negative, normalized first Dirichlet eigenfunction on $M_{\eps,\kappa}$ and $\nu$ the outward pointing normal vector field along $\partial M_{\eps,\kappa}$.
In \cite{MS2} this was proved if the attached cross cap carries the product metric that collapses along the fibers.
In that case the term is of order $\eps^{1/2}$ resulting in the sharp convergence rate $\eps^{1/2}$ for parameters as above.
Also note that $\eps^{1/2}$ is somewhat the natural scaling of the problem as long as the metric on the collapsing part resembles a model
with an isolated eigenvalue at $\lambda_0$.
At the same time, $a_{\eps,\kappa,0}$ is also the scale up to which we are able to locate eigenvalues with eigenfunction that have some $L^2$-norm concentrated on $M_{\eps,\kappa}$.

Motivated by the remarks above, our construction starts from the well-known observation that the standard hyperbolic cusp
$$
\mathcal{C}=\{z = x+iy  \in \mathbb{H}^2  \ : \  y\geq 1\} / (z \mapsto z+1),
$$ 
which has finite area, and continuous spectrum $[1/4,\infty)$,
admits a generalized $\lambda_0$-eigenfunction that lies in $L^1(\mathcal{C})$ but not in $L^2(\mathcal{C})$ (hence the word generalized).
In particular, the $L^2$-normalized $\lambda_0$-eigenfunctions of a compact exhaustion of $\mathcal{C}$ have $L^1$-norm tending to zero\footnote{Note that, by integration by parts, the $L^1$-norm corresponds to \eqref{eq_ideas_approx} up to a multiplicative constant.}.
The truncation of the cusp in our construction of the surfaces $\Sigma_{\eps,\kappa}$ above is chosen in exactly such a way that the failure of the generalized $L^1$-normalized eigenfunction to be in $L^2$
can be observed on a specified scale.

An important consequence of this is that one can find very good approximate eigenfunctions on $\Sigma_{\eps,\kappa}$ by extending Dirichlet eigenfunctions appropriately (see \cref{sec_quasimodes}).
More precisely, the scale on which these will fail to solve the eigenvalue equation is (up to sign and multiplicative constant) precisely given by $a_{\eps,\kappa,0} \sim \eps^{3\alpha/2+1/2}$
if $\psi_{\eps,\kappa}$ is an
$L^2$-normalized Dirichlet eigenfunction on $M_{\eps,\kappa}$.
For $\alpha>1/3$ this will be of size $o(\eps)$, so that we have very good control (in comparison with the area term on scale $\eps$)
on the convergence rate of eigenvalues corresponding to Dirichlet eigenfunctions on $C$.

\section{Spectrum and eigenfunctions of the truncated hyperbolic cusp}  \label{sec_cyl}
In this section we discuss the spectrum and eigenfunctions of the cylinders and cross caps that we attach to the initial surface $\Sigma$.
Since we scale down the metric in the fiber direction, mainly the rotationally symmetric part of the spectrum will be relevant for us.

\smallskip

For $R>0$ fixed, we write 
$$
C = \Sph^1 \times [1,R]
$$
endowed with the metric 
$$
\frac{1}{\kappa^2 y^2} (d \theta^2 + dy^2).
$$

Let $D$ be given by two copies $C_1$ and $C_2$ of $C$ glued to each other along the boundary component $\{y=R\}$.
On $D$, there is a fixed-point free, isometric involution $\tau$ given by $C_1 \ni (\theta , y) \mapsto (\theta + \pi, y) \in C_2$, where we identify $\Sph^1 = \IR/ 2 \pi \IZ$.
We denote by 
$$
M =D/ \tau
$$
the corresponding quotient that has the topology of a cross cap (also called M{\"o}bius strip).
The involution $\tau$ induces an involution $\tau^*$ on $L^2(M)$, which splits the latter into the $+1$ and the $-1$ eigenspaces corresponding
to even respectively odd functions with respect to $\tau$ on $D$.
This in turn implies that any eigenvalue problem on $D$ splits into the two separate corresponding eigenvalue problems on $C$.
Along $\{y=R\}$ the first of these has additional Dirichlet boundary conditions (corresponding to the $-1$ eigenspace of $\tau^*$) and the second one 
additional Neumann boundary conditions (corresponding to the $+1$-eigenspace of $\tau^*$)\footnote{Note that the eigenfunctions are at least $C^{1,\alpha}$ since the metric is Lipschitz.}.
Finally, note that the latter one corresponds to the corresponding eigenvalue problem on $M$, since the lifted functions are exactly the even functions.

\smallskip

Next, we note that for $t \in \IR$ the functions
$$
f_t(\theta,y) = \sin(t/\kappa \log y) y^{1/2}
$$ 
and 
$$
g_t(\theta,y) = \cos(t/\kappa \log y) y^{1/2}
$$
are both solutions to the eigenvalue equation $(\Delta - \lambda)u=0$ on $C$ with
$$
\lambda_t = \frac{\kappa^2}{4}+t^2.
$$
By the uniqueness of solutions to the corresponding ODE, it follows that any rotationally symmetric eigenfunction on $C$ with eigenvalue at least $\kappa^2/4$ has to 
be a linear combination of $f_t$ and $g_t$ for some $t \in \IR$.
Also note that, since $f_{-t}=-f_t$ and $g_{-t}=g_{t}$, it is sufficient to consider $t \in [0,\infty)$.

\smallskip

We denote by $(\lambda_l(C))_{l \geq 0}$ the Dirichlet eigenvalues and by
$(\mu_l(C))_{l \geq 0}$ the Neumann eigenvalues of the Laplacian on $C$.
Moreover, we write $(\lambda_l^0(C))_{l \geq 0}$ for the Dirichlet eigenvalues and
$(\mu_l^0(C))_{l \geq 0}$ for the Neumann eigenvalues of the Laplacian acting on rotationally symmetric 
functions on $C$.
We use analogous notation for eigenvalue problems on $M$.

\begin{lemma} \label{lem_spec_cyl}
For $l \geq 0$, we have that 
$$
\lambda_l^0(C) =\frac{\kappa^2}{4} + (l+1)^2\frac{\kappa^2 \pi^2}{\log^2R}.
$$
for the rotationally symmetric part of the Dirichlet spectrum.
\end{lemma}

\begin{proof}
For any $t \in [0,\infty)$, the function $f_t$ vanishes on $\{y=1\}$.
If we take $t_0 = \kappa \pi / \log R$, the function $f_{t_0}$ vanishes also on $\{y=R\}$ and is positive in the interior of $C$.
In particular, we find that $\lambda_0(C) = \kappa^2/4 + t_0^2$.
As explained above, this implies that any rotationally symmetric Dirichlet eigenfunction on $C$ is given by some $f_t$.
Consequently, since $f_t(R)=0$ if and only if $t = \frac{(l+1)\kappa \pi}{\log R}$ for $l \in \IN$, the assertion follows.
\end{proof}

\begin{lemma} \label{lem_gr_st_cc}
We have that
$$
\frac{\kappa^2}{4} \leq \lambda_0(M) \leq \frac{\kappa^2}{4} + \frac{\kappa^2 \pi^2}{\log^2R} 
$$
\end{lemma}

\begin{proof}
As explained above, a rotationally symmetric Dirichlet eigenfunction on $M$ corresponds to a rotationally symmetric eigenfunction on $C$ with Dirichlet boundary
conditions along $\{y=1\}$ and Neumann boundary conditions along $\{y=R\}$.
Therefore, since the metric on $C$ is conformal to the product metric, $f_t$ is such an eigenfunction if and only if
$$
\left. \partial_y f_t \right|_{y=R}
= 0.
$$
The function $t \mapsto \left. \partial_y f_t \right|_{y=R} $ is continuous.
Moreover, it is positive for $t>0$ very small and negative for $t=\kappa \pi / \log R$.
It follows that there is $0<t_0<\kappa \pi / \log R$ such that $f_{t_0}$ satisfies the Neumann boundary conditions along $\{y=R\}$.
Since $t_0<\kappa \pi / \log R$ it follows that $f_{t_0}$ is non-negative, which in particular implies that it has to be a $\lambda_0(M)$-eigenfunction.
\end{proof}

We have the following general comparison principle for the rotationally symmetric part of the spectrum on $M$. 

\begin{lemma} \label{lem_basic_spec_comp}
For any $l \in \IN$, the rotationally symmetric Dirichlet and Neumann eigenvalues satisfy
$$
\mu_l^0(M) \leq \lambda_l^0(M) \leq \mu_{l+1}^0(M).
$$
\end{lemma}

\begin{proof}
The first of these inequalities always holds by the min-max principle.
To prove the second inequality, we take a non-trivial linear combination $w$ of the first $(l+2)$ rotationally symmetric Neumann eigenfunctions, which is orthogonal to the first $l$ rotationally symmetric Dirichlet eigenfunctions and vanishing along $\partial M$. 
We can do this for dimensional reasons using that $\partial M$ is connected and rotationally symmetric.
We then find that
\begin{equation*}
\lambda_l^0 (M) \leq \frac{\int_M |\nabla w|^2}{\int_M |w|^2} \leq \mu_{l+1}^0(M),
\end{equation*}
which establishes the claim.
\end{proof}

We now compute the rotationally symmetric part of the Neumann eigenvalues of $M$ explicitly.

\begin{lemma} \label{lem_spec_neu}
For $l \geq 1$, we have that
$$
\mu_l^0(M) =\mu_l^0(C)= \frac{\kappa^2}{4} + l^2\frac{\kappa^2 \pi^2}{\log^2R}
$$
\end{lemma}

\begin{proof}
As explained earlier, it suffices to prove the second equality.
Since $\mu_1^0(C) = \mu_1^0(M)\geq \lambda_0^0(M) \geq \kappa^2/4$, by \cref{lem_gr_st_cc} and \cref{lem_basic_spec_comp}, it suffices to consider functions $af_t + bg_t$ with $a,b \in \IR$ and $t \geq 0$.
Now the function $a f_t + b g_t $ satisfies the Neumann boundary condition along $\{y=1\}$ if and only if
$2 at/\kappa =- b$.
Therefore, the Neumann boundary condition along $\{y=R\}$ implies that $\sin(t \log R / \kappa)=0$, which easily implies the assertion
\end{proof}

From \cref{lem_spec_cyl} and \cref{lem_spec_neu} we immediately obtain the following corollary.

\begin{cor} \label{spec_comp_dir_neu}
We have that
$$
\mu_1^0(C) \geq \lambda_0^0(C).
$$
\end{cor}

Finally, we need a lower bound on the $L^2(C)$-norm for an eigenfunction $f_{t_l}$.

\begin{lemma} \label{spec_l2_bd}
We have that
$$
\int_C |f_{t_l}|^2 \geq \frac{\pi}{2 \kappa^2} \log R
$$
for any $\lambda_l^0(C)$ eigenfunction $f_{t_l}$.
\end{lemma}

\begin{proof}
Note that $|\sin((t_l/\kappa) \log y )| \geq 1/\sqrt{2}$ if and only if $(t_l/ \kappa) \log y  \in [(4j+1)\pi/4,(4j+3)\pi/4 ]$ for some $j \in \IN$.
Let $I_j^l:=\left[\exp\left(\frac{(4j+1)\kappa \pi}{4t_l}\right),\exp\left(\frac{(4j+3) \kappa \pi}{4t_l}\right)\right]$.
Since $t_l=(l+1) \kappa \pi / \log R$, we have that $I_0^l,\dots,I_l^l \subset [1,R]$.
Therefore, we obtain
\begin{equation*}
\begin{split}
\int_C |f_{t_l}|^2
&\geq
\frac{2 \pi}{\kappa^2}  \int_1^R |\sin( (t_l / \kappa ) \log y )|^2 y \frac{dy}{y^2}
\geq
\frac{\pi}{\kappa^2}\sum_{j=0}^{l} \int_{I_j^l} \frac{1}{y} dy
\\
& \geq
\frac{\pi}{\kappa^2}\sum_{j=0}^{l} \frac{2 \kappa \pi}{4 t_l} 
=
\frac{\pi}{\kappa^2} \sum_{j=0}^{l}  \frac{\log R}{2(l+1)}
\geq \frac{\pi}{2 \kappa^2} \log R,
\qedhere
\end{split}
\end{equation*}
\end{proof}

We would like to remark that the bound is sharp in its dependence on $R$, since the function $y \mapsto y^{1/2}$ has $L^2$-norm $\log^{1/2}R$.

We now return to the degenerating family of cylinders $C_{\eps,\kappa}$, i.e.\ $C$ with 
$R_\eps=\exp(1/\eps^\alpha)$ and metric 
$$
g_{\eps,\kappa} = \frac{1}{(\kappa y)^2} (\eps^2 d\theta^2 + dy^2).
$$
Similarly, we have the family of degenerating cross caps $M_{\eps,\kappa}$ obtained from the cylinder $C_{\eps,\kappa}$ by doubling and dividing by the action of the involution $\tau$.

Recall that we have fixed $0<\kappa_0<\kappa_1$.
First of all, note that there is some $\eps_0>0$ such that if $\eps \leq \eps_0$, then any eigenfunction on $C_{\eps,\kappa}$ with eigenvalue below $\lambda_{K+1}(\Sigma)+1$ has to be rotationally symmetric, where we denote by $K$ the multiplicity of $\lambda_1(\Sigma)$.
In fact, it follows immediately from separation of variables, that all non rotationally symmetric eigenfunctions need to have eigenvalue at least $\kappa_0^2/\eps^2$. for any $\kappa \in [\kappa_0,\kappa_1]$.
From now on, we will only consider parameters $\eps \leq \eps_0$.
In particular, this ensures that all the results from above apply to all the eigenvalues on $C_{\eps,\kappa}$ below $\lambda_{K+1}(\Sigma)+1$.
Of course, the very same discussion applies to the cross cap $M_{\eps,\kappa}$.

The boundary of $C_{\eps,\kappa}$ consists of two connected components.
We will denote these by
$\partial_1 C_{\eps,\kappa} :=\{y=1\}$ and $\partial_{R_\eps} C_{\eps,\kappa}:=\{y=R_\eps\}$.

\begin{cor} \label{spec_cusp_cor_approx_order}
There is a constant $C=C(\kappa_0,\kappa_1)$ with the following property.
Assume that $\lambda_l(C_{\eps,\kappa}) \leq \lambda_{K+1}(\Sigma)+1$ and denote by
$\psi_{\eps,\kappa,l}$
an $l$-th normalized Dirichlet eigenfunction on $C_{\eps,\kappa}$.
We have 
\begin{equation} \label{eq_spec_cusp_cor_approx_order}
\left| \int_{\partial_1 C_{\eps,\kappa}} \partial_\nu \psi_{\eps,\kappa,l} d\mathcal{H}^1\right| \leq C (l+1) \eps^{3/2 \alpha +1/2}
\end{equation}
and
\begin{equation} \label{eq_spec_cusp_cor_approx_order_2}
\left| \int_{\partial_{R_\eps} C_{\eps,\kappa}} \partial_\nu \psi_{\eps,\kappa,l} d\mathcal{H}^1\right| \leq C (l+1) \eps^{3/2 \alpha +1/2} R_\eps^{-1/2}
\end{equation}
if $\kappa \in [\kappa_0,\kappa_1]$ and $\eps \leq \eps_0$.
\end{cor}

\begin{proof}
By the preceding lemma, the normalized eigenfunction is given by
$$
\left( c_{\eps,\kappa,l} \eps \right)^{-1/2} \sin((t_l/\kappa) \log y) y^{1/2}
$$
with $c_{\eps,k,l} \geq C \log R_\eps = C \eps^{-\alpha}$, where $C=C(\kappa_0,\kappa_1)$.
Along the longer boundary component $\partial_1 C_{\eps,\kappa}$ we have 
$$
\partial_\nu \left( \sin((t_l/\kappa) \log y) y^{1/2} \right) = \kappa \left. \partial_y \left( \sin((t_l/\kappa) \log y) y^{1/2} \right) \right|_{y=1} = t_l,
$$
and $|t_l| \leq C(l+1) \eps^\alpha$.
Along the shorter boundary component $\partial_{R_\eps} C_{\eps,\kappa}$ we have 
$$
\partial_\nu \left( \sin((t_l/\kappa) \log y) y^{1/2} \right)= \kappa \left. y \left( \partial_y  \sin((t_l/\kappa) \log y) y^{1/2}\right) \right|_{y=R_\eps} = t_l R_\eps^{1/2}.
$$
Therefore we find 
\begin{equation*}
\begin{split}
\left| \int_{\partial_1 C_{\eps,\kappa}} \partial_\nu \psi_{\eps,\kappa,l} d\mathcal{H}^1\right| 
&\leq
C (l+1) \left( c_{\eps,\kappa,l} \eps \right)^{-1/2} \eps^{\alpha} \eps
\\
&\leq C (l+1) \eps^{3\alpha/2+1/2}.
\end{split}
\end{equation*}
Similarly, we get
\begin{equation*}
\begin{split}
\left| \int_{\partial_{R_\eps} C_{\eps,\kappa}} \partial_\nu \psi_{\eps,\kappa,l} d\mathcal{H}^1\right| 
&\leq
C (l+1) \left( c_{\eps,\kappa,l} \eps \right)^{-1/2} \eps^{\alpha}R_{\eps}^{1/2} \eps R_\eps^{-1}
\\
&\leq C (l+1) \eps^{3\alpha/2+1/2} R_\eps^{-1/2},
\end{split}
\end{equation*}
which is exactly the assertion.
\end{proof}

In the case of a cross cap the boundary has only one connected component. 
The almost exact same computation that led to the estimate along the longer boundary component $\partial_1C_{\eps,\kappa}$ of $C_{\eps,\kappa}$ 
implies the analogous estimate for the cross caps $M_{\eps,\kappa}$.
There are some minimal changes needed for the $\lambda_0$-eigenfunction.

\begin{cor} \label{spec_cross_cor_approx_order}
There is a constant $C=C(\kappa_0,\kappa_1)$ with the following property.
Assume that $\lambda_l(M_{\eps,\kappa}) \leq \lambda_{K+1}(\Sigma)+1$ and denote by
$\psi_{\eps,\kappa,l}$
an $l$-th normalized Dirichlet eigenfunction on $M_{\eps,\kappa}$.
We have that
\begin{equation} \label{eq_spec_cross_cor_approx_order}
\left| \int_{\partial M_{\eps,\kappa}} \partial_\nu \psi_{\eps,\kappa,l} d\mathcal{H}^1\right| \leq C (l+1) \eps^{3/2 \alpha +1/2},
\end{equation}
if $\kappa \in [\kappa_0,\kappa_1]$ and $\eps \leq \eps_0$.
\end{cor}

\begin{rem} \label{spec_rem_sharp}
Note that, by the remark on the sharpness of the bound in \cref{spec_l2_bd} and our computation of the eigenvalues, we find in particular that
\eqref{eq_spec_cusp_cor_approx_order} and \eqref{eq_spec_cross_cor_approx_order} are sharp in the exponent of $\eps$ for the first eigenfunction.
\end{rem}

\section{Pointwise estimates for eigenfunctions}  \label{sec_pt_bd}

Let $x \in \Sigma$ be the center of a ball $B_{\eps^k}(x)$ which is removed from $\Sigma$ in the construction of $\Sigma_{\eps,\kappa}$.
In the case of attaching a cylinder we have $x \in \{x_0,x_1\}$, in the case of attaching a cross cap we have $x=x_0$.

In order to understand the spectrum of $\Sigma_{\eps,\kappa}$, we need some bounds for eigenfunctions with bounded energy on $\partial B_{\eps^k}(x)$.
For ease of notation, we assume that the ball $B_1(x) \subset \Sigma$
can be endowed with conformal coordinates.
In the case of attaching a cylinder, we also assume that the two balls $B_1(x_0)$ and $B_1(x_1)$ are disjoint. 

\begin{lemma} \label{lem_pt_bd}
There are constants $C=C(\Sigma,\Lambda)$ and $\eps_1>0$ such that the following holds.
Let $u_{\eps,\kappa}$ be a $L^2$-normalized eigenfunction on $\Sigma_{\eps,\kappa}$ with eigenvalue $\lambda_{\eps,\kappa} \leq \Lambda$ 
where $\eps \leq \eps_1$.
If we use Euclidean polar coordinates $(r,\theta)$ centered at $x$ we have the uniform pointwise bounds
\begin{equation} \label{max_eigen_sup_bd}
	|u_{\eps,\kappa}|(r,\theta) \leq C \log \left(\frac{1}{r} \right),
\end{equation}
for $ \eps^k \leq r \leq 1/2$
and
\begin{equation} \label{max_eigen_lipschitz_bd}
	|\nabla u_{\eps,\kappa}| (r, \theta) \leq \frac{C}{r}
\end{equation}
for $2 \eps^k \leq r \leq 1/2.$
\end{lemma}

Note that \cref{lem_pt_bd} is rleated\footnote{In fact it is not very hard to improve \eqref{max_eigen_sup_bd} to $\log^{1/2}(1/\eps)$, but we do not need this.} 
to the integral bound
\begin{equation} \label{eq_int_bd_sobolev}
\int_{\partial B_{\eps^k}(x)} |\varphi|^2 d \mathcal{H}^1
\leq 
C \eps^k \log(1/\eps^k) \|\varphi \|_{W^{1,2}(B_1(x) \setminus B_{\eps^k} (x))}^2
\end{equation}
that holds for any $\varphi \in W^{1,2}(\Sigma_{\eps,\kappa})$ and which can be proved by a straightforward computation in polar coordinates.

\begin{proof}
Recall that we can identify a conformally flat neighborhood of $x$ with $B_1=B(0,1)\subset \IR^2,$
such that $x=0.$
First, observe that up to radius $2\eps^k$ \eqref{max_eigen_sup_bd} is a direct consequence of \eqref{max_eigen_lipschitz_bd}.
In fact, by standard elliptic estimates \cite[Chapter 5.1]{taylor_1},  the functions $u_{\eps,\kappa}$ are 
uniformly bounded in $C^0(B_1 \setminus B_{1/2})$.
Given this, we can integrate the bound \eqref{max_eigen_lipschitz_bd} from $\partial B_{1/2}$ to $\partial B_r$ and find \eqref{max_eigen_sup_bd}.

The bound \eqref{max_eigen_lipschitz_bd} follows from standard elliptic estimates after rescaling the scale $r$ to a fixed scale.
More precisely, we consider the rescaled functions $w_r(z):=u_{\eps,\kappa}(r z).$
On $B_1\setminus B_{\eps^k}$ the metric of $\Sigma$ is uniformly bounded from above and below by the Euclidean metric.
Hence we can perform all computations in the Euclidean metric.
We have
\begin{equation} \label{max_eigen_L2grdbd}
\int_{B_3\setminus B_{1/2}} |\nabla w_r|^2 = \int_{B_{3r}\setminus B_{r/2}} |\nabla u_{\eps,\kappa}|^2,
\end{equation}
since the Dirichlet energy is conformally invariant in dimension two.

Since the Laplace operator is conformally covariant in dimension two, $w_r$ solves the equation
\begin{equation} \label{max_eigen_eqn_sc}
\Delta_e w_r=r^2 f_r \lambda_{\eps,\kappa} w_r,
\end{equation}
with $f_r(z)=f(rz)$ a smooth function and $\Delta_e$ the Euclidean Laplacian.
Since $f \in C^\infty,$ we have uniform $C^\infty$-bounds on $f_r$ for $r \leq 1.$
Taking derivatives, we find that
\begin{equation} \label{max_eigen_eqn_der}
\Delta_e \nabla w_r = r^2 \lambda_{\eps,\kappa} \nabla (f_r w_r),
\end{equation}
where also the gradient is taken with respect to the Euclidean metric.
Since $\lambda_{\eps,\kappa} \leq \Lambda$, the scaling invariance of the Dirichlet energy implies that
\begin{equation*}
\begin{split}
\lambda_{\eps,\kappa}^2 \int_{B_3 \setminus B_{1/2}} |\nabla(f_rw_r)|^2
&=\lambda_{\eps,\kappa}^2 \int_{B_{3r} \setminus B_{r/2}} |\nabla(f u_{\eps,\kappa})|^2
\\
& \leq 
2 \lambda_{\eps,\kappa}^2 \int_{B_{3r} \setminus B_{r/2}} f^2 |\nabla u_{\eps,\kappa}|^2 + u_{\eps,\kappa}^2 |\nabla f|^2
\\
& \leq
C
\end{split}
\end{equation*}
by assumption.
In particular, the right hand side of \eqref{max_eigen_eqn_der} is bounded by 
$Cr^2$ in $L^2(B_3\setminus B_{1/2}).$
Therefore, by elliptic estimates \cite[Chapter 5.1]{taylor_1} we have
\begin{equation*}
\sup_{ \{1\leq s \leq 2\} }|\nabla w_r|(s, \theta) \leq C r^2 + C |\nabla w_r|_{L^2(B_3 \setminus B_{1/2})} \leq C,
\end{equation*}
which scales to
\begin{equation*}
\sup_{\{r\leq s \leq 2r\}}|\nabla u_{\eps,\kappa}| (s,\theta) \leq \frac{C}{r},
\end{equation*}
with $C$ independent of $r.$
This proves the estimate for $r \geq 2 \eps^k.$

To get the estimate \eqref{max_eigen_sup_bd} for the remaining radii we invoke the De Giorgi--Nash--Moser estimate.
In order to do so it is useful to observe that for $\eps$ sufficiently small any component of the boundary of $C_{\eps,\kappa}$ and hence also of $M_{\eps,\kappa}$
has an open neighbourhood 
$$
U_{\eps,\kappa} \subset C_{\eps,\kappa}
$$
which is conformal to $\Sph^1(1) \times [0,4)$ with metric given by
$$
h_{\eps,\kappa} (d\theta^2 + dy^2),
$$
for some function $h_{\eps,\kappa}>0$.
At this point we fix $\eps_1>0$ such that this holds for any $\kappa \in [\kappa_0,\kappa_1]$ if $\eps \leq \eps_1$.

We then consider
$$
V_{\eps,\kappa} = U_{\eps,\kappa} \cup ((B_{4\eps^k}(x) \setminus B_{\eps^k}(x)) \cap \Sigma)
$$ 
which is an open neighbourhood of $\partial B_{\eps^k}(x)$ in $\Sigma_{\eps,\kappa}$.
On $V_{\eps,\kappa}$ we rescale the original metric $g_{\eps,\kappa}$ by the singular conformal factor 
\begin{equation*}
f_{\eps,\kappa}
= 
\begin{cases}
\eps^{-2k} & \ \text{in} \ (B_{4\eps^k}(x) \setminus B_{\eps^k}(x)) \cap \Sigma,
\\
h_{\eps,\kappa}^{-1} & \ \text{in} \ V_{\eps,\kappa} \setminus \Sigma,
\end{cases}
\end{equation*}
i.e. we consider the metric $l_{\eps,\kappa}=f_{\eps,\kappa} g_{\eps,\kappa}$.

Consider the function 
$$
w_{\eps,\kappa} = u_{\eps,\kappa} - (u_{\eps,\kappa})_{V_{\eps,\kappa}},
$$
where $(u)_{V_{\eps,\kappa}}$ denotes the mean value of $u$ on $V_{\eps,\kappa}$ with respect to the rescaled metric $l_{\eps,\kappa}$.
By the conformal invariance of the Dirichlet energy, we find that $w_{\eps,\kappa}$ has gradient bounded in $L^2$ with respect to the rescaled metric,
\begin{equation} \label{dnm_1}
\int_{V_{\eps,\kappa}} |\nabla w_{\eps,\kappa}|^2 dA_{l_{\eps,\kappa}}\leq C,
\end{equation}
where the constant $C$ depends only on $\Lambda$.
It is easy to see that the rescaled metric $l_{\eps,\kappa}$ on $V_{\eps,\kappa}$ is uniformly bounded from above and below almost everywhere by a fixed metric.
In fact, on $V_{\eps,\kappa} \setminus \Sigma$ the metric $l_{\eps,\kappa}$ is the metric of a fixed flat cylinder, and on $V_{\eps,\kappa} \cap \Sigma$ the metric $l_{\eps,\kappa}$ is close to the standard flat metric on (a subset of) the unit disk.
Therefore, there is a constant $C$ independent of $\eps$ and $\kappa$ such that
\begin{equation} \label{dnm_2}
\int_{V_{\eps,\kappa}} |w_{\eps,\kappa}|^2 dA_{l_{\eps,\kappa}} \leq C \int_{V_{\eps,\kappa}} |\nabla w_{\eps,\kappa}|^2 dA_{l_{\eps,\kappa}}.
\end{equation}
Now observe that $w_{\eps,\kappa}$ is a weak solution to the equation 
\begin{equation} \label{conf_diff_eqn}
\Delta_{l_{\eps,\kappa}} w_{\eps,\kappa} = \frac{1}{f_{\eps,\kappa}} \Delta_{g_{\eps,\kappa}} u_{\eps,\kappa} = \frac{1}{f_{\eps,\kappa}} \lambda_{\eps,\kappa} u_{\eps,\kappa},
\end{equation}
thanks to the conformal covariance of the Laplacian in dimension two.
Finally, note that the right hand side of \eqref{conf_diff_eqn} is bounded in $L^2(V_{\eps,\kappa},dA_{l_{\eps,\kappa}})$.
Thanks to this, \eqref{dnm_1}, \eqref{dnm_2}, and \eqref{conf_diff_eqn} we can apply the inhomogeneous De Giorgi--Nash--Moser estimates (see e.g.\ \cite[Chapter 14.9]{taylor_3}) to obtain
\begin{equation*}
\sup_{p \in \partial B_{\eps^k}, q \in \partial B_{2\eps^k} } |u_{\eps,\kappa}(p) - u_{\eps,\kappa}(q)|
=
\sup_{p \in \partial B_{\eps^k}, q \in \partial B_{2\eps^k} } |w_{\eps,\kappa}(p)-w_{\eps,\kappa}(q)| \leq C.
\end{equation*}
When combined with the bound up to radius $2\eps^k$ proved above, this implies
 \eqref{max_eigen_sup_bd}.
\end{proof}

\section{A lower bound for $\mu_1(\Sigma \setminus B_\eps)$ and the multiplicity of the first eigenvalue} \label{sec_neu}


In this section we provide a lower bound for the first Neumann eigenvalue of $\Sigma \setminus B_{\eps}(x)$.
We then apply this to show that $\lambda_1(\Sigma_{\eps,\kappa})$ has to be simple if it is not very close to $\lambda_1(\Sigma)$
already.

\subsection{A lower bound for $\mu_1(\Sigma \setminus B_\eps)$}

It is convenient to state the main result of this section more generally for compact surfaces with boundary.
Let $\Sigma$ be a compact surface with smooth, possibly empty, boundary.
Assume that $\Sigma$ is smooth away from finitely many conical singularities and that $x \in \Sigma \setminus \partial \Sigma$
such that the metric is smooth near $x$.
We denote by $B_{\eps}(x)$ a ball of radius $\eps$ with respect 
to the Euclidean distance in normalized conformal coordinates.

The goal of this section is to prove the following result, that might be of some independent interest.
We chose to give the proof at this point since it uses some of the arguments from the preceding section.
It will only be in the next section that our interest in this bound will become clear.

\begin{theorem} \label{thm_lower_neumann}
There is a constant $C$ depending only on $B_1(x) \subset \Sigma$, such that
the first Neumann eigenvalue of $\Sigma \setminus B_{\eps}(x)$ satisfies
$$
\mu_1(\Sigma \setminus B_\eps(x)) 
\geq \mu_1(\Sigma) - C \eps^2.
$$
\end{theorem}

Note that $\mu_1(\Sigma) = \lambda_1(\Sigma)$ in case $\Sigma$ is a closed surfaces.
From standard elliptic estimates we get the following bounds for a Neumann eigenfunction and its gradient.

\begin{lemma} \label{neu_ell_est}
Let $u_\eps$ be a normalized $\mu_1(\Sigma \setminus B_\eps(x))$-eigenfunction.
If we use Euclidean polar coordinates $(r,\theta)$ centered at $x,$ we have the uniform pointwise bounds
\begin{equation} \label{sup_bd}
	|u_{\eps}|(r,\theta) \leq C \log \left(\frac{1}{r} \right),
\end{equation}
and
\begin{equation} \label{lipschitz_bd}
	|\nabla u_\eps| (r, \theta) \leq \frac{C}{r}
\end{equation}
for any $\eps \leq r \leq 1/2.$
\end{lemma}

\begin{proof}
For radii $r \geq 2 \eps$ this follows from the proof of \cref{lem_pt_bd}.
For the remaining radii, we use the same argument but apply elliptic boundary estimates
\cite[Chapter 5.7]{taylor_1}.
\end{proof}

Besides the pointwise bound  on $u_\eps$ from the previous lemma, the proof of \cref{thm_lower_neumann} relies on a good estimate on the $L^2$-norm of the tangential gradient of  $ u_\eps$ along $\partial B_\eps(x).$
Below, we denote by $\partial_T u_{\eps}$ the gradient of $\left. u_{\eps} \right|_{\partial B_{\eps}}.$
We assume once more for ease of notation that $B_1(x)$ can be endowed with conformal coordinates.

\begin{lemma} \label{tangential_bd}
For a normalized $\mu_1(\Sigma \setminus B_\eps(x))$-eigenfunction $u_\eps$, we have that
\begin{equation} \label{goal}
\int_{\partial B_\eps(x)} |\partial_T  u_\eps|^2 d \mathcal{H}^1 \leq C \eps.
\end{equation}
\end{lemma}

\begin{proof}

We denote by $\tilde u_\eps$ the function obtained by extending $u_\eps$ harmonically to $B_\eps(x),$
where $u_\eps$ denotes a normalized $\mu_1(\Sigma \setminus B_\eps(x))$-eigenfunction.
By a scaling argument, $\tilde u_\eps$ is uniformly bounded in $W^{1,2}(\Sigma)$ in terms of the $W^{1,2}(\Sigma \setminus B_\eps(x))$-norm of $u_\eps$ \cite[p.\,40]{rt}.

Let $w_\eps$ be the unique weak solution to
\begin{equation*}
\begin{cases}
\Delta w_\eps = \mu_1(\Sigma \setminus B_\eps(x))\tilde u_\eps &   \text{in} \ B_1(x)
\\
\hspace{0.3cm} w_\eps = 0  & \text{on} \ \partial B_1(x).
\end{cases}
\end{equation*}
By elliptic estimates, $w_\eps$ is bounded in $W^{3,2}(B_{1/2}(x)),$
which embeds into $C^{1,\alpha}(B_{1/2}(x))$ for any $\alpha<1.$
We can then write
\begin{equation*}
u_\eps=w_\eps  + v_\eps,
\end{equation*}
with $v_\eps \in W^{1,2}(B_1(x)\setminus B_\eps(x))$ a harmonic function. 

Note that the bound \eqref{goal} clearly holds for $w_\eps,$ so it suffices to consider $v_\eps.$
If we denote by $\nu$ the inward pointing normal of $B_\eps(x),$ we have
\begin{equation} \label{infty_bd_norm_der}
|\partial_\nu v_\eps| =| \partial_\nu u_\eps -\partial_\nu w_\eps|=|\partial_\nu w_\eps|
\leq C
\end{equation}
along $\partial B_\eps(x),$ since $w_\eps$ is bounded in $C^{1,\alpha}(B_{1/2}).$ 
Since the Laplace operator is conformally covariant in dimension two, $v_\eps$ is also harmonic with respect to the Euclidean metric.
Therefore, it follows from separation of variables, that we can expand $v_\eps$ in Fourier modes
\begin{equation*}
v=a + b \log(r)+\sum_{n \in \IZ^*} (c_n r^n + \overline{c_{-n}} r^{-n})e^{in\theta}.
\end{equation*}
Here and in the following we suppress the index $\eps.$
Note that we have 
\begin{equation} \label{sum_bd_1}
2 \pi \sum_{n \in \IZ^*} \int_{\eps}^1 |c_n|^2 r^{2n+1} + 2 \real(c_nc_{-n}) r + |c_{-n}|^2 r^{-2n+1} \, dr
\leq \int_{B_1 \setminus B_\eps} |v|^2 
\leq C.
\end{equation}
For $\eps \leq 1/2$ and $n \geq 2$ we can use Young's inequality to find
\begin{equation} \label{sum_bd_2}
\begin{split}
 |c_n|^2 & \int_\eps^1 r^{2n+1} dr  + 2 \real(c_nc_{-n})  \int_\eps^1 r dr + |c_{-n}|^2 \int_\eps^1 r^{-2n+1} dr
\\
& = \frac{|c_n|^2}{2n+2}(1-\eps^{2n+2}) +2 \real(c_nc_{-n}) (1-\eps^2) + \frac{ |c_{-n}|^2}{2n-2} (\eps^{-2n+2}-1)
\\
& \geq \frac{|c_n|^2}{2n+2}(1-\eps^{2n+2}-(n+1)\delta_n) + \frac{ |c_{-n}|^2}{2n-2}\left(\eps^{-2n+2}-1 - \frac{n-1}{\delta_n} \right)
\\
& \geq \frac{|c_n|^2}{8(n+1)} + \frac{ |c_{-n}|^2}{2n-2}(\eps^{-2n+2}-2)
\\
& \geq \frac{|c_n|^2}{8(n+1)},
\end{split} 
\end{equation}
with  $\delta_n=1/(2(n+1))$.
A similar computation gives the same bound for $n=1$.
Of course, the same computation applies to negative $n$.
Therefore, we find from \eqref{sum_bd_1} and \eqref{sum_bd_2}, that
 \begin{equation*}
h_1=\sum_{n>0} (c_n e^{i n \theta } + \overline{c_{-n}} e^{-in\theta}) r^n
\end{equation*}
extends to a harmonic function on all of $B_1,$ which is bounded in $L^2(B_1),$ whence in $C^\infty(B_{1/2}).$\\
Therefore, we are left with bounding the tangential derivative of the harmonic function
\begin{equation*}
h_2=v-h_1-a=b \log(r)+\sum_{n>0} \left( c_{-n} e^{-i n \theta} + \overline{c_{-n}}e^{in\theta} \right) r^{-n}.
\end{equation*}
We now use that the quantity
\begin{equation*}
\rho \int_{\partial B_\rho} \left( (\partial_T h_2)^2 - (\partial_r h_2)^2\right) d \mathcal{H}^1
\end{equation*}
is independent of $\rho$, what can be verified by a straightforward computation.
For $\rho \to \infty$ the term $\rho \int_{\partial B_\rho}(\partial_T h_2)^2 d \mathcal{H}^1$ vanishes, since the integrand decays at least like $\rho^{-3}.$
For the other term, note that $\partial_r \log(r)$ and $\partial_r (h_2 - b \log(r))$
are orthogonal in $L^2(\partial B_\rho).$
Therefore, we have 
\begin{equation*}
\begin{split}
\int_{\partial B_\rho} (\partial_r h_2)^2 d\mathcal{H}^1
&= b^2 \int_{\partial B_\rho} (\partial_r \log(r))^2 d\mathcal{H}^1 
+\int_{\partial B_\rho} (\partial_r(h_2-b\log(r)))^2 d\mathcal{H}^1
\\
& =  \frac{2 \pi b^2}{\rho} + O(\rho^{-3})
\end{split}
\end{equation*}
since the integrand of the second summand decays at least like $\rho^{-4}$ as $\rho \to \infty.$
In conclusion,
\begin{equation} \label{harm_const}
\rho \int_{\partial B_\rho} ((\partial_T h_2)^2 - (\partial_r h_2)^2) d\mathcal{H}^1= - 2 \pi b^2
\end{equation}
for any $\rho \geq \eps.$
This yields
\begin{equation*} 
\int_{\partial B_\eps(x)} (\partial_T h_2)^2 d \mathcal{H}^1 
= \int_{\partial B_\eps(x)} (\partial_r h_2)^2 d \mathcal{H}^1 - \frac{2 \pi b^2}{\eps}
\leq C \eps.
\end{equation*}
thanks to \eqref{infty_bd_norm_der}.
\end{proof}

We can now give the proof of \cref{thm_lower_neumann}

\begin{proof}[Proof of \cref{thm_lower_neumann}]
Let $u_\eps$ be a normalized $\mu_1(\Sigma \setminus B_\eps)$-eigenfunction.
Once again, $\tilde u_\eps$ denotes the $W^{1,2}$-function on $\Sigma$ obtained by extending $u_\eps$ harmonically to $B_\eps$.
We want to use 
$$
v_\eps = \tilde u_\eps - \int_\Sigma \tilde u_\eps
$$ 
as a test function for $\mu_1(\Sigma)$.

We get from \cref{neu_ell_est} that
$$
\left| \int_\Sigma \tilde u_\eps \right| \leq  \int_{B_\eps} |\tilde u_\eps| \leq C|\log(\eps)| \eps^2,
$$
which in turn implies
\begin{equation} \label{eq_mv_bd}
\begin{split}
\int_\Sigma |v_\eps|^2 
&\geq 
\int_{\Sigma \setminus B_\eps} |u_\eps|^2 - 2 \left( \int_\Sigma \tilde u_\eps \right)^2
\\
&\geq
1 - c \eps^4 \log(1/\eps)^2.
\end{split}
\end{equation}
In order to estimate the energy of $\tilde u_\eps$ on $B_\eps$ we may assume for a moment that 
\begin{equation} \label{wlog}
\int_{\partial B_\eps} u_\eps d \mathcal{H}^1=0,
\end{equation}
since subtracting a constant only results in subtracting a constant from $\tilde u_\eps$ in $B_\eps$,
which does not change the energy of $\tilde u_\eps$ in $B_\eps.$
We use $$\hat u_\eps (r, \theta)= \frac{r}{\eps} u_\eps(\eps,\theta)$$ as a competitor.
In order to estimate its energy, we use that \eqref{wlog}, the Poincar{\'e} inequality, and \cref{tangential_bd} imply
\begin{equation} \label{poincare}
\int_{\partial B_\eps} |u_\eps|^2 d \mathcal{H}^1 
\leq C \eps^2 \int_{\partial B_\eps} (\partial_T u_\eps)^2 d \mathcal{H}^1 
\leq C \eps^3.
\end{equation}
Therefore, we get
\begin{equation*} 
\begin{split}
\int_{B_\eps} |\nabla \hat u_{\eps}|^2 
&\leq \frac{C}{\eps^2} \int_0^{2\pi} \int_0^\eps  \left(|u_\eps|^2(\eps,\theta) + (\partial_\theta u_\eps)^2(\eps,\theta) \right)r dr d\theta
\\
& \leq \frac{C}{\eps} \int_{\partial B_\eps} |u_\eps|^2d \mathcal{H}^1 + C \eps \int_{\partial B_\eps} (\partial_T u_\eps)^2d \mathcal{H}^1
\\
& \leq C \eps^2,
\end{split}
\end{equation*}
where we have used \eqref{poincare} and \cref{tangential_bd}.
This implies that
\begin{equation} \label{eq_grd_est}
\int_\Sigma |\nabla \tilde u _\eps|^2 \leq \mu_1(\Sigma \setminus B_\eps) + C\eps^2,
\end{equation}
since $u_\eps$ is assumed to be normalized.
Combining \eqref{eq_mv_bd} and \eqref{eq_grd_est}, we obtain
\begin{equation*}
\mu_1(\Sigma)
\leq
\frac{\int_\Sigma |\nabla  v_\eps|^2}{ \int_\Sigma |v_\eps|^2}
\leq
\frac{\mu_1(\Sigma \setminus B_\eps) + C\eps^2}{1 - c \eps^4\log^2(1/\eps)}
\leq  \mu_1(\Sigma \setminus B_\eps) + C \eps^2.
\qedhere
\end{equation*}
\end{proof}

\subsection{The multiplicity of the first eigenvalue} \label{sec_mult_first}

As an application of the bound from \cref{thm_lower_neumann} we can prove that that the first eigenvalue of $\Sigma_{\eps,\kappa}$ has to be simple
if the gap $\lambda_1(\Sigma) - \lambda_1(\Sigma_{\eps,\kappa})$ is not too small.

We start with the following lemma.
\begin{lemma} \label{first_upper_bd}
There is $\eps_{2}=\eps_{2}(\Sigma,k,\kappa_0)>0$ such that 
$$
\lambda_1(\Sigma_{\eps,\kappa}) \leq \lambda_0(C_{\eps,\kappa})
$$
if $\eps \leq \eps_{2}$.
\end{lemma}
\begin{proof}
We give the argument only for attaching a cylinder, the case of a cross cap is analogous.
Let $\eta_\eps \colon \Sigma \setminus (B_{\eps^k}(x_0) \cup B_{\eps^k}(x_1)) \to [0,1]$ be the log-cut-off function that is given by
\begin{equation} \label{eq_log_cut_off}
\eta_\eps(x) = 
\begin{cases}
1 & \ \text{if} \ x \in B_1(x_0) \setminus B_{\eps^{k/2}}(x_0)
\\
1-\frac{\log(|x|/\eps^{k/2})}{\log(\eps^{k/2})} & \ \text{if} \ x \in B_{\eps^{k/2}}(x_0) \setminus B_{\eps^k}(x_0),
\end{cases}
\end{equation}
near $x_0$ and similarly near $x_1$.
We extend this in the obvious locally constant way to all of $\Sigma_{\eps,\kappa}$.

We use the two dimensional space of test functions spanned by $\eta_\eps$  
and a $\lambda_0(C_{\eps,\kappa})$-eigenfunction $\phi_{\eps,\kappa,0}$ (extended by $0$ to all of $\Sigma_{\eps,\kappa}$).
Since these two have disjoint support any linear combination $w$ of them has 
\begin{equation*}
\begin{split}
\int_{\Sigma_{\eps,\kappa}} |\nabla w|^2 
&\leq 
\max(C/\log(1/\eps),\lambda_0(C_{\eps,\kappa})) \int_{\Sigma_{\eps,\kappa}} |w|^2
\\
&\leq
\lambda_0(C_{\eps,\kappa}) \int_{\Sigma_{\eps,\kappa}} |w|^2
\end{split}
\end{equation*}
for $\eps$ sufficiently small, depending on $\Sigma,k,\kappa_0$.
\end{proof}

\begin{prop} \label{prop_first_simple}
There is $\eps_{3}=\eps_{3}(\Sigma,k,\kappa_0)>0$ with the following property.
For $\eps \leq \eps_{3}$ and  $\kappa \in [\kappa_0,\kappa_1]$ such that $\lambda_1(\Sigma_{\eps,\kappa}) \leq \lambda_1(\Sigma) - \eps^{k}$
the first eigenvalue $\lambda_1(\Sigma_{\eps,\kappa})$ is simple.
\end{prop}

\begin{proof}
We argue by contradiction and assume that $\lambda_1(\Sigma_{\eps,\kappa})$ has multiplicity at least two.
In this case we may choose a first eigenfunction $u_{\eps,\kappa}$ such that
\begin{equation} \label{eq_orthog}
\int_{C_{\eps,\kappa}} u_{\eps,\kappa} =-\int_{\Sigma \setminus B_{\eps^k}} u_{\eps,\kappa} = 0,
\end{equation}
where for simplicity we write $B_{\eps^k} = B_{\eps^k}(x_0) \cup B_{\eps^k}(x_1)$.
Therefore, we find that
\begin{equation*}
\begin{split}
\int_{\Sigma \setminus B_{\eps^k}} |\nabla u_{\eps,\kappa}|^2 
&\leq
\lambda_1(\Sigma_{\eps,\kappa}) \int_{\Sigma \setminus B_{\eps^k}} |u_{\eps,\kappa}|^2
+
\lambda_1(\Sigma_{\eps,\kappa}) \int_{C_{\eps,\kappa}} |u_{\eps,\kappa}|^2
-
\int_{C_{\eps,\kappa}} |\nabla u_{\eps,\kappa}|^2
\\
& \leq
\lambda_1(\Sigma_{\eps,\kappa}) \int_{\Sigma \setminus B_{\eps^k}} |u_{\eps,\kappa}|^2
+
\frac{\lambda_1(\Sigma_{\eps,\kappa})} {\mu_1(C_{\eps,\kappa})} \int_{C_{\eps,\kappa}} |\nabla u_{\eps,\kappa}|^2
-
\int_{C_{\eps,\kappa}} |\nabla u_{\eps,\kappa}|^2
\\
& \leq
\lambda_1(\Sigma_{\eps,\kappa}) \int_{\Sigma \setminus B_{\eps^k}} |u_{\eps,\kappa}|^2,
\end{split}
\end{equation*}
where the last step used \cref{first_upper_bd} and $\lambda_0(C_{\eps,\kappa}) \leq \mu_1(C_{\eps,\kappa})$ thanks to \cref{spec_comp_dir_neu}.

On the other hand, because of \eqref{eq_orthog}, we also have that
$$
\mu_1(\Sigma \setminus B_{\eps^k}) \int_{\Sigma \setminus B_{\eps^k}} | u_{\eps,\kappa}|^2 
\leq
\int_{\Sigma \setminus B_{\eps^k}} |\nabla u_{\eps,\kappa}|^2.
$$
Therefore, if we can show that $u_{\eps,\kappa}$ does not vanish almost everywhere on $\Sigma \setminus B_{\eps^k}$ it follows from the two estimates above and \cref{thm_lower_neumann} that
$$
\lambda_1(\Sigma) - C \eps^{2k} \leq \mu_1(\Sigma \setminus B_{\eps^k}) \leq \lambda_1(\Sigma_{\eps,\kappa})
$$
contradicting our assumptions for $\eps$ sufficiently small.

Suppose that $u_{\eps,\kappa}$ vanishes almost everywhere on $\Sigma \setminus B_{\eps,\kappa}$. 
Since $u_{\eps,\kappa} \in C^0(\Sigma_{\eps,\kappa})$ (this follows from the proof of \cref{lem_pt_bd}), this implies that it is a solution to $(\Delta - \lambda_1(\Sigma_{\eps,\kappa}))u_{\eps,\kappa}=0$ that vanishes along $\partial C_{\eps,\kappa}$.
Since $\lambda_1(\Sigma_{\eps,\kappa}) \leq \lambda_0(C_{\eps,\kappa})$ and $u_{\eps,\kappa}$ is non-constant this implies $\lambda_1(\Sigma_{\eps,\kappa}) = \lambda_0(C_{\eps,\kappa})$.
But then $u_{\eps,\kappa}$ has a sign on $C_{\eps,\kappa}$, contradicting \eqref{eq_orthog}
\end{proof}

The iteration argument in \cref{sec_iter} eventually leading to \cref{thm_main_technical} starts from the very same computation if $\lambda_1(\Sigma_{\eps,\kappa})$
and $\lambda_2(\Sigma_{\eps,\kappa})$ are not equal but very close to each other.
Because we will only be able to work with a linear combination of the first two eigenfunctions there will be error terms arising that are not present in the computation above.
Controlling these error terms is one of our main technical issues.

\section{Limits of eigenvalues} \label{sec_eig_lim}

In this section we prove a first result on the behavior of the eigenvalues of $\Sigma_{\eps,\kappa}$ as $\eps$ approaches $0$.
The proof of our main result will make use of the much more precise, and much more technical, estimates in \cref{sec_conc}, from which one could very easily obtain the lower bound on the eigenvalues proved below.
However, we think that this less technical section could be very helpful to the reader to get a feeling for some features of the construction and a few of the techniques invoked in \cref{sec_conc}.
  
\begin{theorem} \label{thm_eig_lim}
Let $0<\kappa_0<\kappa<\kappa_1$.
We have that
$$
\lim_{\eps \to 0}\lambda_l(\Sigma_{\eps,\kappa}) =  \min \left(\lambda_l(\Sigma), \frac{\kappa^2}{4} \right)
$$ 
for any $l \in \IN$ locally uniformly in $\kappa$.
\end{theorem}

\begin{proof}
We only give the proof for the case of attaching a cylinder. 
The case of a cross cap is completely analogous.

\smallskip

\textsc{Step 1:} 
\textit{Upper bound on eigenvalues.}
\smallskip

For simplicity, we write $B_{\eps^k}=B_{\eps^k}(x_0) \cup B_{\eps^k}(x_1)$.
Let $\eta_\eps \colon \Sigma \setminus B_{\eps^k} \to [0,1]$ be the log-cut-off function from \eqref{eq_log_cut_off} that cuts off
near $x_0$ and $x_1$.
For $\phi \colon \Sigma \to \IR$ an $L^2(\Sigma)$-normalized eigenfunction with eigenvalue $\lambda$,
the function $\eta_\eps \phi$ extends to a smooth function on $\Sigma_{\eps,\kappa}$.
We have that
\begin{equation*}
\begin{split}
\int_{\Sigma_{\eps,\kappa}} |\nabla (\eta_\eps \phi) |^2
&\leq 
\int_{\Sigma_{\eps,\kappa}} \eta_{\eps}^2 |\nabla \phi|^2 
+2 \int_{\Sigma_{\eps,\kappa}} \phi \eta_\eps \nabla \phi \nabla \eta_\eps
+ \int_{\Sigma_{\eps,\kappa}} \phi^2 |\nabla \eta_\eps|^2 
\\
&\leq
\int_{\Sigma \setminus B_{\eps^k}} |\nabla \phi|^2 
+ C \left( \int_{\Sigma \setminus B_{\eps^k}} |\nabla \eta_\eps|^2  \right)^{1/2}
+C \int_{\Sigma \setminus B_{\eps^k}} |\nabla \eta_\eps|^2 
\\
& \leq 
\lambda + \frac{C}{\log^{1/2}(1/\eps)}
\end{split}
\end{equation*}
since $\phi$ and $\nabla \phi$ are bounded on $\Sigma$.

Moreover, if $(\phi_1,\phi_2)$ as above are orthonormal we have that
$$
\left|\int_{\Sigma_{\eps,\kappa}} (\eta_\eps \phi_i) (\eta_\eps \phi_j) - \delta_{ij} \right|
\leq \int_{B_{\eps^{k/2}}} |\phi_i \phi_j|
\leq C \eps^{k}
$$
and similarly
$$
\left| \int_{\Sigma_{\eps,\kappa}} \nabla (\eta_\eps \phi_i) \nabla (\eta_\eps \phi_j) - \delta_{ij} \right|
\leq
\frac{C}{\log^{1/2}(1/\eps)}
$$
for their gradients.
Using these two estimates the min-max principle easily implies that
$$
\limsup_{\eps \to 0} \lambda_l(\Sigma_{\eps,\kappa}) \leq \lambda_l(\Sigma)
$$
for any $l \in \IN$ uniformly in $\kappa \in [\kappa_0,\kappa_1]$.

Similarly, given a non-trivial Dirichlet eigenfunction $\psi_{\eps,\kappa}$ on $C_{\eps,\kappa}$ with eigenvalue $\lambda_{\eps,\kappa}$ note that $\psi_{\eps,\kappa}$ extends by $0$ to a Lipschitz function on all of $\Sigma_{\eps,\kappa}$ with
$$
\int_{\Sigma_{\eps,\kappa}} |\nabla \psi_{\eps,\kappa}|^2 = \lambda_{\eps,\kappa} \int_{\Sigma_{\eps,\kappa}} |\psi_{\eps,\kappa}|^2.
$$
Since there are infinitely many Dirichlet eigenvalues on $C_{\eps,\kappa}$ converging to $\kappa^2/4$, we get that
also
$$
\limsup_{\eps \to 0} \lambda_l(\Sigma_{\eps,\kappa}) \leq \frac{\kappa^2}{4}
$$
for any $k \in \IN$.

\smallskip

\textsc{Step 2:} 
\textit{Lower bound on eigenvalues.}
\smallskip

Given a normalized non-constant eigenfunction $u_{\eps,\kappa}$ on $\Sigma_{\eps,\kappa}$ with uniformly bounded eigenvalue $\lambda_{\eps,\kappa}$, let $v_{\eps,\kappa} \in W^{1,2}(\Sigma \setminus B_{\eps^k})$ denote the restriction of $u_{\eps,\kappa}$ to $\Sigma \setminus B_{\eps^k}$
and $w_{\eps,\kappa} \in W_0^{1,2}(C_{\eps,\kappa})$
the function given by
$$
w_{\eps,\kappa}=u_{\eps,\kappa}-h_{\eps,\kappa},
$$
where $h_{\eps,\kappa} \colon C_{\eps,\kappa} \to \IR$ is the harmonic extension to $C_{\eps,\kappa}$ of $\left. u_{\eps,\kappa} \right|_{\partial C_{\eps,\kappa}} \in W^{1/2,2}(\partial M_{\eps,\kappa})$.

It follows from \cref{lem_pt_bd} and the maximum principle that 
$$
|h_{\eps,\kappa}| \leq C \log(1/\eps^k).
$$
This in turn implies in combination with H{\"o}lder's inequality that
\begin{equation} \label{eq_conv_1}
\int_{C_{\eps,\kappa}} |w_{\eps,\kappa}|^2 \geq \int_{C_{\eps,\kappa}} |u_{\eps,\kappa}|^2 - C \eps^{1/2} \log(1/\eps).
\end{equation}
Note that we have that
\begin{equation*}
\int_{C_{\eps,\kappa}} \nabla (u_{\eps,\kappa}-h_{\eps,\kappa}) \cdot \nabla h_{\eps,\kappa} =0,
\end{equation*}
by integration by parts, since $h_{\eps,\kappa}$ is harmonic and $(u_{\eps,\kappa}-h_{\eps,\kappa})$ vanishes along $\partial C_{\eps,\kappa}$.
This easily implies
\begin{equation} \label{eq_conv_2}
\int_{C_{\eps,\kappa}} |\nabla (u_{\eps,\kappa} - h_{\eps,\kappa})|^2  
= 
\int_{C_{\eps,\kappa}} |\nabla u_{\eps,\kappa}|^2  - \int_{C_{\eps,\kappa}} |\nabla h_{\eps,\kappa}|^2
\leq \int_{C_{\eps,\kappa}} |\nabla u_{\eps,\kappa}|^2.
\end{equation}
We also have that
\begin{equation} \label{eq_conv_3}
\left| (u_{\eps,\kappa})_{\Sigma \setminus B_{\eps^k}} \right|
=
\left| \int_{\Sigma \setminus B_{\eps^k}} u_{\eps,\kappa} \right|
\leq
\int_{C_{\eps,\kappa}} |u_{\eps,\kappa}|
\leq
C \eps^{1/2}
\end{equation}
thanks to H{\"o}lder's inequality and since $(u_{\eps,\kappa})_{\Sigma_{\eps,\kappa}}=0$. 

Therefore, combining \eqref{eq_conv_1}, \eqref{eq_conv_2}, and \eqref{eq_conv_3}, we find that
\begin{equation*}
\begin{split}
\int_{\Sigma_{\eps,\kappa}} |\nabla u_{\eps,\kappa}|^2
&\geq
\int_{\Sigma \setminus B_{\eps^k}} |\nabla u_{\eps,\kappa}|^2 + \int_{C_{\eps,\kappa}} |\nabla w_{\eps,\kappa}|^2
\\
&\geq
\mu_1(\Sigma \setminus B_{\eps^k}) \int_{\Sigma \setminus B_{\eps^k}} |u_{\eps,\kappa} - (u_{\eps,\kappa})_{\Sigma \setminus B_{\eps^k}}|^2
+
\lambda_0(C_{\eps,\kappa}) \int_{C_{\eps,\kappa}} |w_{\eps,\kappa}|^2
\\
&\geq
\mu_1(\Sigma \setminus B_{\eps^k}) \int_{\Sigma \setminus B_{\eps^k}} |u_{\eps,\kappa}|^2 - C \eps 
+\lambda_0(C_{\eps,\kappa}) \int_{C_{\eps,\kappa}} |u_{\eps,\kappa}|^2 - C \eps^{1/2} \log(1/\eps)
\\
&\geq
\min\{\lambda_1(\Sigma),\kappa^2/4\} \int_{\Sigma_{\eps,\kappa}} |u_{\eps,\kappa}|^2 - C \eps^{1/2} \log(1/\eps),
 \end{split}
\end{equation*}
where we have used \cref{thm_lower_neumann} in the last step.
The last estimate clearly implies 
$$
\lambda_{\eps,\kappa} \geq \min\{\lambda_1(\Sigma),\kappa^2/4\} - C \eps^{1/2} \log(1/\eps),
$$
since $u_{\eps,\kappa}$ is normalized.
This proves the lower bound up to index $K-1$.
The general case follows by the exact same computation if we take $u_{\eps,\kappa}$ to be a linear combination of non-trivial eigenfunctions orthogonal to the first $l$ eigenfunctions on $\Sigma$.
\end{proof}

\section[Quasimodes]{Construction of quasimodes} \label{sec_quasimodes}

In this section we construct several quasimodes.
These are test functions on $\Sigma_{\eps,\kappa}$ that are approximate solutions to the eigenvalue equation.
Our main interest in quasimodes stems from the following observation of Ann{\'e}, which we state in a simplified version adapted to our specific setting.
We denote by $(u_{\eps,\kappa,l})_{l \in \IN}$ an orthonormal basis of $L^2(\Sigma_{\eps,\kappa})$ consisting of eigenfunctions.

\begin{lemma}[{\cite[Proposition $1$]{anne_3}}] \label{lem_anne_quasimod}
For any $\Lambda>0$, 
there is a uniform constant $C>0$ with the following property.
Let $f \in W^{1,2}(\Sigma_{\eps,\kappa})$ be a function with $1/2 \leq \|f\|_{L^2(\Sigma_{\eps,\kappa})}\leq 2$ such that
$$
\left| 
\int_{\Sigma_{\eps,\kappa}} \nabla f \nabla \varphi - \lambda \int_{\Sigma_{\eps,\kappa}}  f \varphi \,
\right|
\leq \delta \|\varphi\|_{W^{1,2}(\Sigma_{\eps,\kappa})}
$$
for some $\delta>0$ and any $\varphi \in W^{1,2}(\Sigma_{\eps,\kappa})$, where $\lambda \leq \Lambda$.
Let $0<s<1$ and write 
$$
g=\sum_{\{ l \ \colon \ |\lambda_l(\Sigma_{\eps,\kappa})-\lambda| > s \}} \langle f , u_{\eps,\kappa,l} \rangle_{L^2(\Sigma_{\eps,\kappa})} u_{\eps,\kappa,l}.
$$
Then
$$
\int_{\Sigma_{\eps,\kappa}} |g|^2 + \int_{\Sigma_{\eps,\kappa}} |\nabla g|^2 \leq C \frac{\delta^2}{s^2}.
$$
\end{lemma}
For sake of completeness, we have included a proof in \cref{sec_anne}.

\subsection{Quasimodes concentrated on $\Sigma$} \label{sec_quasimodes_1}

We start with the quasimodes concentrated on $\Sigma$, which we construct roughly by extending a $\lambda_1(\Sigma)$-eigenfunction $\phi$ to $M_{\eps,\kappa}$ by $\phi(x_0)$ and similarly in the case of cylinders near the long boundary component.

\subsubsection{The case of cross caps}

We first discuss the case of $M_{\eps,\kappa}$ in detail.
Let $\eta \colon [1,2] \to [0,1]$ be a function with 
$\eta(1)=0$
and
$\eta(2)=1$.
We then define a cut-off function $\eta_\eps \colon \Sigma_{\eps,\kappa} \to [0,1]$ by
\begin{equation*}
\eta_{\eps}=
\begin{cases}
1 & \ \text{in} \  \Sigma \setminus B_{2 \eps^k} (x_0)
\\
\eta(\eps^{-k} r) & \ \text{in} \ B_{2 \eps^k} (x_0)\setminus B_{\eps^k} (x_0)
\\
0 & \ \text{on} \ M_{\eps,\kappa},
\end{cases}
\end{equation*}
where we use (Euclidean) radial coordinates $(\theta, r)$ in $B_{2 \eps^k}(x_0)$.

For $\phi$ a $L^2(\Sigma)$-normalized $\lambda_1(\Sigma)$-eigenfunction, we define a new function 
\begin{equation*}
\tilde \phi_{\eps,\kappa}=
\begin{cases}
\phi & \ \text{in} \  \Sigma \setminus B_{2 \eps^k}
\\
\eta_\eps u + (1-\eta_\eps)\phi(x_0) & \ \text{in} \ B_{2 \eps^k} \setminus B_{\eps^k}
\\
\phi(x_0) & \ \text{on} \ M_{\eps,\kappa}.
\end{cases}
\end{equation*}

If $\phi(x_0)=0$ the function $\tilde \phi_\eps$ turns out to be a very good quasimode.
If $\phi(x_0)\neq 0$, we can identify the lowest order term on which $\phi_{\eps,\kappa}$ fails to solve the eigenvalue equation very explicitly.
Before we can actually prove this, we need to recall the following observation.

\begin{lemma} \label{mono_unif_sobolev}
Let $1<p<\infty$, then there is $C_p$ independent of $\eps$ and $k$, such that
\begin{equation*}
\| \varphi \|_{L^p(\Sigma \setminus B_{\eps^k}(x_0))}
\leq C_p \|\varphi\|_{W^{1,2}(\Sigma \setminus B_{\eps^k}(x_0))}.
\end{equation*}
\end{lemma}

\begin{proof}
This follows since the harmonic extension operator $W^{1,2}(\Sigma \setminus B_{\eps^k}(x_0)) \to W^{1,2}(\Sigma)$
is uniformly bounded.
See e.g.\ \cite[p.\ 40]{rt}, where this is proved by a scaling argument.
The conclusion then follows by combining this with the Sobolev embedding $W^{1,2}(\Sigma) \hookrightarrow L^p(\Sigma)$.
\end{proof}

\begin{lemma} \label{mono_quasimode_1}
We have for
the function $\tilde \phi_{\eps,\kappa}$ defined above and any $\varphi \in W^{1,2}(\Sigma_{\eps,\kappa})$, that
\begin{equation*}
\left| \int_{\Sigma_{\eps,\kappa}} \nabla \tilde \phi_{\eps,\kappa} \cdot \nabla \varphi - \lambda_1(\Sigma) \int_{\Sigma_{\eps,\kappa}} \tilde \phi_{\eps,\kappa} \varphi  + \lambda_1(\Sigma) \phi(x_0) \int_{M_{\eps,\kappa}} \varphi \, \right|
\leq
C \eps^{k/2} \| \varphi\|_{W^{1,2}(\Sigma_{\eps,\kappa})},
\end{equation*}
where the constant depends only on $\Sigma$.
\end{lemma}

\begin{proof}
We compute
\begin{equation} \label{mono_part_int}
\begin{split}
\int_{\Sigma_{\eps,\kappa}} \nabla \tilde \phi_{\eps,\kappa} \cdot \nabla \varphi
=&
\int_{\Sigma_{\eps,\kappa}} \nabla \phi \cdot \nabla (\eta_\eps \varphi) 
-
\int_{\Sigma_{\eps,\kappa}} \varphi \nabla \phi \cdot \nabla \eta_\eps
+
\int_{\Sigma_{\eps,\kappa}} (\phi-\phi(x_0)) \nabla \eta_\eps \cdot \nabla \varphi
\\
=&
\lambda_1(\Sigma)\int_{\Sigma_{\eps,\kappa}} \phi \eta_\eps \varphi
-
\int_{\Sigma_{\eps,\kappa}} \varphi \nabla \phi \cdot \nabla \eta_\eps
+
\int_{\Sigma_{\eps,\kappa}} (\phi-\phi(x_0)) \nabla \eta_\eps \cdot \nabla \varphi
\\
=&
\lambda_1(\Sigma )\int_{\Sigma_{\eps,\kappa}}\tilde \phi_{\eps,\kappa} \varphi
-
\lambda_1(\Sigma) \phi(x_0) \int_{M_{\eps,\kappa}} \varphi
-
\lambda_1(\Sigma) \phi_0(x_0) \int_{\Sigma \setminus B_{\eps^k}} (1-\eta_\eps) \varphi
\\
&-
\int_{\Sigma_{\eps,\kappa}} \varphi \nabla \phi \cdot \nabla \eta_\eps
+
\int_{\Sigma_{\eps,\kappa}} (\phi-\phi(x_0)) \nabla \eta_\eps \cdot \nabla \varphi,
\end{split}
\end{equation}
since $\eta_\eps \varphi \in W^{1,2}(\Sigma)$ and $\supp \nabla \eta_\eps \subset \Sigma \setminus B_{\eps^k}$.
Let us estimate the last three terms separately.
The first of these is small by H{\"o}lder's inequality,
\begin{equation}
\begin{split}
\left| \lambda_1(\Sigma) \phi_0(x_0) \int_{\Sigma \setminus B_{\eps^k}} (1-\eta_\eps) \varphi \right|
&\leq
C  \area{B_{2\eps^k}}^{1/2} \left( \int_{\Sigma \setminus B_{\eps^k}} |\varphi|^2 \right)^{1/2}
\\
&\leq
C \eps^k \|\varphi\|_{L^2}(\Sigma_{\eps,h}).
\end{split}
\end{equation}
For the second term, we proceed as follows:
Since $\phi$ is smooth, there is a constant $C$ such that $|\nabla \phi| \leq C$.
Therefore, we can invoke  H{\"o}lder's inequality, the scaling invariance of the Dirichlet energy, and \cref{mono_unif_sobolev} to find
\begin{equation} \label{mono_term_1}
\begin{split}
\left| \int_{\Sigma_{\eps,\kappa}} \varphi \nabla \phi \cdot \nabla \eta_\eps \, \right|
& \leq 
C \left( \int_{\Sigma_{\eps,\kappa}}  |\varphi|^p \right)^{1/p} 
\area(B_{2\eps^k})^{1/q}
\left( \int_{\Sigma_{\eps,\kappa}}  |\nabla \eta_\eps|^2 \right)^{1/2}
\\
& \leq C \eps^{2k/q} \| \varphi \|_{W^{1,2}(\Sigma_{\eps,\kappa})},
\end{split}
\end{equation}
where we used that it suffices to integrate over $\supp \nabla \eta_\eps \subset B_{2\eps^k}$ and $1/p+1/q=1/2$.

We now estimate the last term from \eqref{mono_part_int}.
Since $\phi$ is smooth, there is a constant $C$, such that
\begin{equation*}
|\phi - \phi(x_0)| \leq C \eps^k
\end{equation*}
in $B_{2\eps^k}$.
Since $\supp \nabla \eta_\eps \subset B_{2\eps^k}$, this implies
\begin{equation} \label{mono_term_2}
\begin{split}
\left| \int_{\Sigma_{\eps,\kappa}} (\phi - \phi(x_0))\nabla \eta_\eps \cdot \nabla \varphi \, \right|
&\leq 
C \eps^k \left(\int_{\Sigma_{\eps,\kappa}} |\nabla \eta_\eps|^2 \right)^{1/2}
\left(\int_{\Sigma_{\eps,\kappa}} |\nabla \varphi|^2 \right)^{1/2}
\\
& \leq C \eps^k \|\varphi\|_{W^{1,2}(\Sigma_{\eps,\kappa})}
\end{split}
\end{equation}
by H{\"o}lder's inequality and the scaling invariance of the Dirichlet energy.
If we specify to $p=q=4$ in \eqref{mono_term_1} and combine this with \eqref{mono_part_int} and \eqref{mono_term_2}, the assertion follows.
\end{proof}

Because of \cref{mono_quasimode_1}, all the eigenfunctions vanishing in $x_0$ will be mostly irrelevant in order to locate the first eigenvalue of $\Sigma_{\eps,\kappa}$.
It is also already clear that we expect the role of $\phi_0$ to be drastically different -- see also \cite{MS2}.

\subsubsection{The case of cylinders} \label{subsec_quasimod_cyl}

If we attach the cylinder $C_{\eps,\kappa}$ near the points $x_0$ and $x_1$ we can carry out a similar construction.
However, there is one important change in the construction near $x_1$.
In this case we denote by $\eta_\eps \colon \Sigma \to \IR$ the function given by cutting off near $x_0$ and $x_1$ using the cut-off function $\eta$ from above.
Moreover, we let $\rho_{\eps,\kappa} \colon C_{\eps,\kappa} \to [0,1]$ be given by
\begin{equation} \label{eq_cut_off_cyl}
\rho_{\eps,\kappa}(\theta,y)=
\begin{cases}
\eps^k (y-(R_\eps-1/\eps^k)) & \ \text{if} \ y \in [R_\eps-1/\eps^k,R_\eps]
\\
0 & \ \text{else},
\end{cases}
\end{equation}
where we assume that $\eps \leq \eps_{4}$, such that $R_\eps-1/\eps^k \geq 1$. 
Thanks to the conformal invariance of the Dirichlet energy, we have that
\begin{equation} \label{eq_grad_cut_off}
\int_{C_{\eps,\kappa}} |\nabla \rho_{\eps,\kappa}|^2 \leq C \eps^{k+1},
\end{equation}
for some fixed constant $C>0$.
Note that we may also decrease $\eps_4$ such that
\begin{equation} \label{eq_area_cut_off}
\area(\supp \nabla \rho_{\eps,\kappa}) \leq
 \frac{C \eps}{(\eps^k R_{\eps} -1)R_{\eps}}
\leq
 \eps^{k+1}
\end{equation}
if $\eps \leq \eps_4$. 

If $\phi \colon \Sigma \to \IR$ is again a $\lambda_1(\Sigma)$-eigenfunction we define
$$
\tilde \phi_{\eps,\kappa} = 
\begin{cases}
\phi &  \ \text{in} \ \Sigma \setminus ( B_{2\eps^k}(x_1) \cup B_{2 \eps^k (x_0)})
\\
\eta_{\eps} \phi + (1-\eta_\eps) \phi(x_1) & \ \text{in} \ B_{2\eps^k}(x_1)
\\
\eta_{\eps} \phi + (1-\eta_\eps) \phi(x_0) & \ \text{in} \ B_{2\eps^k}(x_0)
\\
\rho_{\kappa,\eps} \phi(x_1) + (1-\rho_{\eps,\kappa}) \phi(x_0) & \ \text{on} \ C_{\eps,\kappa}
\end{cases}
$$

A computation almost completely analogous to the proof of \cref{mono_quasimode_1} using also \eqref{eq_grad_cut_off} and \eqref{eq_area_cut_off}
then implies the following lemma.

\begin{lemma} \label{mono_quasimode_2}
We have for
the function $\tilde \phi_{\eps,\kappa}$ defined above and any $\varphi \in W^{1,2}(\Sigma_{\eps,\kappa})$ that
\begin{equation*}
\left| \int_{\Sigma_{\eps,\kappa}} \nabla \tilde \phi_\eps \cdot \nabla \varphi - \lambda_1(\Sigma) \int_{\Sigma_{\eps,\kappa}} \tilde \phi_\eps \varphi
+\lambda_1(\Sigma) \phi(x_0) \int_{C_{\eps,\kappa}} \varphi
 \, \right|
\leq
C \eps^{k/2} \| \varphi\|_{W^{1,2}(\Sigma_{\eps,\kappa})},
\end{equation*}
where the constant $C$ depends only on $\Sigma$.
\end{lemma}

\subsubsection{Asymptotic expansions from quasimodes - part $1$}

We now fix an orthonormal basis $(\phi_0,\dots,\phi_{K-1} )$ of $\lambda_1(\Sigma)$-eigenfunctions, such that 

\begin{equation} \label{eq_normal_basis}
\phi_1(x_0)= \dots =\phi_{K-1}(x_0)=0.
\end{equation}
We denote the extension of $\phi_l$ to $\Sigma_{\eps,\kappa}$ constructed above by $\tilde \phi_{\eps,\kappa,l}$.
By testing against an eigenfunction on $\Sigma_{\eps,\kappa}$, we obtain from
\cref{mono_quasimode_2} the following expansion of the corresponding eigenvalue.
Of course, if we use \cref{mono_quasimode_1} in place of \cref{mono_quasimode_2} we obtain the analogous statement if we attach a cross cap instead of a cylinder.

\begin{cor} \label{expan_surf}
Let $u_{\eps,\kappa}$ be a normalized eigenfunction on $\Sigma_{\eps,\kappa}$ with eigenvalue $\lambda_{\eps,\kappa}$.
Write 
$$
\tilde m_0 = \int_{\Sigma_{\eps,\kappa}} u_{\eps,\kappa} \tilde \phi_{\eps,\kappa,0}
$$
and
$$
\tilde m = \sum_{j=1}^{K-1}  \left| \int_{\Sigma_{\eps,\kappa}} u_{\eps,\kappa} \tilde \phi_{\eps,\kappa,j} \right|.
$$
Then we have that
$$
\lambda_{\eps,\kappa} 
= 
\lambda_1(\Sigma) -
\frac{\lambda_1(\Sigma) \phi_0(x_0)\int_{C_{\eps,\kappa}}u_{\eps,\kappa} + O(\eps^{k/2}) }{\tilde m_0}
=
\lambda_1(\Sigma) - \frac{O(\eps^{k/2})}{\tilde m}
$$
as $\eps \to 0$ and the error terms depend only on $\Sigma$ provided $\tilde m_0 \neq 0$, or $\tilde m \neq 0$ respectively. 
\end{cor}

\begin{proof}
The only thing not immediate from \cref{mono_quasimode_2} is why we can use the absolute value in the definition of $\tilde m$.
By multiplying $\phi_j$ by $-1$ we may assume that each of the summands is positive and then use the sum of all these eigenfunctions as test function.
\end{proof}

This will be particularly useful once we have the results from \cref{sec_conc}, that give some understanding of the numbers $\tilde m_0$ and $\tilde m$  under some assumptions on $\lambda_{\eps,\kappa}.$
Moreover, this is the first piece of evidence that the eigenfunction $\phi_0$ should play a distinguished role for the interaction of the two parts of the spectrum of $\Sigma_{\eps,\kappa}$.

\subsection{Quasimodes concentrated on a cylinder or cross cap}

We want to construct quasimodes from the Dirichlet eigenfunctions of the cross cap $M_{\eps,\kappa}$ or the cylinder $C_{\eps,\kappa}$, respectively.
In order to obtain a good quasimode we need to find a good extension of such a Dirichlet eigenfunction to $\Sigma \setminus B_{\eps^k}$.
In principal one would like to use the Green's function of $\Delta-\lambda$ with pole at $x_0$.
While this works very well for a \emph{fixed} choice of the parameter $\kappa$, we need to be more careful when considering the whole family $\Sigma_{\eps,\kappa}$.
The presence of a non-trivial kernel of $\Delta - \lambda_l(M_{\eps,\kappa})$ on $\Sigma$ for some $\kappa$ forces us to modify
the Green's function a bit. 
It is worth pointing out, that, while the gain in the scale up to which this extension solves the eigenvalue equation is minimal when compared to  the Dirichlet eigenfunction itself, it will play an absolutely
crucial role that we gain a great deal of qualitative understanding from this.

\subsubsection{The Green's function of $(\Delta_\Sigma - \lambda)$}

We need some preliminaries on a function closely related to the Green's function of the operator $\Delta-\lambda$ on $\Sigma$.
For the convenience of the reader, the short \cref{sec_green} contains a proof of the facts on Green's functions that we make use of below.
Recall that if we normalize $\area(\Sigma)=1$ the Green's function $G(x,y)$ of $\Delta$ solves

\begin{equation*}
\Delta_y G(x,y)=\delta_x - 1.
\end{equation*}
in the sense of distributions.
Near the diagonal, the Green's function is asymptotic to the Green's function of the Euclidean plane.
More precisely, for $x \in \Sigma$ fixed, we have that
\begin{equation} \label{green_asymp_1}
G(x,y)=\frac{1}{2 \pi} \log\left(\frac{1}{|x-y|}\right) + \psi_x(y)
\end{equation}
where $|x-y|$ is the distance with respect to the Euclidean metric in conformal coordinates near $x$ normalized such that $g=fg_e$ with $f(x)=1$ and $\psi_x$ is a smooth function.
Off the diagonal, $G$ is a smooth function.
In particular, we find that
\begin{equation} \label{integral_green_bd}
\int_\Sigma |G(x,y)|^p dy \leq C
\end{equation}
for any $1\leq p < \infty$ and some uniform constant $C=C(\Sigma,p)$.

Let $(\phi_0,\dots,\phi_{K-1})$ be an orthonormal basis of the $\lambda_1(\Sigma)$-eigenspace as in \eqref{eq_normal_basis}. We consider the function
\begin{equation*}
f(y)=G(x_0,y)-\sum_{i=0}^{K-1} \int_\Sigma G(x_0,z) \phi_i(z)dz \, \phi_i(y),
\end{equation*}
which is well-defined by H{\"o}lder's inequality and \eqref{integral_green_bd}.
Also from \eqref{integral_green_bd} and Minkowski's inequality, we find that
\begin{equation} \label{f_lp_bd}
\int_\Sigma |f|^p \leq C,
\end{equation}
for a constant $C=C(\Sigma,p)$.
In particular, for any $\lambda \in (0,\lambda_2(\Sigma))$ there is unique solution $u_\lambda \in W^{1,2}(\Sigma)$ that is orthogonal to $\langle \phi_0, \dots , \phi_{K-1} \rangle$ and such that
\begin{equation*}
(\Delta-\lambda)u_\lambda = \lambda f+1
\end{equation*}
since $f$ and the constant functions are orthogonal to the kernel and hence also the cokernel of $(\Delta-\lambda)$
(which for the relevant $\lambda$ is trivial if $\lambda \neq \lambda_1(\Sigma)$ and  equal to $\langle \phi_0, \dots , \phi_{K-1} \rangle$ if $\lambda = \lambda_1(\Sigma)$).
It follows from \eqref{f_lp_bd} and standard elliptic estimates that $u_\lambda$ is uniformly bounded in $W^{2,p}(\Sigma)$ as long as $\lambda \in [\delta_0 ,\lambda_{K+1}(\Sigma)-\delta_0]$ for some small $\delta_0>0$.
We now fix $\delta_0>0$ once and for all such that 
\begin{equation}
\label{delta}
\delta_0 < \lambda_1(\Sigma) < \lambda_{K+1}(\Sigma) - \delta_0. 
\end{equation}
The Sobolev embedding theorem yields that $u_\lambda$ is uniformly bounded in $C^{1,\alpha}(\Sigma)$ for some $\alpha>0$.
Consider the function 
\begin{equation*}
H_\lambda(y)=G(x_0,y)+u_\lambda(y),
\end{equation*}
which solves
\begin{equation} \label{eq_green_proj_kernel}
\begin{split}
(\Delta - \lambda) H_\lambda
&= 
- \lambda \sum_{i=0}^{K-1}\int_\Sigma G(x_0,z) \phi_i(z)dy \, \phi_i
\\
& =
- \frac{\lambda}{\lambda_1(\Sigma)} \phi_0(x_0)\phi_0
\end{split}
\end{equation}
in $\Sigma \setminus \{x_0\}$ by the normalization \eqref{eq_normal_basis}.

Since $u_\lambda$ is uniformly bounded in $C^{1,\alpha}(\Sigma)$ for a fixed $\alpha>0$, we find from \eqref{green_asymp_1} that the function
\begin{equation*}
e_{\lambda}(y):=H_\lambda(y)- \frac{1}{2 \pi} \log\left(\frac{1}{|x_0-y|}\right) 
\end{equation*}
is uniformly bounded in $C^{1,\alpha}(\Sigma)$.
Therefore, 
\begin{equation*}
e_{\eps,\lambda}:= \frac{2\pi e_\lambda(x_0)}{\log(1/\eps^k)}=o(1)
\end{equation*}
as $\eps \to 0$ uniformly in $\lambda \in [\delta_0,\lambda_2(\Sigma)-\delta_0]$.

The right hand side of \eqref{eq_green_proj_kernel} is the second reason for the special role of $\phi_0$.

\subsubsection{Quasimodes concentrated on the cross cap}

For $\eps \leq \eps_0$ and $l \in \IN$ such that $\lambda_l(M_{\eps,\kappa}) \leq \lambda_{K+1}(\Sigma)+1$ the corresponding
$L^2(M_{\eps,\kappa})$-normalized Dirichlet eigenfunction $\psi_{\eps,\kappa,l}$ is rotationally symmetric by the choice of $\eps_0$ -- see the text preceding \cref{spec_cusp_cor_approx_order}.
The eigenvalue
$\lambda_l(M_{\eps,\kappa})$ is given approximately by $\kappa^2/4+(l+1)^2 \kappa^2 \pi^2 \eps^{2 \alpha}$.
We want to find a good extension $\tilde \psi_{\eps,\kappa,l} \colon \Sigma_{\eps,\kappa} \to \IR$.
Essentially, we do this using the function $H_\lambda$ constructed above, scaled in such  a way that the normal derivatives along $\partial B_{\eps^k}(x_0) = \partial M_{\eps,\kappa}$ cancel out.

We write\footnote{We use the convention that the domain of integration also indicates which normal and measure we use.
This is particularly important along $\partial M_{\eps,\kappa}=\partial B_{\eps^k}$, where these differ significantly.}
$$
a_{\eps,\kappa,l} =  -\int_{\partial M_{\eps,\kappa}} \partial_\nu \psi_{\eps,\kappa,l} d \mathcal{H}^1
$$
for the factor by which we have to scale $H_\lambda$ on $\Sigma \setminus B_{\eps^k}(x_0)$.
Recall that \cref{spec_cusp_cor_approx_order} gives
$$
|a_{\eps,\kappa,l}| \leq C (l+1)  \eps^{3\alpha / 2  +1/2}
$$
for some fixed constant $C>0$ as long as $\kappa \in [\kappa_0,\kappa_1]$ (with $\kappa_0>0$).

Recall the cut-off functions $\eta_\eps \colon \Sigma \setminus B_{\eps^k}(x_0) \to  [0,1]$ defined in \cref{sec_quasimodes_1}.
We define $\tilde \psi_{\eps,\kappa,l} \in W^{1,2}(\Sigma_{\eps,\kappa})$ as follows,
\begin{equation*}
\tilde \psi_{\eps,\kappa,l}(y)=
\begin{cases}
 a_{\eps,\kappa,l} \left(\eta_\eps H_\lambda(y) + (1-\eta_\eps)(y)   \left(\frac{1}{2\pi}\log \left( \frac{1}{|x_0-y|} \right) + e_{\lambda}(x_0) \right) \right) & \text{on} \ \Sigma \setminus B_{\eps^k} (x_0)
\\
\psi_{\eps,\kappa,l} + a_{\eps,\kappa,l}(1+e_{\eps,\lambda}) \frac{1}{2\pi}\log\left(\frac{1}{\eps^k}\right) & \text{on} \ M_{\eps,\kappa},
\end{cases}
\end{equation*}
where $\lambda=\lambda_l(M_{\eps,\kappa})$.
By construction, $\tilde \psi_{\eps,\kappa,l}$ is a Lipschitz function, in particular $\tilde \psi_{\eps,\kappa,l} \in W^{1,2}(\Sigma_{\eps,\kappa})$.
The key property of $\tilde \psi_{\eps, \kappa,l}$ 
is that we have
\begin{equation} \label{eq_cancel_1}
\begin{split}
\int_{\partial (\Sigma \setminus B_{\eps^k}(x_0))} \varphi \partial_{\nu_{\it{eucl}}} \tilde \psi_{\eps,\kappa,l}d\mathcal{H}^1_{\it{eucl}}
&=
\frac{a_{\eps,\kappa,l}}{2\pi} \int_{\partial B_{\eps^k}(x_0)} \varphi \partial_r \log(r) d \mathcal{H}^1_{\it{eucl}}
\\
&=
\frac{a_{\eps,\kappa,l}}{2\pi \eps^k} \int_{\partial B_{\eps^k}(x_0)} \varphi d \mathcal{H}^1_{\it{eucl}}
\\
&=
\frac{\kappa a_{\eps,\kappa,l}}{2\pi \eps} \int_{\partial M_{\eps,\kappa}} \varphi  d \mathcal{H}^1
\\
&=
-\int_{\partial M_{\eps,\kappa}} \varphi \partial_{\nu}  \tilde \psi_{\eps,\kappa,l} d\mathcal{H}^1,
\end{split}
\end{equation}
since $\partial_\nu \tilde \psi_{\eps,\kappa,l}$ is a constant function along $\partial M_{\eps,\kappa}$.
Moreover, in dimension two, we have that
\begin{equation} \label{eq_cancel_2}
  \int_{\partial (\Sigma \setminus B_{\eps^k}(x_0))} \varphi \partial_\nu \tilde \psi_{\eps,k,l} d \mathcal{H}^1 
=
\int_{\partial (\Sigma \setminus B_{\eps^k}(x_0))} \varphi \partial_{\nu_{\it{eucl}}} \tilde \psi_{\eps,\kappa,l}d\mathcal{H}^1_{\it{eucl}},
\end{equation}
since $g$ is conformal to the Euclidean metric on $\partial (\Sigma \setminus B_{\eps^k})$.

This will allow us to measure very precisely on which scale $\tilde \psi_{\eps,\kappa,l}$ fails to be in the kernel of $(\Delta-\lambda)$ on $\Sigma_{\eps,\kappa}$.

\begin{lemma} \label{lem_asymp_1}
For any $\varphi \in W^{1,2}(\Sigma_{\eps,\kappa})$ we have that
\begin{equation} \label{eq_asymp_1}
\begin{split}
\int_{\Sigma_{\eps,\kappa}}
&
 \nabla \tilde \psi_{\eps,\kappa,l} \nabla \varphi 
- \lambda_l(M_{\eps,\kappa})
\int_{\Sigma_{\eps,\kappa}}  \tilde \psi_{\eps,\kappa,l}  \varphi
=
- a_{\eps,\kappa,l} \frac{\lambda_l(M_{\eps,\kappa})}{\lambda_1(\Sigma)}\phi_0(x_0) \int_{\Sigma \setminus B_{2\eps^k}}  \phi_0 \varphi 
\\
& - a_{\eps,\kappa,l} \log(1/\eps^k) \frac{\lambda_l(M_{\eps,\kappa})}{2 \pi}(1+e_{\eps,\lambda_l(M_{\eps,\kappa})}) \int_{M_{\eps,\kappa}} \varphi + \|\varphi\|_{W^{1,2}(\Sigma_{\eps,\kappa})} O(\eps^k \log(1/\eps^k)).
\end{split}
\end{equation}
as $\eps \to 0$, uniformly in $\kappa \in [\kappa_0,\kappa_1]$ and uniformly for $\lambda_l(M_{\eps,\kappa}) \in [\delta_0,\lambda_{K+1}(\Sigma)-\delta_0].$
\end{lemma}

It will be crucial later that this only needs $\lambda_l$ to be controlled rather than $l$. 
For the definition of $\delta_0$ see (\ref{delta}).

\begin{rem} \label{rem_order}
Note that, by H{\"o}lder's inequality,
\begin{equation*}
\begin{split}
\left |a_{\eps,\kappa,l} \log(1/\eps^k)  \int_{M_{\eps,\kappa}} \varphi \right|
&\leq 
|a_{\eps,\kappa,l}| \log(1/\eps^k) \area(M_{\eps,\kappa})^{1/2}  \|\varphi\|_{L^2(\Sigma_{\eps,\kappa})}
\\
& \leq
 C |a_{\eps,\kappa,l}| \eps^{1/2} \log(1/\eps^k)  \|\varphi\|_{L^2(\Sigma_{\eps,\kappa})},
 \end{split}
\end{equation*}
so that the first summand on the right hand side of \eqref{eq_asymp_1} is (typically) the term of lower order. 
\end{rem}

\begin{proof}
We fix $l \in \IN$ and write $\lambda=\lambda_l(M_{\eps,\kappa})$.\
Integration by parts gives that
\begin{equation*}
\begin{split}
 \int_{\Sigma_{\eps,\kappa}}  \nabla \tilde \psi_{\eps,k,l} \nabla \varphi
 &=
 \int_{\Sigma \setminus B_\eps(x_0)} \nabla \varphi \nabla \tilde \psi_{\eps,k,l}
 + 
 \int_{M_{\eps,\kappa}} \nabla \varphi \nabla \tilde \psi_{\eps,k,l}
 \\
 &=
 \int_{\Sigma \setminus B_\eps(x_0)} \varphi \Delta \tilde \psi_{\eps,k,l} 
 -
 \int_{\partial B_\eps(x_0)} \varphi \partial_\nu \tilde \psi_{\eps,k,l} d \mathcal{H}^1
 \\
 &+
  \int_{M_{\eps,\kappa}} \varphi \Delta \tilde \psi_{\eps,k,l} 
  +
  \int_{\partial M_{\eps,\kappa}} \varphi \partial_{\nu} \tilde \psi_{\eps,k,l} d \mathcal{H}^1
  \\
  & =
   \int_{\Sigma \setminus B_\eps(x_0)} \varphi \Delta \tilde \psi_{\eps,k,l} 
   + \int_{M_{\eps,\kappa}} \varphi \Delta \tilde \psi_{\eps,k,l},
\end{split}
\end{equation*}
where the last step follows from \eqref{eq_cancel_1} and \eqref{eq_cancel_2}.

Moreover, we have that
\begin{equation*}
\int_{M_{\eps,\kappa}} \varphi \Delta \tilde \psi_{\eps,k,l} =
\lambda \int_{M_{\eps,\kappa}} \varphi \tilde \psi_{\eps,k,l} - \frac{\lambda}{2\pi} (1+e_{\eps,\lambda}) a_{\eps,\kappa,l} \log(1/\eps^{k})  \int_{M_{\eps,\kappa}} \varphi .
\end{equation*}
In order to estimate the integral on $\Sigma \setminus B_{\eps^k}(x_0)$ note that we have 
\begin{equation*}
\begin{split}
\int_{\Sigma \setminus B_{\eps^k}(x_0)} \varphi (\Delta - \lambda) \tilde \psi_{\eps,k,l}
=&
- \frac{\lambda}{\lambda_1(\Sigma)} a_{\eps,\kappa,l} \int_{\Sigma \setminus B_{2 \eps^k}(x_0)} \phi_0(x_0)\phi_0 \varphi
\\
&+
\int_{B_{2 \eps^k}(x_0) \setminus B_{\eps^k}(x_0)} \varphi(\Delta-\lambda)\tilde \psi_{\eps,k,l}.
\end{split}
\end{equation*}
It remains to estimate the second summand.
First note that, on $B_{2 \eps^k}(x_0) \setminus B_{\eps^k}(x_0)$, we have that
\begin{equation*}
\begin{split}
a_{\eps,\kappa,l}^{-1}(\Delta - \lambda) \tilde \psi_{\eps,k,l}
=&
- \eta_\eps \frac{\lambda}{\lambda_1(\Sigma)} \phi_0(x_0)\phi_0  
  + 
 2\nabla \eta_\eps \cdot \left( \nabla H_\lambda- \frac{1}{2\pi} \nabla \log \left( \frac{1}{|x_0-y|} \right) \right)
 \\
 &+
 \Delta \eta_\eps \left(H_\lambda -  \frac{1}{2 \pi}  \log \left( \frac{1}{|x_0-y|}\right) - e_\lambda(x_0) \right)
 \\
 &- 
 \lambda (1-\eta_\eps) \left( \frac{1}{2 \pi}  \log \left( \frac{1}{|x_0-y|}\right) + e_\lambda(x_0) \right)
 \end{split}
\end{equation*}
since $\Delta \log(1/|x_0-y|)=0$ thanks to the conformal covariance of the Laplacian.
Therefore, we find that
\begin{equation*}
\begin{split}
\left| \int_{B_{2 \eps^k}(x_0) \setminus B_{\eps^k}(x_0)} \right. & \left. \varphi (\Delta - \lambda) \tilde \psi_{\eps,k,l}
\right|
\leq 
C |a_{\eps,\kappa,l}|
\left( \int_{B_{2\eps^k} (x_0) \setminus B_{\eps^k}(x_0)} |\varphi|^2 \right)^{1/2} \times
\\
& 
\times \left( \left(
 \int_{B_{2\eps^k} \setminus B_{\eps^k}}  |\phi_0(x_0) \phi_0|^2 +  |\nabla \eta_\eps|^2 |\nabla e_\lambda|^2 +  |\Delta \eta_\eps|^2 |e_\lambda - e_\lambda (x_0)|^2 \right)^{1/2}  \right.
\\
&+ 
\left.
\left( \int_{B_{2\eps^k} (x_0) \setminus B_{\eps^k}(x_0)}  \lambda^2 |1-\eta_\eps|^2 |\log(1/\eps^k) +e_\lambda(x_0)|^2 \right)^{1/2} \right)
\\
\leq &
C |a_{\eps,\kappa,l}| \eps^k \log(1/\eps^k) (\eps^{2k} + C + C + C \eps^{2k} \log^2(1/\eps^k))^{1/2}
\\
\leq & 
C |a_{\eps,\kappa,l}| \eps^{k} \log(1/\eps^k)
\end{split}
\end{equation*}
for $\eps>0$ sufficiently small,
where we have used \eqref{eq_int_bd_sobolev} and $e_\lambda \in C^{1,\alpha}(\Sigma)$.
Also note that $|a_{\eps,\kappa,l}|$ is uniformly bounded as long as $\lambda$ is.
The assertion now follows from combining all the above estimates.
\end{proof}

\subsubsection{Quasimodes concentrated on the cylinder}

The construction of quasimodes $\tilde \psi_{\eps,\kappa,l}$ concentrated on the cylinder $C_{\eps,\kappa}$ is almost completely analogous to the construction for cross caps above.

We do not have to be so careful in extending the Dirichlet eigenfunction $\psi_{\eps,\kappa,l}$ near the short boundary component $\partial_{R_\eps} C_{\eps,\kappa} = \{y=R_\eps\}$ to $\Sigma \setminus B_{\eps^k}$ since, by \cref{spec_cusp_cor_approx_order}, the $L^1$-norm of the normal derivative $\partial_\nu \psi_{\eps,\kappa,l}$ is on a much smaller scale than
the corresponding norm on the long boundary component.
More precisely,
in this case we write
$$
a_{\eps,\kappa,l}
=-\int_{\partial_1 C_{\eps,\kappa}} \partial_\nu \psi_{\eps,\kappa,l} d \mathcal{H}^1,
$$
where $\psi_{\eps,\kappa,l}$ is a normalized $\lambda_l(C_{\eps,\kappa})$-eigenfunction that we assume to be rotationally symmetric.
We have from \cref{spec_cusp_cor_approx_order} that 
$$
|a_{\eps,\kappa,l}| \leq C (l+1) \eps^{3\alpha/2+1/2}
$$
as long as $\lambda_l(C_{\eps,\kappa}) \leq \lambda_{K+1}(\Sigma)$ and $\eps \leq \eps_0$.
On the other hand, also from \cref{spec_cusp_cor_approx_order}, we find that
\begin{equation} \label{eq_scale_short}
\begin{split}
\int_{\partial_{R_\eps} C_{\eps,\kappa}} |\varphi \partial_\nu \psi_{\eps,\kappa,l}| d \mathcal{H}^1
&\leq 
C (l+1) \eps^{3\alpha/2+1/2} \log(1/\eps) R_\eps^{-1/2} \|\varphi\|_{W^{1,2}(\Sigma_{\eps,\kappa})}
\\
& \leq
C \eps^{k} \|\varphi\|_{W^{1,2}(\Sigma_{\eps,\kappa})}
\end{split}
\end{equation}
if $\lambda_l(C_{\eps,\kappa}) \leq \lambda_{K+1}(\Sigma)$ and $\eps \leq \eps_5$
thanks to \eqref{eq_int_bd_sobolev}.

Something that we need to take care of that was not present in the case of a cross cap is that we need to extend $H_\lambda$ from $\partial B_{\eps^k}(x_1)$ to $C_{\eps,\kappa}$.
We do this exactly as in the construction preceeding \cref{mono_quasimode_2} by interpolating $H_\lambda$ with $H_\lambda(x_1)$ and then cutting-off the constant function $H_{\lambda}(x_1)$ near the corresponding
component of $\partial C_{\eps,\kappa}$.

In order to simplify the notation a little bit, we will simply write 
$$
B_{\eps^k} = B_{\eps^k}(x_0) \cup B_{\eps^k}(x_1)
$$
from here on and $\eta_\eps \colon \Sigma \setminus B_{\eps^k} \to [0,1]$ is the cut-off function constructed from $\eta$ above that cuts off near $x_0$ and near $x_1$.

The precise construction of the quasimodes  $\tilde \psi_{\eps,\kappa,l}$ is as follows:
\begin{equation*}
\tilde \psi_{\eps,\kappa,l}(y)=
\begin{cases}
 a_{\eps,\kappa,l} \left(\eta_\eps H_\lambda(y) + (1-\eta_\eps) (y)  \left(\frac{1}{2\pi}\log \left( \frac{1}{|x_0-y|} \right) + e_{\lambda}(x_0) \right) \right) & \text{on} \  \Sigma \setminus (B_{\eps^k} (x_0) \cup B_{2 \eps^k} (x_1))
\\
a_{\eps,\kappa,l} (\eta_\eps H_\lambda + (1- \eta_\eps) H_\lambda(x_1) ) & \text{in} \ B_{2 \eps^k}(x_1) \setminus B_{\eps^k}(x_1).
\\
\psi_{\eps,\kappa,l} + a_{\eps,\kappa,l}(\rho_{\eps,\kappa} H_\lambda(x_1)
+(1-\rho_{\eps,\kappa})(1+e_{\eps,\lambda}) \frac{1}{2\pi}\log\left(\frac{1}{\eps^k}\right))
& \text{on} \ C_{\eps,\kappa}
\end{cases}
\end{equation*}
where $\lambda=\lambda_l(C_{\eps,\kappa})$ and $\rho_{\eps,\kappa} \colon C_{\eps,\kappa} \to [0,1]$ is the cut-off function introduced in \eqref{eq_cut_off_cyl}.

The arguments from the proofs of \cref{mono_quasimode_2} and \cref{lem_asymp_1} using \eqref{eq_grad_cut_off} and \eqref{eq_area_cut_off} combined with
\eqref{eq_scale_short} imply the following.

\begin{lemma} \label{lem_asymp_2}
For any $\varphi \in W^{1,2}(\Sigma_{\eps,\kappa})$ we have that
\begin{equation} \label{eq_asymp_2}
\begin{split}
\int_{\Sigma_{\eps,\kappa}}
&
 \nabla \tilde \psi_{\eps,\kappa,l} \nabla \varphi 
- \lambda_l(C_{\eps,\kappa})
\int_{\Sigma_{\eps,\kappa}}  \tilde \psi_{\eps,\kappa,l}  \varphi
=
- a_{\eps,\kappa,l} \frac{\lambda_l(C_{\eps,\kappa})}{\lambda_1(\Sigma)} \phi_0(x_0) \int_{\Sigma \setminus B_{2\eps^k}}  \phi_0 \varphi 
\\
& - a_{\eps,\kappa,l} \log(1/\eps^k) \frac{\lambda_l(C_{\eps,\kappa})}{2 \pi}(1+e_{\eps,\lambda_l(C_{\eps,\kappa})}) \int_{C_{\eps,\kappa}} \varphi + \|\varphi\|_{W^{1,2}(\Sigma_{\eps,\kappa})} O(\eps^k \log(1/\eps^k)),
\end{split}
\end{equation}
as $\eps \to 0$, uniformly in $\kappa \in [\kappa_0,\kappa_1]$ and uniformly in $\lambda_l(C_{\eps,\kappa}) \in [\delta_0,\lambda_{K+1}(\Sigma)-\delta_0].$
\end{lemma}

Let us also record for later use that
$$
\int_{C_{\eps,\kappa}} |\varphi (\psi_{\eps,\kappa,l} -  \tilde \psi_{\eps,\kappa,l})|
\leq
C |a_{\eps,\kappa,l}| \eps^{1/2} \log(1/\eps) \|\varphi\|_{L^2}
$$
thanks to H{\"o}lder's inequality.

\subsubsection{Asymptotic expansions from quasimodes - part $2$}

Like in the case of quasimodes concentrated on $\Sigma$ we can also in this case get asmymptotic expansions
of eigenvalues from testing the quasimodes against eigenfunctions.

\begin{cor} \label{expan_cyl}
Let $u_{\eps,\kappa}$ be a normalized eigenfunction on $\Sigma_{\eps,\kappa}$ with eigenvalue $\lambda_{\eps,\kappa}$.
For $l \in \IN$ such that $\lambda_l(C_{\eps,\kappa}) \in [\delta_0,\lambda_{K+1}(\Sigma)-\delta_0]$,
write 
$$
\tilde n_{l} = \int_{\Sigma_{\eps,\kappa}} u_{\eps,\kappa} \tilde \psi_{\eps,\kappa,l}.
$$
Then we have that
\begin{equation*}
\begin{split}
\lambda_{\eps,\kappa} 
&= 
\lambda_l(C_{\eps,\kappa}) -
\frac{a_{\eps,\kappa,l}  \frac{\lambda_l(C_{\eps,\kappa})}{\lambda_1(\Sigma)} \phi_0(x_0) \int_{\Sigma \setminus B_{2 \eps^k}}  \phi_0 u_{\eps,\kappa} }{\tilde n_l}
\\
-& \frac{ a_{\eps,\kappa,l} \log(1/\eps^k) \frac{\lambda_l(C_{\eps,\kappa})}{2 \pi}(1+e_{\eps,\lambda_l(C_{\eps,\kappa})}) \int_{C_{\eps,\kappa}} \varphi }{\tilde n_l}
\\
&+ \frac{ \|\varphi\|_{W^{1,2}(\Sigma_{\eps,\kappa})} O(\eps^k \log(1/\eps^k)) } {\tilde n_l }
\end{split}
\end{equation*}
as $\eps \to 0$ and the error term depends only on $\Sigma, \kappa_0,\kappa_1$, and $\delta_0$.
\end{cor}

\section{$L^2$-decomposition of eigenfunctions} \label{sec_conc}

The main goal of this section is to understand how well eigenfunctions on $\Sigma_{\epsilon,\kappa}$ with eigenvalue close to $\lambda_1(\Sigma)$ or $\lambda_0(C_{\eps,\kappa})$ can be 
approximated by, or decomposed into, the quasimodes constructed in the previous section.
Afterwards we will use this to derive the preliminary lower bound  
$$
\lambda_2(\Sigma_{\eps,\kappa})
\geq \lambda_1(\Sigma) - d_0 \eps \log(1/\eps)
$$
for a large set of parameters $\kappa$ and some constant $d_0>0$.
While this looks very far away from the bound we want to obtain eventually, it will turn out to be extremely useful, because it also holds for $\kappa$ relatively small\footnote{In fact, for $\kappa$ not too small, we get the corresponding bound for $\lambda_1(\Sigma_{\eps,\kappa})$, but we will not make use of this.}.
In the next section we will use the results from this section to compare $\lambda_1(\Sigma_{\eps,\kappa})$ to $\mu_1(\Sigma \setminus B_{\eps^k})$
for very specific choices of the parameter $\kappa.$

A major technical difficulty is that the number of eigenvalues close to $\lambda_0(C_{\eps,\kappa})$
is unbounded for $\eps \to 0.$
However, if $\alpha<1$, it turns out that one can locate eigenvalues up to a scale smaller than the separation of 
two consecutive eigenvalues on $C_{\eps,\kappa}$.

In this section we will formulate and prove all results only for the case of attaching cylinders.
The case of cross caps is completely analogous invoking our previous estimates in the case of cross caps instead of those for cylinders.

\subsection{Decomposing eigenfunctions on the cylinder} \label{sec_conc_cyl}

We fix 
$$
\{\psi_{\eps,\kappa,l}\}_{l \geq 0} \subset L^2(C_{\eps,\kappa})
$$ 
an orthonormal basis of Dirichlet eigenfunctions on $C_{\eps,\kappa}$.
From here on we restrict to parameters $\eps \leq \eps_6:= \min\{\eps_0, \dots, \eps_5\}$, which
only depend on $\Sigma,\delta_0,\kappa_0,\kappa_1$.
In particular, all $\psi_{\eps,\kappa,l}$ with eigenvalue $\lambda_l(C_{\eps,\kappa}) \leq \lambda_{K+1}(\Sigma) - \delta_0$ are rotationally symmetric and we can use the corresponding
quasimodes constructed in the preceding section.
We start with the following observation that is to some extent already implicit in the proof of \cref{thm_eig_lim}.

\begin{lemma} \label{L^2_lem_cyl}
Let $u_{\eps,\kappa}$ be an $L^2$-normalized eigenfunction on $\Sigma_{\eps,\kappa}$ with eigenvalue $\lambda_{\eps,\kappa} \leq \Lambda$ for some fixed $\Lambda>0$.
Assume that $\lambda_N(C_{\eps,\kappa})> \lambda_{\eps,\kappa}+c$ for some $c>0$ and some $N \in \IN$.
Then, for $\eps \leq \eps_6$, we have
$$
\sum_{l = N}^\infty \left( \int_{C_{\eps,\kappa}} u_{\eps,\kappa} \psi_{\eps,\kappa,l} \right)^2 \leq C \left( 1+ \frac{1} {c^2}  \right) \eps \log(1/\eps)
$$
for a uniform constant $C=C(\Sigma, k,\kappa_0,\kappa_1,\Lambda)$.
\end{lemma}

\begin{proof}
Let $v_{\eps,\kappa} \in W_0^{1,2}(C_{\eps,\kappa})$
be defined by
$$
v_{\eps,\kappa} = u_{\eps,\kappa} - h_{\eps,\kappa},
$$
where $h_{\eps,\kappa}$ is the harmonic extension of $\left. u_{\eps,\kappa} \right|_{\partial C_{\eps,\kappa}}$
to $C_{\eps,\kappa}$.
Since $v_{\eps,\kappa} \in W_0^{1,2}(C_{\eps,\kappa})$, testing it against a normalized $\lambda_l(C_{\eps,\kappa})$-eigenfunction
$\psi_{\eps,\kappa,l}$ gives
\begin{equation*}
\begin{split}
\lambda_{\eps,\kappa} \int_{C_{\eps,\kappa}} u_{\eps,\kappa} \psi_{\eps,\kappa,l} 
&=
\int_{C_{\eps,\kappa}} (\Delta u_{\eps,\kappa}) \psi_{\eps,\kappa,l} 
=
\int_{C_{\eps,\kappa}} (\Delta v_{\eps,\kappa}) \psi_{\eps,\kappa,l} 
\\
&=
\int_{C_{\eps,\kappa}}  v_{\eps,\kappa} \Delta \psi_{\eps,\kappa,l} =
\int_{C_{\eps,\kappa}}  (u_{\eps,\kappa}- h_{\eps,\kappa}) \Delta \psi_{\eps,\kappa,l} 
\\
&=
\lambda_l(C_{\eps,\kappa}) \int_{C_{\eps,\kappa}}  u_{\eps,\kappa} \psi_{\eps,\kappa,l} 
-
\lambda_l(C_{\eps,\kappa}) \int_{C_{\eps,\kappa}} h_{\eps,\kappa} \psi_{\eps,\kappa,l}. 
\end{split}
\end{equation*}
By assumption, $\lambda_N(C_{\eps,\kappa}) - \lambda_{\eps,\kappa} \geq c$, therefore, we find that
\begin{equation} \label{eq_high_freq_1}
\begin{split}
\left(\int_{C_{\eps,\kappa}} u_{\eps,\kappa} \psi_{\eps,\kappa,l} \right)^2 
&=
\left(\frac{\lambda_l(C_{\eps,\kappa})}{\lambda_l(C_{\eps,\kappa})-\lambda_{\eps,\kappa}}\right)^2 \left( \int_{C_{\eps,\kappa}} h_{\eps,\kappa}\psi_{\eps,\kappa,l} \right)^2
\\
&\leq
\left(1+\frac{\Lambda}{c}\right)^2 \left( \int_{C_{\eps,\kappa}} h_{\eps,\kappa}\psi_{\eps,\kappa,l} \right)^2
\end{split}
\end{equation}
for any $l \geq N$.

Since the normalized Dirichlet eigenfunctions of $C_{\eps,\kappa}$
form an orthogonal basis of $L^2(C_{\eps,\kappa})$
we also have that 
\begin{equation} \label{eq_high_freq_2}
\sum_{l=N}^\infty \left(  \int_{C_{\eps,\kappa}} h_{\eps,\kappa} \psi_{\eps,\kappa,l} \right)^2
\leq
\int_{C_{\eps,\kappa}} |h_{\eps,\kappa}|^2
\leq
C \eps \log(1/\eps),
\end{equation}
since $|h_{\eps,\kappa}| \leq C \log(1/\eps)$ if $\eps \leq \eps_6$ thanks to the maximum principle and \cref{lem_pt_bd}, 
where $C=C(\Sigma, k,\kappa_0,\kappa_1,\Lambda)$.
The assertion follows by combining \eqref{eq_high_freq_1} and \eqref{eq_high_freq_2}.
\end{proof}

We now apply this to the first eigenfunction if $\kappa=\kappa_1$.
Recall that $\kappa_1$ is chosen such that $\lambda_0(C_{\eps,\kappa}) \geq \lambda_1(\Sigma) + c$ for some $c>0$.
Since the normalized Dirichlet eigenfunctions on $C_{\eps,\kappa}$ form a basis of $L^2(C_{\eps,\kappa})$, \cref{L^2_lem_cyl} gives control on the entire $L^2$-norm of a $\lambda_1(\Sigma_{\eps,\kappa})$-eigenfunction in this case.

\begin{cor} \label{L^2_cor_cyl}
Denote by $u_{\eps,\kappa_1}$ a normalized
$\lambda_1(\Sigma_{\eps,\kappa_1})$-eigenfunction.
Then we have for $\eps \leq \eps_6$ that
$$
\int_{C_{\eps,\kappa}} |u_{\eps,\kappa_1}|^2 \leq \frac{C}{c^2}\eps \log(1/\eps),
$$
where $C=C(\Sigma,k,\kappa_1)$.
\end{cor}

Before we proceed, we add some more restrictions on $\kappa_0$ and $\kappa_1$.
We may change $\kappa_0$ and $\kappa_1$ so that we still have
$$
\frac{\kappa_0^2}{4} < \lambda_1(\Sigma) < \frac{\kappa_1^2}{4} < \lambda_{K+1}(\Sigma)
$$ 
and such that we can choose $\gamma>0$ with
\begin{equation} \label{eq_extra_assum_kappa}
\lambda_{K+1}(\Sigma) - \frac{\delta_0}{2} \leq \kappa^2/4 + \gamma^2 \kappa^2 \pi^2  \leq \lambda_{K+1}(\Sigma) - \frac{\delta_0}{4}.
\end{equation}
for any $\kappa \in [\kappa_0,\kappa_1]$.

\begin{lemma} \label{L^2_mass_cyl}
Let $u_{\eps,\kappa}$ be a normalized eigenfunction on $\Sigma_{\eps,\kappa}$ with eigenvalue 
$$
\lambda_{\eps,\kappa} \leq \lambda_N(C_{\eps,\kappa}) + (N+1)\eps^{3\alpha/2+1/2-\tau} \leq \lambda_{K+1}(\Sigma) - \delta_0 
$$
for some $0<\tau<(1-\alpha)/2$ and $N \geq 0$.
Then we have that
$$
\int_{C_{\eps,\kappa}} \left| u_{\eps,\kappa}(x) - \sum_{l=0}^N \int_{C_{\eps,\kappa}} u_{\eps,\kappa}(y) \psi_{\eps,\kappa,l}(y)\,  dy\ \psi_{\eps,\kappa,l}(x) \right|^2 dx 
\leq
C \eps^{1-\alpha}
$$
for some uniform constant $C(\Sigma,k,\kappa_0,\kappa_1,\delta_0)>0$ if $\eps \leq \eps_{7}=\eps_{7}(\kappa_0,\kappa_1,\tau) \leq \eps_6$.
\end{lemma}

Note that neither the constant nor $\eps_7$ depend on $N$. 
The dependence on $N$ is instead incorporated in the assumption.

\begin{proof}
Using $\gamma$ as in \eqref{eq_extra_assum_kappa},
in $L^2(C_{\eps,\kappa})$, we can decompose $\left. u_{\eps,\kappa} \right|_{C_{\eps,\kappa}}$ as
$$
u_{\eps,\kappa} =p_{\eps,\kappa} + \sum_{l = N+1}^{\lceil \gamma \eps^{-\alpha} \rceil} n_{\eps,\kappa,l} \psi_{\eps,\kappa,l} + r_{\eps,\kappa},
$$
where $r_{\eps,\kappa}$ consists of Dirichlet eigenfunctions with eigenvalue at least $\kappa^2/4 + (\gamma \eps^{-\alpha})^2 \kappa^2 \pi^2 \eps^{2\alpha} = \kappa^2/4 + \gamma^2 \kappa^2 \pi^2$, $n_{\eps,\kappa,l} \in \IR$, and
$$
p_{\eps,\kappa} =  \sum_{l=0}^N \int_{C_{\eps,\kappa}} u_{\eps,\kappa}(y) \psi_{\eps,\kappa,l}(y)\,  dy\ \psi_{\eps,\kappa,l}.
$$
By multiplying $\psi_{\eps,\kappa,l}$ by $-1$ if necessary, we may assume that $n_{\eps,\kappa,l} \geq 0$. 

By assumption $\lambda_{\eps,\kappa} \leq \lambda_{K+1}(\Sigma) - \delta_0$, which implies that
$ \kappa^2/4 + \gamma^2 \kappa^2 \pi^2 - \lambda_{\eps,\kappa} \geq \delta_0/2$.
Therefore, we find from \cref{L^2_lem_cyl} that
\begin{equation} \label{L^2_exp_mass_higher}
\int_{C_{\eps,\kappa}} |r_{\eps,\kappa}|^2 \leq \frac{C}{\delta_0^2} \eps \log(1/\eps^k)
\end{equation}
if $\eps \leq \eps_1.$

The estimate for the second summand is slightly more involved.
For $l \in [N+1,\lceil \gamma \eps^{-\alpha} \rceil]$ 
we first notice that
\begin{equation} \label{eq_bd_sep}
\begin{split}
\lambda_l(C_{\eps,\kappa}) - \lambda_{\eps,\kappa} 
&\geq
\kappa^2 \pi^2 \left( (N+1+(l-N))^2 \eps^{2 \alpha} - (N+1)^2 \eps^{2\alpha} \right) - (N+1) \eps^{3/2 \alpha +1/2 - \tau}  
\\
& \geq
4c_0 \left( (N+1)(l-N) + (l-N)^2\right) \eps^{2\alpha} - (N+1) \eps^{3\alpha/2 +1/2 - \tau}
\\
&\geq
2c_0 (l-N)l \eps^{2 \alpha}
\\
& \geq c_0 (l-N)(l+1) \eps^{2\alpha}
\end{split}
\end{equation}
for some uniform constant $c_0=c_0(\kappa_0,\kappa_1)$ if $\eps \leq \eps_{7}(\kappa_0,\kappa_1,\tau)$ since $3\alpha/2+1/2 - \tau > 2 \alpha$ and $l \geq N+1$.
For $l$ fixed and if $\eps_{7} \leq \eps_6 $ we can now use \cref{expan_cyl} to find that
\begin{equation} \label{L^2_eq_exp_mass_bd}
\begin{split}
n_{\eps,\kappa,l}
&\leq 
\frac{C |a_{\eps,\kappa,l}|}{\lambda_l(C_{\eps,\kappa})-\lambda_{\eps,\kappa}} + C (l+1) \eps^{(3\alpha +1)/2}
\\
&\leq
\frac{C (l+1)\eps^{(3\alpha+1)/2}}{c_0 (l-N)(l+1) \eps^{2\alpha}} +C (l+1) \eps^{(3 \alpha+1)/2}
\\
&\leq
\frac{C}{c_0 (l-N)} \eps^{(1-\alpha)/2} + C \gamma \eps^{-\alpha} \eps^{(3\alpha+1)/2}
\\
&\leq
\frac{C}{c_0 (l-N)} \eps^{(1-\alpha)/2} + C \gamma \eps^{(1+\alpha)/2}
\\
& \leq
\frac{C}{c_0 (l-N)} \eps^{(1-\alpha)/2}
\end{split}
\end{equation}
where we have used twice that $l+1 \leq \lceil \gamma \eps^{-\alpha} \rceil \leq 2 \gamma \eps^{-\alpha}$ if $\eps \leq \eps_{7}$ (upon decreasing $\eps_{7}$).
We now sum this bound across the relevant scales to obtain
\begin{equation*}
 \sum_{l = N+1}^{\lceil \gamma \eps^{-\alpha} \rceil} n_{\eps,\kappa,l}^2
 \leq
 \frac{4C^2}{c_0^2} \eps^{1-\alpha} \sum_{l = N+1}^{\lceil \gamma \eps^{-\alpha} \rceil} \frac{1}{(l-N)^2}
 \leq C \eps^{1-\alpha}.
\end{equation*}
The lemma follows from this combined with \eqref{L^2_exp_mass_higher}.
\end{proof}

\begin{rem}
It is not clear if the same argument also works for \cref{L^2_lem_cyl}, which would allow us to invoke the much less subtle bound \eqref{eq_int_bd_sobolev} instead of \cref{lem_pt_bd}.
The reason for this is that the Fourier coefficients of non-rotationally symmetric eigenfunctions are given by modified Bessel functions (as functions of $y$) rather than being constant like in the case of a product metric.
\end{rem}

\cref{L^2_mass_cyl} will be particularly useful if $N=0$. This case is recorded explicitly below.

\begin{cor} \label{L^2_ground_mass_cyl}
Let $u_{\eps,\kappa}$ be a normalized eigenfunction on $\Sigma_{\eps,\kappa}$ with eigenvalue $\lambda_{\eps,\kappa} \leq \lambda_0(C_{\eps,\kappa}) + \eps^{3\alpha/2+1/2-\tau}$
for some $0<\tau<(1-\alpha)/2$.  
Then we have
$$
\int_{C_{\eps,\kappa}} \left| u(x) - \int_{C_{\eps,\kappa}} u(y) \psi_{\eps,\kappa,0}(y)\,  dy\ \psi_{\eps,\kappa,0}(x) \right|^2 dx 
\leq
C \eps^{1-\alpha} 
$$
for some uniform constant $C(\Sigma,k,\kappa_0,\kappa_1,\delta_0)>0$ if $\eps \leq \eps_{7}$.
\end{cor}

Using a similar argument we can prove the first part needed to get the preliminary $\eps\log(1/\eps)$ bound.

\begin{lemma} \label{L^2_mean_bd}
Let $u_{\eps,\kappa}$ be a $L^2$-normalized eigenfunction on $\Sigma_{\eps,\kappa}$ with eigenvalue $\lambda_{\eps,\kappa} \leq \lambda_1(C_{\eps,\kappa})$.
There is $C=C(\Sigma,k,\kappa_0,\kappa_1, \delta_0,\Lambda)$ such that
\begin{equation*}
\left| \int_{C_{\eps,\kappa}} u_{\eps,\kappa} \right| \leq C \eps \log(1/\eps)
\end{equation*}
for $\eps \leq \eps_{7}$.
\end{lemma}

\begin{proof}
We decompose $u_{\eps,\kappa}$ restricted to $C_{\eps,\kappa}$ exactly as above.
Since 
$$
\left| \int_{C_{\eps,\kappa}} \psi_{\eps,\kappa,l} \right| \leq \frac{4}{\lambda_l(C_{\eps,\kappa})} |a_{\eps,\kappa,l}| \leq C (l+1) \eps^{(3\alpha+1)/2},
$$
for $l \leq \lceil \gamma \eps^{-\alpha} \rceil$ and $\eps \leq \eps_6$, where we also use \eqref{eq_scale_short},
we find from \eqref{L^2_eq_exp_mass_bd} (in the case $N=0$) that
\begin{equation} \label{eq_pre_bd_1}
 \sum_{l = 1}^{\lceil \gamma \eps^{-\alpha} \rceil} |n_{\eps,\kappa,l}| \left| \int_{C_{\eps,\kappa}} \psi_{\eps,\kappa,l} \right|
 \leq 
 \sum_{l = 1}^{\lceil \gamma \eps^{-\alpha} \rceil}  \frac{4C}{c_0} \eps^{1+\alpha}
 \leq 
 C \eps.
\end{equation}
Similarly, we find that
\begin{equation} \label{eq_pre_bd_2}
|n_{\eps,\kappa,0} | \left| \int_{C_{\eps,\kappa}} \psi_{\eps,\kappa,0} \right|
\leq C \eps^{(3\alpha+1)/2}
\end{equation}
if $\eps \leq \eps_7$.
Finally, note that H{\"o}lder's inequality combined with \eqref{L^2_exp_mass_higher} implies
\begin{equation} \label{eq_pre_bd_4}
\int_{C_{\eps,\kappa}} |r_{\eps,\kappa}| 
\leq 
C \eps^{1/2} \log(1/\eps^k) \area(C_{\eps,\kappa}) 
\leq 
C \eps \log(1/\eps).
\end{equation}
The lemma now follows by combining \eqref{eq_pre_bd_1}-\eqref{eq_pre_bd_4}.
\end{proof}

\begin{rem} \label{rem_asymp_improved}
This also gives an improved estimate for the second term in the asymptotic expansion from \cref{lem_asymp_2}.
\end{rem}

\subsection{Decomposing eigenfunctions on the surface} \label{sec_conc_surf}

Up to this point we tried to understand how the restriction of an eigenfunction on $\Sigma_{\eps,\kappa}$ to $C_{\eps,\kappa}$ looks like.
We will now complement this by a similar discussion for the restriction to $\Sigma \setminus B_{\eps^k}$.
Before we can prove quantitative estimates similar to those above, we need the following qualitative statement.

\begin{lemma} \label{L^2_conc_surf}
Let $u_{\eps,\kappa}$ be a normalized eigenfunction on $\Sigma_{\eps,\kappa}$ with eigenvalue $0<\lambda\leq \lambda_{\eps,\kappa} \leq \Lambda < \lambda_{K+1}(\Sigma)$ for some constants $\lambda,\Lambda$. Then there are normalized $\lambda_1(\Sigma)$-eigenfunctions $\phi_{\eps,\kappa}$ such that
\begin{equation} \label{L^2_eq_surf_eigen}
\int_{\Sigma \setminus B_{\eps^k}} \left| u_{\eps,\kappa}(x)  - \int_{\Sigma \setminus B_{\eps^k}} u_{\eps,\kappa} (y) \phi_{\eps,\kappa}(y) \, dy \, \phi_{\eps,\kappa}(x) \right|^2 dx = o(1)
\end{equation}
as $\eps \to 0$ uniformly in $[\kappa_0,\kappa_1]$ and in all eigenfunctions $u_{\eps,\kappa}$ as above.
\end{lemma}

\begin{proof}
Let $\tilde u_{\eps,\kappa} \in W^{1,2}(\Sigma)$ be the function obtained by extending $u_{\eps,\kappa}$
harmonically to $B_{\eps^k}$.
It follows from \cref{lem_pt_bd} and the maximum principle that 
\begin{equation} \label{eq_l2_compare}
 \|u_{\eps,\kappa}\|_{L^2(\Sigma \setminus B_{\eps^k})} 
\leq 
\|\tilde u_{\eps,\kappa}\|_{L^2(\Sigma) }  
\leq  
\|u_{\eps,\kappa}\|_{L^2(\Sigma \setminus B_{\eps^k})} + C \eps^{k} \log(1/\eps)
\end{equation}
as $\eps \to 0$ uniformly in $\kappa \in [\kappa_0,\kappa_1]$.
In particular, the assertion follows if we can prove \eqref{L^2_eq_surf_eigen} with $u_{\eps,\kappa}$ replaced by $\tilde u_{\eps,\kappa}$.
In order to do so we use that $\tilde u_{\eps,\kappa}$ is uniformly bounded in $W^{1,2}(\Sigma)$ \cite[p.\,40]{rt}.
Therefore we may extract a subsequence $\eps_n \to 0$  and $\kappa_n \to \kappa_*$ such that $\tilde u_{\eps_n,\kappa_n} \to u_*$ weakly in $W^{1,2}(\Sigma)$ and strongly
in $L^2(\Sigma)$ and such that also the corresponding eigenvalues $\lambda_{{\eps_n},{\kappa_n}}$ converge to $\lambda_* \in [\lambda,\Lambda]$.
By weak convergence in $W^{1,2}(\Sigma)$ it is easy to see that $u_*$ solves
\begin{equation} \label{L^2_eq_eigen}
\Delta u_* = \lambda_* u_*
\end{equation}
in $\Sigma \setminus \{x_0,x_1\}$.

Since $u_* \in W^{1,2}(\Sigma)$ and $C_c^\infty (\Sigma \setminus \{x_0,x_1\}) \subset W^{1,2}(\Sigma)$ is dense, we obtain that $u_*$ in fact solves \eqref{L^2_eq_eigen} on all of $\Sigma$.
In particular, by the choice of $\lambda$ and $\Lambda$, we need to have $\lambda_* = \lambda_1(\Sigma)$ (or $u_* = 0$ in which case the assertion trivially holds).
Combined with strong convergence in $L^2(\Sigma)$ and \eqref{eq_l2_compare} this implies that
$$
u_* = \lim_{n \to \infty} \| u_{\eps_n,\kappa_n} \|_{L^2(\Sigma \setminus B_{\eps^k})} \phi
$$
for a normalized $\lambda_1(\Sigma)$-eigenfunction $\phi$.
\end{proof}

It turns out to be useful to have a version of the previous lemma also for a linear combination of eigenfunctions (with a potentially unbounded number of summands).

\begin{lemma} \label{L^2_conc_surf_combi}
Let $w_{\eps,\kappa}$ be an $L^2$-normalized function that is given by a linear combination of at most $M \leq o(1) \eps^{-1/2}$ eigenfunctions $\Sigma_{\eps,\kappa}$ whose eigenvalues are contained in
$[\lambda_1(\Sigma)-\eps^{1/4},\lambda_1(\Sigma)+\eps^{1/4}]$.
Then there are normalized $\lambda_1(\Sigma)$-eigenfunctions $\phi_{\eps,\kappa}$ such that
\begin{equation} \label{L^2_eq_surf_eigen}
\int_{\Sigma \setminus B_{\eps^k}} \left| w_{\eps,\kappa}(x)  - \int_{\Sigma \setminus B_{\eps^k}} w_{\eps,\kappa} (y) \phi_{\eps,\kappa}(y) \, dy \, \phi_{\eps,\kappa}(x) \right|^2 dx = o(1)
\end{equation}
as $\eps \to 0$ uniformly in $[\kappa_0,\kappa_1]$, the $o(1)$-term in the assumption and in all functions $w_{\eps,\kappa}$ as above.
\end{lemma}

The proof is very similar to the argument for \cref{L^2_conc_surf}. 
Our choice of the scale $\eps^{1/4}$ is not very particular, but adopted to our applications.

\begin{proof}
By assumption, $w_{\eps,\kappa}$ is uniformly bounded in $W^{1,2}(\Sigma_{\eps,\kappa})$.
In particular, the argument from \cref{L^2_conc_surf} provides us with a weak limit $w_* \in W^{1,2}(\Sigma)$ and it suffices to prove that
$w_*$ is a $\lambda_1(\Sigma)$-eigenfunction, since the convergence is strong in $L^2(\Sigma)$.
Let us write
$$
w_{\eps,\kappa} = \sum_{i=1}^M \alpha_{\eps,\kappa,i} u_{\eps,\kappa,i},
$$
as a linear combination of eigenfunctions, where also $M$ is allowed to depend on $\eps$ and $\kappa$ as specified in the assumptions.
We denote by $\lambda_{j_i}(\Sigma_{\eps,\kappa})$ the eigenvalue corresponding to $u_{\eps,\kappa,i}$.
Given $\eta \in C_c^{\infty}(\Sigma \setminus \{x_0,x_1\})$, we have for $\eps$ sufficiently small that
\begin{equation*}
\begin{split}
	\left|
		\int_{\Sigma} \nabla \eta \nabla \tilde w_{\eps,\kappa} - \lambda_1(\Sigma) \int_{\Sigma} \eta \tilde w_{\eps,\kappa}
	\right|
	&=
	\left|
		\int_{\Sigma} \nabla \eta \nabla  w_{\eps,\kappa} - \lambda_1(\Sigma) \int_{\Sigma} \eta  w_{\eps,\kappa}
	\right|
	\\
	&=
	\left|
		\int_{\Sigma} \eta \Delta w_{\eps,\kappa} - \lambda_1(\Sigma) \int_{\Sigma} \eta \tilde w_{\eps,\kappa}
	\right|
	\\
	&=
	\left|
		\int_{\Sigma} \eta \sum_{i=1}^M \alpha_{\eps,\kappa,i}(\lambda_{j_i}(\Sigma_{\eps,\kappa}) - \lambda_1(\Sigma))u_{\eps,\kappa,i} 
	\right|
	\\
	& \leq
	C \eps^{1/4} \left( \int_{\supp \eta} \left (\sum_{i=1}^M |\alpha_{\eps,\kappa,i}u_{\eps,\kappa,i} | \right)^2 \right)^{1/2}
	\\
	& \leq
	C \eps^{1/4} \left( M \sum_{i=1}^M |\alpha_{\eps,\kappa,i}|^2 \right)^{1/2}
	\\
	& \leq o(1) \|w_{\eps,\kappa}\|_{L^2(\Sigma_{\eps,\kappa})},
\end{split}
\end{equation*}
where we have used Young's inequality.
From here we can conclude as in \cref{L^2_conc_surf}.
\end{proof}
\subsection{A preliminary estimate for the rate of convergence} \label{sec_conc_pre_est}

Combining \cref{L^2_mean_bd} and \cref{L^2_conc_surf}, we can now prove the $\eps \log(1/\eps)$-bound for $\lambda_2(\Sigma_{\eps,\kappa})$.

\begin{lemma} \label{L^2_lem_pre_bd}
Assume that $0<\tau < (1-\alpha)/2$.
There is a constant $d_0=d_0(\Sigma,k,\kappa_0,\kappa_1, \delta_0)$ with the following property.
Let $u_{\eps,\kappa}$ be a normalized $\lambda_1(\Sigma_{\eps,\kappa})$-eigenfunction.
Assume that 
\begin{equation} \label{assump_conc_surf}
\left( \int_{C_{\eps,\kappa}} |u_{\eps,\kappa}|^2 \right)^{1/2} \leq 1- \eps^{\tau}.
\end{equation}
Then we have 
$$
\lambda_2(\Sigma_{\eps,\kappa}) \geq \lambda_1(\Sigma) - d_0 \eps \log(1/\eps)
$$
if $\eps \leq \eps_8(\Sigma,k,\kappa_0,\kappa_1,\delta_0,\tau)$.
\end{lemma}

This estimate will later force us to work only with parameters $\alpha<1/2$,
since this implies that the gap ($\sim \eps^{2\alpha}$) of two consecutive Dirichlet eigenvalues on $C_{\eps,\kappa}$
is strictly larger than the bound we start with (which is $\eps \log(1/\eps)$).

\begin{proof}
If we have in addition that 
$$
\int_{\Sigma \setminus B_{\eps^k}} |u_{\eps,\kappa}|^2 \geq 1/2
$$ 
the assertion follows immediately from \cref{expan_surf} combined with \cref{L^2_mean_bd} (which is applicable thanks to \cref{first_upper_bd}) and \cref{L^2_conc_surf}.
Therefore, we may assume that we also have 
\begin{equation} \label{eq_extra_assum_1}
\int_{\Sigma \setminus B_{\eps^k}} |u_{\eps,\kappa}|^2 \leq 1/2.
\end{equation}
This in turn implies by \cref{first_upper_bd} combined with \cref{L^2_ground_mass_cyl} that
\begin{equation} \label{eq_first_cyl}
\left| \int_{C_{\eps,\kappa}} u_{\eps,\kappa} \psi_{\eps,\kappa,0} \right| \geq 1/4
\end{equation}
for $\eps$ sufficiently small.

From the assumption \eqref{assump_conc_surf} we find that
\begin{equation*}
\begin{split}
\left|\int_{\Sigma_{\eps,\kappa}} u_{\eps,\kappa} \tilde \psi_{\eps,\kappa,0} \right|
&\leq 
C \int_{C_{\eps,\kappa}} |u_{\eps,\kappa} \psi_{\eps,\kappa,0}| 
+C \eps^{3\alpha/2+1} \log(1/\eps^k)
+ C \eps^{3\alpha/2+1/2} \int_{\Sigma \setminus B_{\eps^k}} |u_{\eps,\kappa}|
\\
&\leq 
1-\eps^{\tau} + C \eps^{3 \alpha /2 + 1/2}
\\
&\leq 
1- \eps^{\tau}/2
\end{split}
\end{equation*}
for $\eps$ sufficiently small (depending on nothing apart from $C$).
It then follows from \cref{lem_anne_quasimod} applied to the quasimode $\tilde \psi_{\eps,\kappa,0}$ (\cref{lem_asymp_2}) that the second eigenvalue satisfies
$$
\lambda_2(\Sigma_{\eps,\kappa}) \leq \lambda_0(C_{\eps,\kappa}) + \eps^{\tau'} \leq \lambda_1(C_{\eps,\kappa})
$$
for some $\tau < \tau' < (1-\alpha)/2$ if $\eps \leq \eps(\tau')$.

Therefore, we obtain from \cref{L^2_ground_mass_cyl} combined with \eqref{eq_extra_assum_1} and \eqref{eq_first_cyl} that the corresponding normalized eigenfunction $v_{\eps,\kappa}$ has
$$
\int_{\Sigma \setminus B_\eps} |v_{\eps,\kappa}|^2 \geq 1/8
$$
for $\eps$ sufficiently small.
In particular, it follows from \cref{L^2_conc_surf} that for $\eps$ sufficiently small we can find a $\lambda_1(\Sigma)$-eigenfunction $\phi_{\eps,\kappa}$
such that
$$
\int_{\Sigma \setminus B_{\eps,\kappa}}  \phi_{\eps,\kappa} v_{\eps,\kappa} \geq 1/16.
$$
Let $\tilde \phi_{\eps,\kappa}$ denote the quasimode constructed from $\phi_{\eps,\kappa}$ in \cref{subsec_quasimod_cyl}.
We then obtain from \cref{expan_surf} combined with \cref{L^2_mean_bd} (which improves the leading order term in the expansion)
that
\begin{equation*}
\lambda_2(\Sigma_{\eps,\kappa}) \geq \lambda_1(\Sigma) - d \int_{C_{\eps,\kappa}} v \geq \lambda_1(\Sigma) - d_0 \eps \log(1/\eps).
\qedhere
\end{equation*}
\end{proof}

\subsection{Decomposing eigenfunctions on the surface (continued)} \label{sec_conc_surf_cont}

Under additional assumptions on $\lambda_2(\Sigma_{\eps,\kappa})$ we can obtain a quantitative version of \cref{L^2_conc_surf} that allows us to 
locate all the eigenfunctions associated to $\lambda_1(\Sigma)$-eigenfunctions with a definite rate.
Recall that $\tilde \phi_{\eps,\kappa,0}$ is the extension of the (up to scaling) unique $\lambda_1(\Sigma)$-eigenfunction that does not vanish in $x_0$.

\begin{lemma} \label{L^2_surface_part}
Let $\alpha \leq \alpha_0 < 9/16$,  $0<\tau<(1-\alpha)/2, k \geq 9$, and $0<\delta_1 \leq 1/4$.
Moreover, assume that $\lambda_2(\Sigma_{\eps,\kappa}) \leq \min\{\lambda_1(\Sigma) - \eps^{k/4}, \lambda_0(C_{\eps,\kappa}) + \eps^{3 \alpha / 2+1/2-\tau} \}$.
If $u_{\eps,\kappa}$ is a normalized $\lambda_1(\Sigma_{\eps,\kappa})$-eigenfunction
and 
$v_{\eps,\kappa}$ is a normalized $\lambda_2(\Sigma_{\eps,\kappa})$-eigenfunction,
we have that
\begin{equation*}
\begin{split}
\int_{\Sigma_{\eps,\kappa}} \left|\tilde \phi_{\eps,\kappa,0}(x) - \int_{\Sigma_{\eps,\kappa}} u_{\eps,\kappa}(y) \tilde \phi_{\eps,\kappa,0}(y)\, dy\, u_{\eps,\kappa}(x) - \int_{\Sigma_{\eps,\kappa}} v_{\eps,\kappa}(y) \tilde \phi_{\eps,\kappa,0}(y)\, dy \, v_{\eps,\kappa}(x) \  \right|^2 \, dx
\\
\leq
C( \eps^{2 \delta_1} + \eps^{1/2-\delta_1 + 2(\tau - \alpha)})
\end{split}
\end{equation*}
if $\eps \leq \eps_{9}(\Sigma,k,\kappa_0,\kappa_1,\tau,\delta_1, \alpha_0)$ and where $C=C(\Sigma,k,\kappa_0,\kappa_1,\delta_0,\delta_1)$.
\end{lemma}

\begin{rem}
Note that $\tau - \alpha$ is very small if we take $\alpha$ close to $1/3$ and $\tau$ close to $(1-\alpha)/2$.
\end{rem}

\begin{proof}
We proceed in four steps.

\smallskip

\textsc{Step 1:} 
\textit{Bounding $\lambda_1(\Sigma) - \lambda_0(C_{\eps,\kappa})$}
\smallskip

This is very similar to the argument for \cref{L^2_lem_pre_bd}.
Since we assume for the second eigenvalue that
$
\lambda_2(\Sigma_{\eps,\kappa}) \leq \lambda_0(C_{\eps,\kappa}) + \eps^{3\alpha/2+1/2-\tau}
$
with $\tau < (1-\alpha/2)$,
it follows from \cref{L^2_mass_cyl} that the restrictions of the first two eigenfunctions to $C_{\eps,\kappa}$ are up to a uniform $o(1)$-term (in $L^2$) given by  $\psi_{\eps,\kappa,0}$ up to scaling.
Similarly, since $\lambda_2(\Sigma_{\eps,\kappa}) \leq \lambda_1(\Sigma) - \eps^{k/4}$, we find from \cref{L^2_conc_surf} combined with \cref{expan_surf},
that the restrictions to $\Sigma \setminus B_{\eps^k}$ are up to a uniform $o(1)$-term (in $L^2$) given by $\phi_0$ up to scaling.
In particular, there is $w_{\eps,\kappa}$ which is a first or a second eigenfunction, such that
$$
\int_{\Sigma_{\eps,\kappa}} |w_{\eps,\kappa} \tilde \phi_{\eps,\kappa,0}| \geq 1/4
$$ 
if $\eps$ is sufficiently small.
Therefore, the assumptions, combined with \cref{expan_surf} once again imply that
$$
\lambda_1(\Sigma) - C \eps^{1/2} \leq \lambda_2(\Sigma_{\eps,\kappa}) \leq \lambda_0(C_{\eps,\kappa}) + \eps^{3\alpha/2+1/2-\tau} 
\leq \lambda_0(C_{\eps,\kappa}) + C \eps^{1/2}.
$$
Analogously, by testing against $\tilde \psi_{\eps,\kappa,0}$, and we obtain that
$$
\lambda_0(C_{\eps,\kappa}) \leq \lambda_1(\Sigma) + C \eps^{3/2\alpha+1/2} \leq \lambda_1(\Sigma) + C \eps^{1/2}.
$$

\textsc{Step 2:} 
\textit{Bounding the number of relevant eigenvalues}
\smallskip

Let $N \in \IN$ be maximal such that 
\begin{equation} \label{eq_restr_N}
\lambda_N(C_{\eps,\kappa})+(N+1)\eps^{3\alpha/2+1/2 - \tau} \leq \lambda_{1}(\Sigma) + \eps^{1/4}.
\end{equation}
Note that \eqref{eq_restr_N} and the first step imply that
\begin{equation} \label{eq_restr_N_2}
C_0^{-1}\eps^{1/8-\alpha} \leq N \leq C_0 \eps^{1/8-\alpha}
\end{equation}
for a fixed constant $C_0=C_0(\kappa_0,\kappa_1,\tau)$.

Let now $w_{\eps,\kappa}$ be a linear combination of $M \leq 2C_0\eps^{1/8-\alpha}$ eigenfunctions with eigenvalues
\begin{equation} \label{eq_restr_ev}
\lambda_{\eps,\kappa} \in (\lambda_1(\Sigma) - \eps^{1/4}, \lambda_N(C_{\eps,\kappa})+(N+1)\eps^{3\alpha/2+1/2 - \tau}).
\end{equation}
Thanks to our assumption on $\alpha$ and our choice of $M$. 
it follows from \cref{L^2_conc_surf_combi} that $\left.w_{\eps,\kappa} \right|_\Sigma$ is uniformly close in $L^2(\Sigma)$ to a linear combination of $\lambda_1(\Sigma)$-eigenfunctions for $\eps$ sufficiently small depending only on $\Sigma,k,\kappa_0,\kappa_1,\alpha_0$.

Next, we write 
$$
w_{\eps,\kappa} = \sum_{i=1}^M \alpha_{\eps,\kappa,i} u_{\eps,\kappa,i}
$$
as a linear combination of eigenfunctions. 
Let us write $p_{\eps,\kappa,i}$ for the projection of $u_{\eps,\kappa,i}$ considered in \cref{L^2_mass_cyl}.
Then we can apply \cref{L^2_mass_cyl} to each $u_{\eps,\kappa,i}$ (scaled by $a_{\eps,\kappa,i}$) and find combined with Minkowski's and Young's inequality that
\begin{equation*}
\begin{split}
\left\| w_{\eps,\kappa} - \sum_{i=1}^M \alpha_{\eps,\kappa,i} p_{\eps,\kappa,i} \right\|_{L^2(C_{\eps,\kappa})}^2
&\leq 
C \left(\sum_{i=1}^M |\alpha_{\eps,\kappa,i}| \right)^2 \eps^{1-\alpha} 
\leq 
C M \sum_{i=1}^M \alpha_{\eps,\kappa,i}^2 \eps^{1-\alpha}
\\
&\leq 
C M \eps^{1-\alpha}
\leq 
C \eps^{9/8-2 \alpha} \to 0
\end{split}
\end{equation*}
by assumption.
Therefore, $\left. w_{\eps,\kappa}\right|_{C_{\eps,\kappa}}$ is uniformly close in $L^2(C_{\eps,\kappa})$ to a linear combination of the first $(N+1)$-Dirichlet eigenfunctions on $C_{\eps,\kappa}$ for $\eps$ sufficiently small, now
also depending on $\tau$.
Thus, we have proved that any space of linear combinations of eigenfunctions with eigenvalue as in \eqref{eq_restr_ev} and dimension at most $M$ can in fact have dimension at most $(N+1) + K$ for $\eps$ sufficiently small.
(Recall that we denote by $K=\mult(\lambda_1(\Sigma))$ the multiplicity of the first eigenvalue on $\Sigma$).
Since
$$
N+K+1 \leq \frac{3}{2}C_0 \eps^{1/8-\alpha} \leq 2 C_0 \eps^{1/8-\alpha} -1
$$
for $\eps$ sufficiently small thanks to \eqref{eq_restr_N_2}, we find that the space spanned by \emph{all} eigenfunctions as above can have dimension at most $N+1+K$.

\smallskip

\textsc{Step 3:} 
\textit{Approximating the corresponding eigenfunctions}
\smallskip

Since $\lambda_2(\Sigma_{\eps,\kappa}) \leq \min( \lambda_0(C_{\eps,\kappa}) +  \eps^{3\alpha/2+1/2-\tau},\lambda_1(\Sigma) - \eps^{k/4})$, for $\eps$ sufficiently small, it follows from the second step that there are at most
$N + K -1$ eigenfunctions with eigenvalue 
\begin{equation} \label{eq_ev_range}
\lambda_{\eps,\kappa} \in (\min(\lambda_1(\Sigma) - \eps^{k/4}, \lambda_0(C_{\eps,\kappa}) + \eps^{3\alpha/2+1/2-\tau} ),
\lambda_N(C_{\eps,\kappa}) + (N+1) \eps^{3\alpha/2 + 1/2 -\tau}),
\end{equation}
where we still assume \eqref{eq_restr_N}.

On the other hand, we have the quasimodes $\tilde \phi_{\eps,\kappa,1}, \dots , \tilde \phi_{\eps,\kappa,K-1}$ from \cref{mono_quasimode_1} 
and the quasimodes
$\tilde \psi_{\eps,\kappa,1},\dots, \tilde \psi_{\eps,\kappa,N}$ from \cref{lem_asymp_2} whose approximate eigenvalues are in the very same range.
For $l \in [1,N]$ let 
$$
I_{\eps,\kappa,l} = (\lambda_l(C_{\eps,\kappa}) - (l+1)\eps^{3\alpha/2+1/2-\tau} , \lambda_l(C_{\eps,\kappa}) + (l+1)\eps^{3\alpha/2+1/2-\tau}).
$$

Note that by the choice of $\tau$, for $\eps$ sufficiently small (depending only on $\tau,\kappa_0,\kappa_1$), the intervals $I_{\eps,\kappa,l}$
are pairwise disjoint (cf.\ computation \eqref{eq_bd_sep} in the proof of \cref{L^2_mass_cyl}).
Moreover, there is some $l_* \in [1,N]$ such that
$
I_{\eps,\kappa,l} \cap (\lambda_1(\Sigma) - \eps^{k/4}/2,\lambda_1(\Sigma) + \eps^{k/4}/2) = \emptyset
$
if $l \neq l_*$.
It follows from \cref{lem_anne_quasimod} combined with \cref{mono_quasimode_2} and \cref{lem_asymp_1} 
that each of the intervals $I_{\eps,\kappa,l}$ contains at least one eigenvalue if $l \neq l_*$ and $I_{\eps,\kappa,l_*} \cup (\lambda_1(\Sigma) - \eps^{k/4}/2 , \lambda_1(\Sigma) + \eps^{k/4}/2)$ contains at least $K$ eigenvalues.
Therefore, invoking the previous step, and since the intervals $I_{\eps,\kappa,l}$ are disjoint, we find that each $I_{\eps,\kappa,l}$ for $l \neq l_*$ contains exactly one eigenvalue and 
$I_{\eps,\kappa,l_*} \cup (\lambda_1(\Sigma) - \eps^{k/4}/2 , \lambda_1(\Sigma) + \eps^{k/4}/2)$ contains exactly $K$ eigenvalues.
Using  \cref{lem_anne_quasimod} combined with \cref{mono_quasimode_2} and \cref{lem_asymp_1} once again, we find that
for any eigenfunction $w_{\eps,\kappa}$  (not a linear combination as above) with eigenvalue in $I_{\eps,\kappa,l}$, $l \neq l_*$ the corresponding quasimode $q_{\eps,\kappa}$ satisfies
\begin{align} \label{eq_approx_eigen}
\| w_{\eps,\kappa} - q_{\eps,\kappa} \|_{L^2(\Sigma_{\eps,\kappa})} \leq C \eps^{\tau} 
\end{align}
for a uniform constant $C=C(\Sigma,k,\kappa_0,\kappa_1,\delta_0)$ if $\eps$ is sufficiently small.
The very same assertion holds for $I_{\eps,\kappa,l_*} \cup (\lambda_1(\Sigma) - \eps^{k/4}/2 , \lambda_1(\Sigma) + \eps^{k/4}/2)$. In this case we need to apply the Gram--Schmidt process\footnote{This will make the estimate worse by a constant depending only on the dimension of the space, which is bounded by $K$ by construction.}
 first and $q_{\eps,\kappa}$ can be a linear combination of the corresponding quasimodes.

Therefore, it follows that for any such eigenfunction $w_{\eps,\kappa}$ we have that
\begin{equation} \label{eq_contribution_cyl}
\begin{split}
\left| \int_{\Sigma_{\eps,\kappa}} w_{\eps,\kappa} \tilde \phi_{\eps,\kappa,0} \right| 
&\leq 
C \| w_{\eps,\kappa} - q_{\eps,\kappa} \|_{L^2(\Sigma_{\eps,\kappa})}
+\int_{\Sigma \setminus B_{\eps^k}} |q_{\eps,\kappa} \tilde \phi_{\eps,\kappa,0}|
+ \int_{C_{\eps,\kappa}}  |q_{\eps,\kappa} \tilde \phi_{\eps,\kappa,0} |
\\
&\leq C \eps^{\tau} + C(N+1)\eps^{3\alpha/2+1/2} + C \eps^{1/2} + C(N+1)\eps^{3(\alpha+1)/2} \log(1/\eps^k) 
\\
&\leq C \eps^{\tau}
\end{split}
\end{equation}
since the assumption \eqref{eq_restr_N} in particular implies that $(N+1) \eps^{3\alpha/2+1/2-\tau} \leq C$, and also $\tau \leq 1/2$ by assumption.

\smallskip

\textsc{Step 4:} 
\textit{Conclusion}
\smallskip

We now apply \cref{lem_anne_quasimod} with $\delta=\eps^{1/2}$ and $s=1/2\eps^{1/2-\delta_1}$ for some $\delta_1 \in (0,1/4]$ to the quasimode from \cref{mono_quasimode_2}.
By the first two steps, 
there are at most $C\eps^{1/4-\delta_1/2 -  \alpha}$ eigenvalues in $J_s=[\lambda_1(\Sigma) - s , \lambda_1(\Sigma) + s ]$.
By  \cref{lem_anne_quasimod} the spectral projection of $\tilde \phi_{\eps,\kappa,0}$ to the complement of $J_s$ has $L^2$-norm at most $\eps^{\delta_1}$.
If we apply the third step we find from \eqref{eq_contribution_cyl}
that each but the first two eigenfunctions contributes at most $C \eps^{2\tau}$ to $\|\tilde \phi_{\eps,\kappa,0} \|_{L^2}^2$ so that their total contribution is at most
$C \eps^{2\tau +1/4-\delta_1/2 -  \alpha }  \leq C \eps^{1/4+\alpha - \delta_1 + 2(\tau-\alpha)} \leq  \eps^{1/2 - \delta_1 + 2(\tau-\alpha)}.$
\end{proof}

The second and third step from the argument above in particular imply the following.

\begin{lemma} \label{L^2_lem_mult}
For $k \geq 9$, $\alpha<9/16$ and $\lambda_2(\Sigma_{\eps,\kappa}) \leq \min\{\lambda_1(\Sigma) - \eps^{k/4}, \lambda_0(C_{\eps,\kappa}) + \eps^{3\alpha/2+1/2-\tau}) \}$, we have 
$$
\lambda_3(\Sigma_{\eps,\kappa}) > \lambda_2(\Sigma_{\eps,\kappa})
$$
for $\eps \leq \eps_9$.
\end{lemma}

\subsection{Decomposing the first two eigenfunctions} \label{sec_conc_decomp_two}

We combine all our efforts of this and the preceding section to obtain a good quantitativ $L^2$-picture of the first two eigenfunctions 
assuming that the first eigenfunction is not too close to $\lambda_1(\Sigma)$
and that some interaction between the spectrum of $\Sigma$ and the spectrum of $C_{\eps,\kappa}$ is observable.

Starting at this point we assume that we have
$$
\phi_0(x_0) \neq 0.
$$
If this does not hold \cref{thm_main_technical} follows from \cref{thm_eig_lim} combined with \cref{lem_anne_quasimod} and \cref{expan_surf} in a straightforward way, see also \cite{MS2}.
In particular, we may normalize the sign of $\phi_0$ by requiring that
$$
\phi_0(x_0)>0.
$$
We also fix the sign of the $\lambda_0(C_{\eps,\kappa})$-eigenfunction $\psi_{\eps,\kappa,0}$ such that $\psi_{\eps,\kappa,0} \geq 0$,
so that in particular
$$
a_{\eps,\kappa,0}>0.
$$

By  \cref{L^2_lem_mult}, if $\lambda_2(\Sigma_{\eps,\kappa}) \leq \min\{  \lambda_1(\Sigma) - \eps^{k/4} , \lambda_0(C_{\eps,\kappa}) + \eps^{3\alpha/2+1/2 -\tau} \}$, the space spanned by the $\lambda_1(\Sigma_{\eps,\kappa})$- and $\lambda_2(\Sigma_{\eps,\kappa})$-eigenfunctions
is two dimensional.
Since the first eigenvalue is simple thanks to \cref{prop_first_simple}, it follows that the first two eigenvalues are both simple.

We now choose a normalized $\lambda_1(\Sigma_{\eps,\kappa})$-eigenfunction $u_{\eps,\kappa}$ and 
a $\lambda_2(\Sigma_{\eps,\kappa})$-eigenfunction $v_{\eps,\kappa}$.
Slightly conflicting with our earlier notation (which we will not use again), we write
\begin{equation} \label{eq_def_mass}
\begin{split}
m_{\eps,\kappa,1}&:= \int_{\Sigma \setminus B_{\eps^k}} u_{\eps,\kappa} \phi_0,
\ \ 
n_{\eps,\kappa,1}:=\int_{C_{\eps,\kappa}} u_{\eps,\kappa} \psi_{\eps,\kappa,0},
\\
m_{\eps,\kappa,2}& :=  \int_{\Sigma \setminus B_{\eps^k}} v_{\eps,\kappa} \phi_0,
\ \
n_{\eps,\kappa,2}:=\int_{C_{\eps,\kappa}} v_{\eps,\kappa} \psi_{\eps,\kappa,0}.
\end{split}
\end{equation}
By multiplying $u_{\eps,\kappa}$ and $v_{\eps,\kappa}$ by $-1$ if necessary, we can and will assume from here on that 
$n_{\eps,\kappa,1},n_{\eps,\kappa,2} \geq 0$.

In order to use the asymptotic expansions given in \cref{sec_quasimodes} it will be helpful to understand the numbers $m_{\eps,\kappa,i}$ and $n_{\eps,\kappa,i}$ using the results from this section.

\begin{prop}  \label{L^2_main_decomp}
Assume that $\alpha \leq \alpha_0<9/16$,
$\lambda_2(\Sigma_{\eps,\kappa}) \leq \min\{ \lambda_0(C_{\eps,\kappa}) + \eps^{3 \alpha / 2+1/2-\tau}, \lambda_1(\Sigma) - \eps^{k/4}\}$ 
for some $0<\tau<(1-\alpha)/2$ and $k \geq 9$.
Then we have 
\begin{equation} \label{eq_L^2_decomp_1}
m_{\eps,\kappa,i}^2+n_{\eps,\kappa,i}^2 = 1 - O(\eps^\tau) - O(\eps^{1/4+2(\tau-\alpha)})
\end{equation}
and 
\begin{equation} \label{eq_L^2_decomp_2}
m_{\eps,\kappa,1} n_{\eps,\kappa,1}+m_{\eps,\kappa,2}n_{\eps,\kappa,2}=O(\eps^{\tau}).
\end{equation}
and
\begin{equation} \label{eq_L^2_decomp_3}
m_{\eps,\kappa,1}=n_{\eps,\kappa,2} + O(\eps^{1/8+(\tau-\alpha)}) + O(\eps^{\tau})
\end{equation}
and
\begin{equation} \label{eq_L^2_decomp_4}
m_{\eps,\kappa,2}=-n_{\eps,\kappa,1} + O(\eps^{1/8+(\tau-\alpha)}) + O(\eps^{\tau}),
\end{equation}
where the remainder terms are uniform $\alpha$ and in $\kappa \in [\kappa_0,\kappa_1]$ as long as $\alpha_0$ and $\tau$ are fixed.
\end{prop}

\begin{proof}
We can decompose the quasimode $\tilde \phi_{\eps,\kappa,0}$ as 
$$
\tilde \phi_{\eps,\kappa,0} = \tilde m_{\eps,\kappa,1} u_{\eps,\kappa} + \tilde m_{\eps,\kappa,2} v_{\eps,\kappa} + r_{\eps,\kappa},
$$
where $r_{\eps,\kappa}$ is orthogonal to $u_{\eps,\kappa}$ and $v_{\eps,\kappa}$. 
Moreover, we clearly have 
$$
|\tilde m_{\eps,\kappa,i} - m_{\eps,\kappa,i}| \leq C \eps^{1/2}
$$ 
for $i=1,2$ and thus \cref{L^2_surface_part} (with the specific choice $\delta_1=1/4$) 
implies 
\begin{equation} \label{eq_decomp_surf}
m_{\eps,\kappa,1}^2+m_{\eps,\kappa,2}^2 = 1 - O(\eps^{1/4+2(\tau-\alpha)}).
\end{equation}
Similarly, we can decompose the quasimode $\tilde \psi_{\eps,\kappa,0}$ as
$$
\tilde \psi_{\eps,\kappa,0} = \tilde n_{\eps,\kappa,1} u_{\eps,\kappa} + \tilde n_{\eps,\kappa,2} v_{\eps,\kappa} + s_{\eps,\kappa} + t_{\eps,\kappa},
$$
where $s_{\eps,\kappa}$ is the spectral projection onto the sum of eigenspaces with eigenvalues
in $[\lambda_3(\Sigma_{\eps,\kappa}), \lambda_0(C_{\eps,\kappa}) + 2 \eps^{3\alpha/2+1/2-\tau}]$ 
and $t_{\eps,\kappa}$  is orthogonal to the first three summands.
In this case we clearly have
$$
|\tilde n_{\eps,\kappa,i} - n_{\eps,\kappa,i}| \leq C \eps^{3\alpha/2+1/2}.
$$
It follows from \cref{lem_anne_quasimod} applied to \cref{lem_asymp_2} that
$$
\|t_{\eps,\kappa}\|_{L^2(\Sigma_{\eps,\kappa})}^2 \leq C \eps^{2 \tau}.
$$
Arguing exactly as in the proof of \cref{L^2_surface_part}, we find also 
$$
\|s_{\eps,\kappa}\|_{L^2(\Sigma_{\eps,\kappa})}^2 \leq C \eps^{2 \tau}.
$$
(We have to take into account exactly those eigenfunctions very close to the $\lambda_1(\Sigma)$-eigenfunctions that vanish at $x_0$.)
Combining the last three estimates  we arrive at

\begin{equation} \label{eq_decomp_cyl}
n_{\eps,\kappa,1}^2 + n_{\eps,\kappa,2}^2 = 1 - O(\eps^{2 \tau}).
\end{equation}

Note that \eqref{eq_decomp_surf} and \eqref{eq_decomp_cyl} combined with the trivial estimate
$
m_{\eps,\kappa,2}^2+n_{\eps,\kappa,2}^2 \leq 1
$
implies that
$$
m_{\eps,\kappa,1}^2+n_{\eps,\kappa,1}^2 
\geq
m_{\eps,\kappa,1}^2+n_{\eps,\kappa,1}^2 + m_{\eps,\kappa,2}^2 + n_{\eps,\kappa,2}^2 -1
\geq
1 -O(\eps^{2\tau}) + O(\eps^{1/4-2(\tau-\alpha)}).
$$
This in turn, combined with $m_{\eps,\kappa,1}^2 + n_{\eps,\kappa,1}^2 \leq 1$ gives that
$$
m_{\eps,\kappa,1}^2+n_{\eps,\kappa,1}^2 = 1 -O(\eps^{2\tau}) + O(\eps^{1/4-2(\tau-\alpha)}).
$$
The exact same computation also implies that
$$
m_{\eps,\kappa,2}^2+n_{\eps,\kappa,2}^2 = 1 -O(\eps^{2\tau}) + O(\eps^{1/4-2(\tau-\alpha)}),
$$
which already proves \eqref{eq_L^2_decomp_1}.
The remaining assertions can be obtained as follows.
First, note that $\|\tilde \psi_{\eps,\kappa,0} - n_{\eps,\kappa,1} u_{\eps,\kappa} - n_{\eps,\kappa,2} v_{\eps,\kappa} \|_{L^2(\Sigma_{\eps,\kappa})} \leq C \eps^{\tau}$ in particular implies that
\begin{equation*}
|m_{\eps,\kappa,1} n_{\eps,\kappa,1} + m_{\eps,\kappa,2} n_{\eps,\kappa,2}| =
\left| \int_{\Sigma \setminus B_{\eps^k}} (n_{\eps,\kappa,1} u_{\eps,\kappa} + n_{\eps,\kappa,2} v_{\eps,\kappa}) \phi_0 \right|
\leq C \eps^{\tau}.
\end{equation*}
Moreover, $m_{\eps,\kappa,1}$ and $n_{\eps,\kappa,1}$ need to have the same sign whenever $|m_{\eps,\kappa,1}| \geq \eps^{1/2}$ and $|n_{\eps,\kappa,1}| \geq \eps^{3\alpha/2+1/4}$ thanks to \cref{expan_cyl} combined with \cref{first_upper_bd}:
Under these assumptions on $m_{\eps,\kappa,1}$ and $n_{\eps,\kappa,1}$, the sign of 
$\lambda_0(C_{\eps,\kappa})-\lambda_1(\Sigma_{\eps,\kappa})$ coincides with the sign of $m_{\eps,\kappa,1}/n_{\eps,\kappa,1}$.
 But we know that the former is non-negative once $\eps$ is sufficiently small.
 \end{proof}

\begin{rem}
This is the first time that we actually crucially make use of our precise knowledge of the form of the leading order term in the asymptotic expansion \cref{expan_cyl}.
In our preceding estimates we could also have worked with Dirichlet eigenfunctions of $C_{\eps,\kappa}$ for the expense of paying by a multiplicative $\log(1/\eps)$-term.
\end{rem}

\cref{L^2_main_decomp} will be a key technical tool, but it assumes that $\lambda_2(\Sigma_{\eps,\kappa}) \leq \lambda_1(\Sigma) - \eps^{k/4}$.
We can not ensure to have this in any step of our iteration argument in \cref{sec_iter}.
Instead, there is a much simpler estimate that handles this case. 
It will be useful to state it in a rather general form.

\begin{lemma} \label{L^2_lem_lambda_2_close}
Assume that
$\lambda_1(\Sigma_{\eps,\kappa}) \leq \lambda_1(\Sigma) - \eps^{k/4}$
and $|m_{\eps,\kappa,1}| \geq  \eps^{1/8}$.
Then
$$
\lambda_0(C_{\eps,\kappa}) \geq \lambda_2(\Sigma_{\eps,\kappa})- \frac{\phi_0(x_0) \lambda_0(C_{\eps,\kappa})}{\lambda_1(\Sigma)}\frac{n_{\eps,\kappa,1}}{m_{\eps,\kappa,1}} a_{\eps,\kappa,0} + \frac{O(\eps^{3\alpha/2+3/4})}{m_{\eps,\kappa,1}^2},
$$
and the remainder term depends only on $\Sigma,k,\delta_0,\kappa_0,\kappa_1$.
\end{lemma}

Note that the assumption in particular implies that the first eigenvalue is simple thanks to \cref{prop_first_simple}.

\begin{proof}
We can decompose 
$$
\tilde \psi_{\eps,\kappa,0} = \tilde n_{\eps,\kappa,1} u_{\eps,\kappa} + r_{\eps,\kappa},
$$
where $|\tilde n_{\eps,\kappa,1} - n_{\eps,\kappa,1}| \leq C \eps^{3 \alpha/2 +1/2}$ and $r_{\eps,\kappa}$ is the spectral projection of $\tilde \psi_{\eps,\kappa,0}$
to the sum of all eigenspaces with eigenvalues at least $\lambda_2(\Sigma_{\eps,\kappa}).$
We then have 
\begin{equation} \label{eq_est_0}
\int_{\Sigma_{\eps,\kappa}} \tilde \psi_{\eps,\kappa,0} r_{\eps,\kappa}=
\int_{\Sigma_{\eps,\kappa}} |r_{\eps,\kappa}|^2,
\end{equation}
and also
$$
\|r_{\eps,\kappa}\|_{L^2(\Sigma_{\eps,\kappa})}^2 + \tilde n_{\eps,\kappa,1}^2 = \| \tilde \psi_{\eps,\kappa,0} \|_{L^2(\Sigma_{\eps,\kappa})}^2 = 1 + O(\eps^{3 \alpha /2 +1/2}).
$$
This implies 
\begin{equation} \label{eq_est_1}
\|r_{\eps,\kappa}\|_{L^2(\Sigma_{\eps,\kappa})}^2  = 1 -  n_{\eps,\kappa,1}^2 - O(\eps^{3 \alpha /2 +1/2})
\geq m_{\eps,\kappa,1}^2 - O(\eps^{3 \alpha/2+1/2}),
\end{equation}
since $m_{\eps,\kappa,1}^2+n_{\eps,\kappa,1}^2 \leq 1$.
Similarly, we find 
\begin{equation} \label{eq_est_2}
 \int_{\Sigma \setminus B_{\eps^k}} \phi_0 r_{\eps,\kappa} 
= -  n_{\eps,\kappa,1}m_{\eps,\kappa,1} + O(\eps^{3 \alpha/2 +1/2}).
\end{equation}
A short computation implies 
\begin{equation} \label{eq_est_3}
\int_{\Sigma_{\eps,\kappa}} \nabla \tilde \psi_{\eps,\kappa,0} \nabla r_{\eps,\kappa}
\geq
\lambda_2(\Sigma_{\eps,\kappa}) \int_{\Sigma_{\eps,\kappa}} \tilde \psi_{\eps,\kappa,0} r_{\eps,\kappa}.
\end{equation}
The lemma now easily follows from \cref{lem_asymp_2} together with \eqref{eq_est_0}, \eqref{eq_est_1}, \eqref{eq_est_2}, and \eqref{eq_est_3} using the assumption $|m_{\eps,\kappa,1}| \geq \eps^{1/8}$ to remove the error term from the denominator.
\end{proof}

\subsection{Locating scales of interaction}

The picture for the first eigenfunction that we have at this point in particular gives that the first eigenfunction is very concentrated on $C_{\eps,\kappa}$ for $\kappa=\kappa_0$
and very concentrated on $\Sigma$ for $\kappa=\kappa_1$.
In fact, we even have quantitative control on this behaviour.
For intermediate values of $\kappa$ the picture changes dramatically, at least if $\lambda_1(\Sigma_{\eps,\kappa})$
is not too close to $\lambda_1(\Sigma)$.
In this subsection we record a quantitative version of this.

\begin{lemma} \label{lem_quant_kappa0}
Let $\tau < (1-\alpha)/2$.
There are $C=C(\Sigma,k,\delta_0,\kappa_0)>0$ and $\eps_{10}=\eps_{10}(\Sigma,k,\delta_0,\kappa_0,\tau)>0$ such that,
at $\kappa_0$ we have,
that 
$$
\int_{C_{\eps,\kappa}} u_{\eps,\kappa} \psi_{\eps,\kappa,0} \geq 1 - C \eps^{\tau}
$$
for $\eps \leq \eps_{10}$.
\end{lemma}

Recall that we normalized the sign of the left hand side at \eqref{eq_def_mass}.

\begin{proof}
Arguing as in the second step of the proof of \cref{L^2_surface_part}, using \cref{L^2_ground_mass_cyl} and \cref{L^2_conc_surf}
we find that there is exactly one eigenfunction below $\lambda_0(C_{\eps,\kappa}) + \eps^{3/2\alpha+1/2 -\tau}$ for $\eps$ sufficiently small.
Therefore, we can approximate this eigenfunction $u_{\eps,\kappa}$ by the quasimode $\tilde \psi_{\eps,\kappa,0}$ with
$$
\|u_{\eps,\kappa} - \tilde \psi_{\eps,\kappa,0}\|_{L^2(\Sigma_{\eps,\kappa})} \leq C \eps^{\tau}.
$$
But this implies 
$$
\int_{C_{\eps,\kappa}} u_{\eps,\kappa} \psi_{\eps,\kappa,0}
=
1 + O(\eps^{3\alpha/2+1/2}) + \int_{C_{\eps,\kappa}}u_{\eps,\kappa} (u_{\eps,\kappa} - \tilde \psi_{\eps,\kappa,0}),
$$
where the last term is in $O(\eps^\tau)$ by H{\"o}lder's inequality and the preceding bound.
\end{proof}

\begin{lemma} \label{L^2_lem_loc_scale}
Let $\tau < (1-\alpha)/2, k \geq 9$ and
assume that $\lambda_1(\Sigma_{\eps,\kappa}) \leq \lambda_1(\Sigma) - \eps^{k/4}$ for any $\kappa \in [\kappa_0,\kappa_1]$.
There is $\eps_{11}=\eps_{11}(\Sigma,k,\delta_0,\kappa_0,\kappa_1)>0$ and $c=c(\Sigma,\kappa_1,\kappa_2)>0$ such that for any 
$$
\rho \in (c^{-1} \eps^{3\alpha/2+1/2},c\,\eps^{3\alpha/2+1/2-\tau/2})
$$ and $\eps \leq \eps_{11}$
there is $\kappa_\eps \in (\kappa_0,\kappa_1)$ such that
there is a normalized $\lambda_1(\Sigma_{\eps,\kappa_\eps})$ eigenfunction such that
$$
a_{\eps,\kappa_\eps,0} \frac{\phi_0(x_0) \lambda_0(C_{\eps,\kappa})}{\lambda_1(\Sigma)} \frac{n_{\eps,\kappa,1}}{m_{\eps,\kappa,1} }= \rho.
$$
\end{lemma}

The proof gets slightly more complicated for two reasons.
We have not ruled out the possibility $m_{\eps,\kappa,1}=0$ and we do not only want to control the size of $\rho$ but also its sign.

\begin{proof}
We first show a similar result for $n_{\eps,\kappa,1}/m_{\eps,\kappa,1}$ for an appropriately scaled set of parameters $\rho$, that is more aligned 
with our previous estimates.
Namely, for any $\rho  \in (1,\eps^{-\tau/2})$ and $\eps$ sufficiently small, there is
$\kappa_\eps$ such that 
\begin{equation} \label{eq_find_rho}
\frac{n_{\eps,\kappa,1}}{m_{\eps,\kappa,1} }= \rho.
\end{equation}

At $\kappa_0$, it follows from \cref{lem_quant_kappa0} that we have
\begin{equation} \label{eq_behav_1}
n_{\eps,\kappa_0,1}^2 \geq 1 - C \eps^{\tau}
\end{equation}
for $\eps$ sufficiently small.
When combined with the trivial estimate $m_{\eps,\kappa,1}^2 + n_{\eps,\kappa,1}^2 \leq 1$ this implies that also
\begin{equation} \label{eq_behav_2}
m_{\eps,\kappa_0,1}^2 \leq C \eps^{\tau}
\end{equation}
for $\eps$ sufficiently small.

Similarly, by \cref{L^2_cor_cyl} combined with \cref{L^2_conc_surf}, 
the assumption $\lambda_1(\Sigma_{\eps,\kappa}) \leq \lambda_1(\Sigma_{\eps,\kappa}) - \eps^{k/4}$, and \cref{expan_cyl}, we find that
\begin{equation} \label{eq_behav_3}
m_{\eps,\kappa_1,1}^2 \geq 3/4,
\end{equation}
for $\eps$ sufficiently small.
This implies that for such $\eps$ we also have
\begin{equation} \label{eq_behav_4}
n_{\eps,\kappa_1,1}^2 \leq 1/4.
\end{equation}

For $\eps$ fixed, such that \eqref{eq_behav_1}-\eqref{eq_behav_4} apply, we now consider the function 
$$
\rho(\kappa) = n_{\eps,\kappa,1}/m_{\eps,\kappa,1}.
$$
Note that $\rho$ defines a continuous function on the open subset of $[\kappa_0,\kappa_1]$ on which $m_{\eps,\kappa,1}$ does not vanish
thanks to \cref{prop_first_simple}.
If $\rho$ is not defined on all of $ [ \kappa_0,\kappa_1 ]$ consider the maximal subinterval of the 
type $(\kappa_{\eps,0},\kappa_1] \subset [\kappa_0,\kappa_1]$, on which $m_{\eps,\kappa,1} \neq 0$. 
In particular, $\rho$ defines a continuous function on $(\kappa_{\eps,0},\kappa_1]$.
Assume first that $\kappa_{\eps,0}>\kappa_0$.
We claim that we need to have 
$$
\lim_{\kappa \searrow \kappa_{\eps,0}} |\rho(\kappa)| = +\infty.
$$
The only way this could not happen is that also
$$
\lim_{\kappa \searrow \kappa_{\eps,0}} n_{\eps,\kappa,1} \to 0.
$$
But this is impossible.
Since this would imply on the one hand that 
$$\int_{C_{\eps,\kappa}} |u_{\eps,\kappa}|^2 \to 0
$$ 
by \cref{L^2_ground_mass_cyl} and \cref{first_upper_bd}.
On the other hand, \cref{L^2_conc_surf} then implies that we need to find a quasimode $\tilde \phi_{\eps,\kappa,i}$ with $i \geq 0$ (since $i=0$ is ruled out by construction) such that
$$
\int \tilde \phi_{\eps,\kappa,0} u_{\eps,\kappa} \geq 1/(4K)
$$ 
for $\eps$ sufficiently small.
This implies 
$$
\lambda_1(\Sigma_{\eps,\kappa}) \geq \lambda_1(\Sigma) - \eps^{k/4}
$$ contradicting our assumption.

It now follows from the observation that $m_{\eps,\kappa,1}$ and $n_{\eps,\kappa,1}$ have the same sign whenever $|m_{\eps,\kappa,1}| \geq \eps^{1/2}$ and $|n_{\eps,\kappa,1}| \geq \eps^{3\alpha/2+1/4}$ from
 the proof of \cref{L^2_main_decomp} 
and the intermediate value theorem that we need to have 
$$
\lim_{\kappa \searrow \kappa_{\eps,0}} \rho(\kappa) = +\infty.
$$
The assertion at \eqref{eq_find_rho} then follows from the intermediate value theorem.
If $\rho$ is defined on all of $[\kappa_0,\kappa_1]$ the argument is almost identical.

The assertion of the lemma now follows immediately from the mean value theorem applied on an interval as above to the function
$$
\eps^{-3\alpha/2-1/2} a_{\eps,\kappa,0} \frac{\phi_0(x_0) \lambda_0(C_{\eps,\kappa})}{ \lambda_1(\Sigma)} \frac{n_{\eps,\kappa,1}}{m_{\eps,\kappa,1}}.
$$
Note that this is simply the function from above multiplied by a continuous positive function that is bounded above and below by uniform constants thanks to \cref{spec_rem_sharp} and since $a_{\eps,\kappa,0}>0$.
\end{proof}

\section{Lower bound for the first eigenvalue} \label{sec_asymp}

In this section we iteratively improve our preliminary bound \cref{L^2_lem_pre_bd} and use this to conclude our main technical result \cref{thm_main_technical}.
Throughout this entire section we assume that
$$
1/3< \alpha \leq \alpha_0<9/16
$$
for some fixed $\alpha_0$.

\subsection{The main estimate}

The computation below is similar to the one in the proof of \cref{prop_first_simple}, but we now have to carefully keep track
of the error terms arising since we have to work with a linear combination of two
rather than a single eigenfunction.

Suppose we have are in the situation of \cref{sec_conc_decomp_two}.
Then we have a first eigenfunction $u_{\eps,\kappa}$ and a second eigenfunction $v_{\eps,\kappa}$, which are both unique up to scaling
and we normalize their sign by 
requiring that $n_{\eps,\kappa,1},n_{\eps,\kappa,2}\geq0$ (cf.\ \eqref{eq_def_mass}).
We may now choose $\beta_{\eps,\kappa} \in \IR$ such that the function 
$$
w_{\eps,\kappa} = u_{\eps,\kappa} + \beta_{\eps,\kappa} v_{\eps,\kappa}
$$
satisfies
$$
\int_{C_{\eps,\kappa}} w_{\eps,\kappa} = 
-\int_{\Sigma \setminus B_{\eps^k}} w_{\eps,\kappa} = 0.
$$

Using $w_{\eps,\kappa}$ as a test function we obtain the following.

\begin{lemma} \label{comp_neu_bd}
For $k \geq 9$ there is $\eps_{12}=\eps_{12}(\Sigma,k,\delta_0,\kappa_0,\kappa_1,\tau)>0$ with the following property. 
For any $\eps \leq \eps_{12}$,
if $n_{\eps,\kappa,1}^2 \leq 1-\eps^{ \tau}$ where $ \tau < (1-\alpha)/2$ and $\lambda_2(\Sigma_{\eps,\kappa}) \leq \lambda_1(\Sigma) - \eps^{k/4}$, we have that
\begin{equation*}
\begin{split}
\mu_1(\Sigma \setminus B_{\eps^k})
\leq &
\frac{ \lambda_1(\Sigma_{\eps,\kappa}) + \beta_{\eps,\kappa}^2 \lambda_2(\Sigma_{\eps,\kappa})} {1+\beta_{\eps,\kappa}^2 }
\\
&+
\left( \frac{\lambda_1(\Sigma_{\eps,\kappa})+\beta_{\eps,\kappa}^2 \lambda_2(\Sigma_{\eps,\kappa}) -  (1+\beta_{\eps,\kappa}^2)\mu_1(C_{\eps,\kappa}) }{\mu_1(C_{\eps,\kappa})(m_{\eps,\kappa,1}+\beta_{\eps,\kappa} m_{\eps,\kappa,2})^2}\right) \lambda_2(\Sigma_{\eps,\kappa})
\end{split}
\end{equation*}
provided that the denominator of the last term is non-zero.\footnote{We show in \cref{beta_asymp} below that this holds under some mild additional assumptions.}
\end{lemma}

\begin{proof}
First note that we have that $\lambda_2(\Sigma_{\eps,\kappa}) \leq \lambda_0(C_{\eps,\kappa}) + \eps^{3\alpha/2+1/2-\tau}$.
This is since if not, we have from \cref{lem_anne_quasimod} applied to \cref{lem_asymp_2} that 
$
n_{\eps,\kappa,1}^2 \geq 1 - C \eps^{2\tau} > 1 - \eps^{\tau}
$ 
for $\eps$ sufficiently small contradicting our assumption.
In particular, the discussion preceding the lemma applies and we have $\beta_{\eps,\kappa}$ and then also $w_{\eps,\kappa}$ as above.

The $L^2$-norm of $w_{\eps,\kappa}$ is given by
\begin{equation} \label{eq_w_l2}
\begin{split}
\int_{\Sigma_{\eps,\kappa}} & |w_{\eps,\kappa}|^2 
=
1+\beta_{\eps,\kappa}^2,
\end{split}
\end{equation}
since $u_{\eps,\kappa}$ and $v_{\eps,\kappa}$ are normalized and orthogonal in $L^2$.
Similarly, we find for the gradient that
\begin{equation} \label{eq_w_h1}
\begin{split}
\int_{\Sigma_{\eps,\kappa}} & |\nabla w_{\eps,\kappa}|^2 
= 
\lambda_1(\Sigma_{\eps,\kappa}) + \beta_{\eps,\kappa}^2 \lambda_2(\Sigma_{\eps,\kappa})
=
\frac{\lambda_1(\Sigma_{\eps,\kappa})+\beta_{\eps,\kappa}^2 \lambda_2(\Sigma_{\eps,\kappa})}{1+\beta_{\eps,\kappa}^2} \int_{\Sigma_{\eps,\kappa}} |w_{\eps,\kappa}|^2,
\end{split}
\end{equation}
where we used that the gradients of $u_{\eps,\kappa}$ and $v_{\eps,\kappa}$ are orthogonal in $L^2$ and \eqref{eq_w_l2}.
In particular, we have that
\begin{equation} \label{eq_w_h1_2}
\int_{\Sigma_{\eps,\kappa}} | \nabla w_{\eps,\kappa}|^2 \leq (1+\beta_{\eps,\kappa}^2) \lambda_2(\Sigma_{\eps,\kappa}).
\end{equation}
Since $w_{\eps,\kappa}$ integrates to zero over $C_{\eps,\kappa}$, we also have that
$$
\int_{C_{\eps,\kappa}} |w_{\eps,\kappa}|^2 \leq \frac{1}{\mu_1(C_{\eps,\kappa})} \int_{C_{\eps,\kappa}} |\nabla w_{\eps,\kappa}|^2,
$$
which implies
\begin{equation*} \label{eq_key}
\begin{split}
\int_{\Sigma \setminus B_{\eps^k}}  |\nabla w_{\eps,\kappa}|^2
= &
\frac{\lambda_1(\Sigma_{\eps,\kappa})+\beta_{\eps,\kappa}^2 \lambda_2(\Sigma_{\eps,\kappa})}{1+\beta_{\eps,\kappa}^2}  \int_{\Sigma \setminus B_{\eps^k}} |w_{\eps,\kappa}|^2
+\frac{\lambda_1(\Sigma_{\eps,\kappa})
+\beta_{\eps,\kappa}^2 \lambda_2(\Sigma_{\eps,\kappa})}{1+\beta_{\eps,\kappa}^2} \int_{C_{\eps,\kappa}} |w_{\eps,\kappa}|^2
\\
&- \int_{C_{\eps,\kappa}} |\nabla w_{\eps,\kappa}|^2
\\
\leq&
\frac{\lambda_1(\Sigma_{\eps,\kappa})+\beta_{\eps,\kappa}^2 \lambda_2(\Sigma_{\eps,\kappa})}{1+\beta_{\eps,\kappa}^2}  \int_{\Sigma \setminus B_{\eps^k}} |w_{\eps,\kappa}|^2
\\
&+\left( \frac{\lambda_1(\Sigma_{\eps,\kappa})+\beta_{\eps,\kappa}^2 \lambda_2(\Sigma_{\eps,\kappa}) -  (1+\beta_{\eps,\kappa}^2)\mu_1(C_{\eps,\kappa})}{(1+\beta_{\eps,\kappa}^2)\mu_1(C_{\eps,\kappa})}\right) \int_{C_{\eps,\kappa}} |\nabla w_{\eps,\kappa}|^2
\\
 \leq &
\frac{\lambda_1(\Sigma_{\eps,\kappa})+\beta_{\eps,\kappa}^2 \lambda_2(\Sigma_{\eps,\kappa})}{1+\beta_{\eps,\kappa}^2} \int_{\Sigma \setminus B_{\eps^k}} |w_{\eps,\kappa}|^2
\\
&+\left( \frac{\lambda_1(\Sigma_{\eps,\kappa})+\beta_{\eps,\kappa}^2 \lambda_2(\Sigma_{\eps,\kappa}) -  (1+\beta_{\eps,\kappa}^2)\mu_1(C_{\eps,\kappa})}{\mu_1(C_{\eps,\kappa})}\right) \lambda_2(\Sigma_{\eps,\kappa}),
\end{split}
\end{equation*}
where we have used \eqref{eq_w_h1_2} in the last step.

In order to derive an upper bound for $\mu_1(\Sigma \setminus B_{\eps^k})$ from this, we need a lower bound on $\int_{\Sigma \setminus B_{\eps^k}} |w_{\eps,\kappa}|^2$ in terms of $\beta_{\eps,\kappa}$ and $\eps$.
We can combine the estimate
\begin{equation*}
\int_{\Sigma \setminus B_{\eps^k}} w_{\eps,\kappa} \phi_0
 \leq 
  \left(\int_{\Sigma \setminus B_{\eps^k}} |w_{\eps,\kappa}|^2 \right)^{1/2},
\end{equation*}
that uses the normalization of $\phi_0$, with
\begin{equation*}
\int_{\Sigma \setminus B_{\eps^k}} w_{\eps,\kappa} \phi_0
=
m_{\eps,\kappa,1} + \beta_{\eps,\kappa} m_{\eps,\kappa,2},
\end{equation*}
where we may assume that this integral is non-negative (by working with $-w_{\eps,\kappa}$ instead for this step if necessary)
to obtain
\begin{equation} \label{eq_w_mass_surf}
\int_{\Sigma \setminus B_{\eps^k}} |w_{\eps,\kappa}|^2 \geq (m_{\eps,\kappa,1} + \beta_{\eps,\kappa} m_{\eps,\kappa,2})^2.
\end{equation}
Since, by construction, $w_{\eps,\kappa}$ is an admissible test function for $\mu_1(\Sigma \setminus B_{\eps^k})$, the above estimates immediately imply the assertion.
\end{proof}

As stated above \cref{comp_neu_bd} is only partially helpful yet. 
The left hand side is controlled well from below in terms of $\lambda_1(\Sigma)$ thanks to \cref{thm_lower_neumann}, but we 
have to work out the right hand side much more precisely.
We can do this relying on the results from \cref{sec_conc}.

\begin{lemma} \label{beta_asymp}
Assume that 
$1/3 \leq \alpha \leq 65/192, k \geq 9$ and $3/4 \geq |m_{\eps,\kappa,1}|,|n_{\eps,\kappa,2}| \geq \eps^{1/32}$
and
$$
c_{\eps,\kappa}:=\lambda_1(\Sigma) - \lambda_2(\Sigma_{\eps,\kappa}) 
\geq \eps^{65/64}.
$$
We then have that
\begin{equation} \label{eq_expan_gap}
\begin{split}
\frac{ \lambda_1(\Sigma_{\eps,\kappa}) + \beta_{\eps,\kappa}^2 \lambda_2(\Sigma_{\eps,\kappa}) }{ 1 + \beta_{\eps,\kappa}^2} 
\leq & 
\lambda_2(\Sigma_{\eps,\kappa}) 
-\frac{n_{\eps,\kappa,1}}{m_{\eps,\kappa,1}} \frac{\phi_0(x_0) \lambda_0(C_{\eps,\kappa})}{\lambda_1(\Sigma)} a_{\eps,\kappa,0} 
\\
& +
2\frac{\phi_0(x_0)^2 \lambda_0(C_{\eps,\kappa})^2 }{\lambda_1(\Sigma)^2} \frac{a_{\eps,\kappa,0}^2}{ c_{\eps,\kappa}}
+ 
O(\eps^{65/64})
\end{split}
\end{equation}
and
\begin{equation} \label{eq_expan_error}
\left( \frac{\lambda_1(\Sigma_{\eps,\kappa})+\beta_{\eps,\kappa}^2 \lambda_2(\Sigma_{\eps,\kappa}) -  (1+\beta_{\eps,\kappa}^2)\mu_1(C_{\eps,\kappa}) }{\mu_1(C_{\eps,\kappa})(m_{\eps,\kappa,1}+\beta_{\eps,\kappa} m_{\eps,\kappa,2})^2}\right) \lambda_2(\Sigma_{\eps,\kappa}) \leq
2 c_{\eps,\kappa} + O(\eps^{129/128})
\end{equation}
as $\eps \to 0$ with uniform error terms as long as $\kappa \in [\kappa_0,\kappa_1]$, and $m_{\eps,\kappa,i},n_{\eps,\kappa,i}$ are as above.
\end{lemma}

In view of \eqref{eq_expan_error} the bound from \cref{comp_neu_bd} seems to be entirely useless.
However, if $m_{\eps,\kappa,1}$ is very small, the gap in \eqref{eq_expan_gap} to $\lambda_2(\Sigma_{\eps,\kappa})$ is very large compensating for the loss coming from \eqref{eq_expan_error}.
In the next section we will make use of this observation and iteratively use \cref{comp_neu_bd} to improve the estimate on the gap $\lambda_1(\Sigma) - \lambda_0(C_{\eps,\kappa})$ by adjusting the parameter $\kappa$ carefully.

\begin{rem}
This is the second time we make use of our precise knowledge of the leading order term in the expansion coming from the quasimodes concentrated on $C_{\eps,\kappa}$, see \cref{expan_cyl}.
\end{rem}

\begin{proof}
Recall that the eigenvalue asymptotics from \cref{expan_surf} implies that
\begin{equation} \label{eq_eig1_expan_1}
\begin{split}
\lambda_1(\Sigma_{\eps,\kappa})
&=
\lambda_1(\Sigma)
-
\frac{ \phi_0(x_0) \lambda_1(\Sigma)  \int_{C_{\eps,\kappa}} u_{\eps,\kappa} + O(\eps^{k/2})}{m_{\eps,\kappa,1} + O(\eps \log(1/\eps))}
\\
&=
\lambda_1(\Sigma) -  \phi_0(x_0) \lambda_1(\Sigma) \frac{ \int_{C_{\eps,\kappa}} u_{\eps,\kappa} } {m_{\eps,\kappa,1}}  + O(\eps^{3/2} \log^2(1/\eps))
\end{split}
\end{equation}
since we assume $|m_{\eps,\kappa,1}| \geq \eps^{1/4}$ and $k \geq 3$ and where we use \cref{L^2_mean_bd}.
The expansion from \cref{expan_cyl} combined with \cref{L^2_mean_bd} (see also \cref{rem_asymp_improved}) implies 
\begin{equation} \label{eq_eig1_expan_2}
\begin{split}
\lambda_1(\Sigma_{\eps,\kappa})
&=
\lambda_0(C_{\eps,\kappa})  
- \frac{ \phi_0(x_0) \frac{\lambda_0(C_{\eps,\kappa})}{ \lambda_1(\Sigma)} m_{\eps,\kappa,1} a_{\eps,\kappa,0} + O(\eps^{3(\alpha+1)/2} \log(1/\eps))}{n_{\eps,\kappa,1} + O(\eps^{3\alpha/2+1/2})} 
\\
&=
\lambda_0(C_{\eps,\kappa}) - \frac{\phi_0(x_0) \lambda_0(C_{\eps,\kappa})}{\lambda_1(\Sigma)} \frac{m_{\eps,\kappa,1}}{n_{\eps,\kappa,1}}  a_{\eps,\kappa,0} + O(\eps^{3/2} \log^2(1/\eps))
\end{split}
\end{equation}
since we have that $|n_{\eps,\kappa,1}| \geq \eps^{1/4}$ thanks to \cref{L^2_main_decomp}, and $k \geq 3$.\footnote{Of course, the error term here is non-sharp, and similarly in some formulae below. Since we combine the different expansions in a second, we only keep track of the worst case.}
Similarly, we find for the second eigenvalue that
\begin{equation}  \label{eq_eig2_expan_1}
\begin{split}
\lambda_2(\Sigma_{\eps,\kappa})
=&
\lambda_1(\Sigma)
-
\frac{ \phi_0(x_0) \lambda_1(\Sigma)  \int_{C_{\eps,\kappa}} v_{\eps,\kappa} + O(\eps^{k/2})}{m_{\eps,\kappa,2} + O(\eps \log(1/\eps))}
\\
&=
\lambda_1(\Sigma) -  \phi_0(x_0) \lambda_1(\Sigma)  \frac{ \int_{C_{\eps,\kappa}} v_{\eps,\kappa} } {m_{\eps,\kappa,2}} + O(\eps^{3/2} \log^2(1/\eps))
\end{split}
\end{equation}
since  $|m_{\eps,\kappa,2}| \geq \eps^{1/4}$ thanks to \cref{L^2_main_decomp}, $k \geq 3$, and where we use \cref{L^2_mean_bd} again.
The expansion coming from \cref{expan_cyl} gives 
\begin{equation} \label{eq_expan_sec_eigen}
\begin{split}
\lambda_2(\Sigma_{\eps,\kappa})
=&
\lambda_0(C_{\eps,\kappa})  
- \frac{  \phi_0(x_0) \frac{\lambda_0(C_{\eps,\kappa})}{\lambda_1(\Sigma)} m_{\eps,\kappa,2} a_{\eps,\kappa,0} + O(\eps^{3(\alpha+1)/2} \log(1/\eps))}{n_{\eps,\kappa,2} + O(\eps^{3\alpha/2+1/2})} 
\\
&=
\lambda_0(C_{\eps,\kappa}) - \frac{ \phi_0(x_0) \lambda_0(C_{\eps,\kappa}) }{\lambda_1(\Sigma)}  \frac{m_{\eps,\kappa,2}}{n_{\eps,\kappa,2}}a_{\eps,\kappa,0} + O(\eps^{3/2}\log^2(1/\eps))
\end{split}
\end{equation}
since we assume $|n_{\eps,\kappa,2}| \geq \eps^{1/4}$ and $k \geq 3$.

\smallskip

This implies, under our assumptions from above that
\begin{equation} \label{eq_expan_beta_0}
\begin{split}
\phi_0 & (x_0) \lambda_1 (\Sigma)  \frac{\int_{C_{\eps,\kappa}} u_{\eps,\kappa}}{m_{\eps,\kappa,1}}
=(\lambda_1(\Sigma) - \lambda_2(\Sigma_{\eps,\kappa})) + (\lambda_2(\Sigma_{\eps,\kappa}) - \lambda_1(\Sigma_{\eps,\kappa})) + O(\eps^{3/2}\log^2(1/\eps))
\\
&=
\phi_0(x_0) \lambda_1(\Sigma)  \frac{\int_{C_{\eps,\kappa}} v_{\eps,\kappa}}{m_{\eps,\kappa,2}}
 +
\frac{\phi_0(x_0)\lambda_0(C_{\eps,\kappa})}{\lambda_1(\Sigma)}  \left( \frac{ m_{\eps,\kappa,1} }{n_{\eps,\kappa,1}} - \frac{m_{\eps,\kappa,2}}{n_{\eps,\kappa,2}} \right) a_{\eps,\kappa,0}
+ O(\eps^{3/2} \log^2(1/\eps)),
\end{split}
\end{equation}
where the error terms are uniformly in $\kappa$ as long as $m_{\eps,\kappa,i},n_{\eps,\kappa,i}$ are as above.

We will write
$$
c_{\eps,\kappa}' : = \phi_0(x_0) \lambda_1(\Sigma) \frac{\int_{C_{\eps,\kappa}} v_{\eps,\kappa} }{m_{\eps,\kappa,2}}
$$
and note that $\lambda_2(\Sigma_{\eps,\kappa}) - \lambda_1(\Sigma)=-c_{\eps,\kappa}' + O(\eps^{3/2}\log(1/\eps))$,
by \eqref{eq_eig2_expan_1}.
The assumption $\lambda_1(\Sigma) - \lambda_2(\Sigma_{\eps,\kappa}) \geq \eps^{9/8}$ 
combined with $|m_{\eps,\kappa,2}|\geq 1/16$ (thanks to our assumptions and \cref{L^2_main_decomp}) implies
$$
\left| \int_{C_{\eps,\kappa}} v_{\eps,\kappa} \right| \geq C \eps^{9/8} |m_{\eps,\kappa,2}| \geq  C \eps^{9/8}
$$
for $\eps$ sufficiently small.
Therefore, we may divide by $-\int_{C_{\eps,\kappa}} v_{\eps,\kappa}$ in \eqref{eq_expan_beta_0} to obtain
\begin{equation} \label{eq_expan_beta_1}
\begin{split}
\beta_{\eps,\kappa}
=& - 
\frac{\int_{C_{\eps,\kappa}} u_{\eps,\kappa}} { \int_{C_{\eps,\kappa}} v_{\eps,\kappa}}
\\
=&
-\frac{m_{\eps,\kappa,1}}{m_{\eps,\kappa,2}} 
+ 
\frac{\phi_0(x_0)\lambda_0(C_{\eps,\kappa})}{\lambda_1(\Sigma)} \left(\frac{m_{\eps,\kappa,1}}{n_{\eps,\kappa,2}} - \frac{m_{\eps,\kappa,1}^2}{m_{\eps,\kappa,2} n_{\eps,\kappa,1}}\right) \frac{a_{\eps,\kappa,0}}{c_{\eps,\kappa}'} 
\\
&+ O(\eps^{3/8} \log^2(1/\eps)).
\end{split}
\end{equation}
We can rewrite this using \cref{L^2_main_decomp} as follows.
Since we assume $\alpha < 3/8$, we can apply \cref{L^2_main_decomp} with $\tau=\alpha - 1/16 \geq 1/4$.
Also recall that we have $|m_{\eps,\kappa,1}| \leq 3/4$ by assumption, and hence also $|m_{\eps,\kappa,2}|,|n_{\eps,\kappa,1}| \geq 1/16$
by \cref{L^2_main_decomp} if $\eps$ is sufficiently small.
Therefore, we find from \eqref{eq_L^2_decomp_4} that
$$
- \frac{m_{\eps,\kappa,1}}{m_{\eps,\kappa,2}} = \frac{m_{\eps,\kappa,1}}{n_{\eps,\kappa,1}} + O(\eps^{1/16}).
$$
Similarly, from \eqref{eq_L^2_decomp_3} and \eqref{eq_L^2_decomp_4}, since we also assume $|n_{\eps,\kappa,2}|,|m_{\eps,\kappa,1}| \geq \eps^{1/32}$, we find that
$$
\left( \frac{m_{\eps,\kappa,1}}{n_{\eps,\kappa,2}} -  \frac{m_{\eps,\kappa,1}^2}{m_{\eps,\kappa,2} n_{\eps,\kappa,1}}\right)
=
\left( 1 + \frac{m_{\eps,\kappa,1}^2}{n_{\eps,\kappa,1}^2}\right) + O(\eps^{1/32}).
$$
Combining these with the formula \eqref{eq_expan_beta_1} for $\beta_{\eps,\kappa}$ from above we arrive at

\begin{equation} \label{eq_expan_beta_2}
\beta_{\eps,\kappa} = 
\frac{m_{\eps,\kappa,1}}{n_{\eps,\kappa,1}}
+
\left(1 + \frac{m_{\eps,\kappa,1}^2}{n_{\eps,\kappa,1}^2} \right) \frac{\phi_0(x_0) \lambda_0(C_{\eps,\kappa})}{\lambda_1(\Sigma)} \frac{ a_{\eps,\kappa,0} }{c_{\eps,\kappa}'}
+
O(\eps^{1/64})
\end{equation}
where we also use the assumption that $|a_{\eps,\kappa,0}/c_{\eps,\kappa}'| \eps^{1/64}\leq 1$.

\smallskip

We can now combine the first expansion \eqref{eq_expan_beta_1} for $\beta_{\eps,\kappa}$ with the expansions for $\lambda_1(\Sigma_{\eps,\kappa})$ and $\lambda_2(\Sigma_{\eps,\kappa})$ from above coming from the quasimode concentrated on the cylinder to 
find that
\begin{equation*}
\begin{split}
\lambda_1& (\Sigma_{\eps,\kappa}) 
+ 
\beta_{\eps,\kappa}^2 \lambda_2(\Sigma_{\eps,\kappa})
\\&=
(1+\beta_{\eps,\kappa}^2) \lambda_0(C_{\eps,\kappa})
-
\left(
\frac{m_{\eps,\kappa,1}}{n_{\eps,\kappa,1}} + \frac{m_{\eps,\kappa,1}^2}{m_{\eps,\kappa,2} n_{\eps,\kappa,2}}
\right)
\frac{\phi_0(x_0) \lambda_0(C_{\eps,\kappa})}{\lambda_1(\Sigma)} a_{\eps,\kappa,0}
\\
&+
2 \frac{m_{\eps,\kappa,1}}{n_{\eps,\kappa,2}} \left(\frac{m_{\eps,\kappa,1}}{n_{\eps,\kappa,2}} - \frac{m_{\eps,\kappa,1}^2}{m_{\eps,\kappa,2} n_{\eps,\kappa,1}}  \right) \frac{\phi_0(x_0)^2 \lambda_0(C_{\eps,\kappa})^2}{\lambda_1(\Sigma)^2} \frac{a_{\eps,\kappa,0}^2}{c_{\eps,\kappa}'} 
+O(\eps^{33/32}\log^2(1/\eps))
\end{split}
\end{equation*}
since $|m_{\eps,\kappa,i}|,|n_{\eps,\kappa,i}| \geq \eps^{1/32}$ and $\eps^{1/64}|a_{\eps,\kappa,0}/c_{\eps,\kappa}| \leq 1$, which implies that $|a_{\eps,\kappa,0}^2/c_{\eps,\kappa}'| \leq \eps^{63/64}$ for $\eps$ sufficiently small. 

From \cref{L^2_main_decomp} once more, we also have that
$$
\left(
\frac{m_{\eps,\kappa,1}}{n_{\eps,\kappa,1}} + \frac{m_{\eps,\kappa,1}^2}{m_{\eps,\kappa,2} n_{\eps,\kappa,2}}
\right)
=
O(\eps^{1/32})
$$
as long as $|m_{\eps,\kappa,1}| \leq 3/4$
and
$$
\frac{m_{\eps,\kappa,1}}{n_{\eps,\kappa,2}} = 1 + O(\eps^{1/32})
$$
and
$$
\frac{m_{\eps,\kappa,1}^2}{m_{\eps,\kappa,2} n_{\eps,\kappa,1}} = - \frac{m_{\eps,\kappa,1}^2}{n_{\eps,\kappa,1}^2} + O (\eps^{1/32}).
$$
Using $|a_{\eps,\kappa,0}^2/c_{\eps,\kappa}'| \leq \eps^{63/64}$ once again we obtain
\begin{equation} \label{eq_expan_enum}
\begin{split}
\lambda_1(\Sigma_{\eps,\kappa}) + & \beta_{\eps,\kappa}^2 \lambda_2(\Sigma_{\eps,\kappa})
=
(1+\beta_{\eps,\kappa}^2)\lambda_0(C_{\eps,\kappa})
\\
&+ 2 \left(1+ \frac{m_{\eps,\kappa,1}^2}{n_{\eps,\kappa,1}^2} \right) \frac{ \phi_0(x_0)^2 \lambda_0(C_{\eps,\kappa})^2 } {\lambda_1(\Sigma)^2} \frac{a_{\eps,\kappa,0}^2}{c_{\eps,\kappa}'}
+O(\eps^{65/64}).
\end{split}
\end{equation}
Note that \eqref{eq_expan_beta_2} in particular implies that 
$$
1+\beta_{\eps,\kappa}^2 \geq 1 + \frac{m_{\eps,\kappa,1}^2}{n_{\eps,\kappa,1}^2}
$$
for $\eps$ sufficiently small 
since $\alpha \leq 65/192$ and $|c_{\eps,\kappa}'|\leq C \eps \log(1/\eps)$ for $\eps$ sufficiently small thanks to \cref{L^2_lem_pre_bd} which is applicable by assumption.
Therefore, \eqref{eq_expan_enum} combined with \eqref{eq_expan_sec_eigen} and using $|a_{\eps,\kappa,0}^2/c_{\eps,\kappa}'| \leq \eps^{63/64}$ to change $c_{\eps,\kappa}'$ into $c_{\eps,\kappa}$ we easily obtain \eqref{eq_expan_gap}.

From the expansion \eqref{eq_expan_beta_1} of $\beta_{\eps,\kappa}$ we also obtain
\begin{equation*} 
m_{\eps,\kappa,1} + \beta_{\eps,\kappa} m_{\eps,\kappa,2}
=
m_{\eps,\kappa,2} \left(\frac{m_{\eps,\kappa,1}}{n_{\eps,\kappa,2}} - \frac{m_{\eps,\kappa,1}^2}{m_{\eps,\kappa,2}n_{\eps,\kappa,1}}\right) \frac{\phi_0(x_0) \lambda_0(C_{\eps,\kappa})}{\lambda_1(\Sigma)} \frac{a_{\eps,\kappa,0}}{c_{\eps,\kappa}'} + O(\eps^{1/8} \log^2(1/\eps)).
\end{equation*}
When combined with \cref{L^2_main_decomp} this easily implies that
\begin{equation*}
\begin{split}
(m_{\eps,\kappa,1} &+ \beta_{\eps,\kappa} m_{\eps,\kappa,2})^2 
\\
&= 
 \left(\frac{m_{\eps,\kappa,1}}{n_{\eps,\kappa,2}} - \frac{m_{\eps,\kappa,1}^2}{m_{\eps,\kappa,2}n_{\eps,\kappa,1}}\right) \frac{\phi_0(x_0)^2 \lambda_0^2(C_{\eps,\kappa})}{\lambda_1(\Sigma)^2} \frac{a_{\eps,\kappa,0}^2}{c_{\eps,\kappa}'^2}
  +  O(\eps^{1/16} \log^2(1/\eps))
 \end{split}
\end{equation*}
as long as $|m_{\eps,\kappa,i}|,|n_{\eps,\kappa,i}| \geq \eps^{1/32}$ since $|a_{\eps,\kappa}/c_{\eps,\kappa}|\leq \eps^{-1/64}$.
Combining this with \eqref{eq_expan_enum} from above and with $\lambda_0(C_{\eps,\kappa}) \leq \mu_1 (C_{\eps,\kappa})$ (see \cref{spec_comp_dir_neu}\footnote{exploiting this was also key in \cref{prop_first_simple}}) we obtain that
\begin{equation*} 
\begin{split}
\left( \frac{\lambda_1(\Sigma_{\eps,\kappa})+\beta_{\eps,\kappa}^2 \lambda_2(\Sigma_{\eps,\kappa}) -  (1+\beta_{\eps,\kappa}^2)\mu_1(C_{\eps,\kappa}) }{\mu_1(C_{\eps,\kappa})(m_{\eps,\kappa,1}+\beta_{\eps,\kappa} m_{\eps,\kappa,2})^2}\right) \lambda_2(\Sigma_{\eps,\kappa})
\leq (2c_{\eps,\kappa} + O(\eps^{129/128}) ) 
\frac{\lambda_2(\Sigma_{\eps,\kappa})}{\mu_1(C_{\eps,\kappa})},
\end{split}
\end{equation*}
since $\alpha < 65/192$ and making use of the $\eps\log(1/\eps)$ bound \cref{L^2_lem_pre_bd}.
Moreover we have that
\begin{equation*}
\lambda_2(\Sigma_{\eps,\kappa})
\leq \lambda_0(C_{\eps,\kappa}) + \eps^{3\alpha/2+1/2-\tau}
\leq \mu_1(C_{\eps,\kappa}) + \eps^{3\alpha/2+1/2-\tau},
\end{equation*}
which easily implies that
\begin{equation*}
\frac{\lambda_2(\Sigma_{\eps,\kappa})}{ \mu_1(C_{\eps,\kappa})}
\leq
1+
C \eps^{3\alpha/2+1/2-\tau}
\leq 
1 +  C \eps^{1/2}
\end{equation*}
by assumption and 
since $\mu_1(C_{\eps,\kappa}) \geq c>0$.
From this the conclusion easily follows using $c_{\eps,\kappa} \leq d_0 \eps \log(1/\eps)$ once again.
\end{proof}

\subsection{The iteration argument} \label{sec_iter}

We now use the bound from the previous section
to iteratively improve the gap
$$
\lambda_1(\Sigma) - \lambda_0(C_{\eps,\kappa})
$$
while staying in a range of parameters for $\kappa$ in which $m_{\eps,\kappa,1}$ is small. 
Eventually, this will show that we may choose $\kappa_\eps$ such that
$$
\lambda_1(\Sigma) - \lambda_0(C_{\eps,\kappa_\eps}) = o(\eps)
$$
and $m_{\eps,\kappa_{\eps},1} \in (1/4,3/4)$, which will easily imply \cref{thm_main_technical}.

\begin{prop} \label{prop_iter}
There is $\alpha_0>1/3$ with the following property.
For $\alpha \in (1/3,\alpha_0)$ such that
\begin{equation} \label{eq_extra_assum_gap}
\lambda_1(\Sigma_{\eps,\kappa}) \leq \lambda_1(\Sigma) - \eps^{k/4}
\end{equation}
for any $ \kappa \in [\kappa_0,\kappa_1]$,
there are $\rho_0>2$ and $\eps_*>0$ such that 
$$
\lambda_1(\Sigma) - \lambda_0(C_{\eps,\kappa}) = o(\eps)
$$
if $\kappa$ is such that $n_{\eps,\kappa,1}/m_{\eps,\kappa,1} \leq \rho_0$ and $\eps \leq \eps_*$.
\end{prop}

Let us briefly explain the idea of the argument.
Starting from the initial bound \cref{L^2_lem_pre_bd} on scale $\eps \log(1/\eps)$ that holds for a large set of parameters $\kappa$ we have some control on
the unfavourable first term on the right hand side of \eqref{eq_expan_error}.
We then select $\kappa$ in such a way that the negative term on the right hand side of \eqref{eq_expan_gap} compensates for this.  
This then forces the second term on the right hand side of \eqref{eq_expan_gap} to be not too small, which means that we get an upper bound on $c_{\eps,\kappa}$ which turns out to be much better than the initial bound.
The key observation then is that at the same time our specific choice of $\kappa$ also gives an improved bound for $c_{\eps,\kappa}$ for some parameters $\kappa$, which in fact comes from much precise control on $\lambda_0(C_{\eps,\kappa})$.
Repeating this argument, while making sure to stay in the set of parameters $\kappa$, for which we have an improved bound, we can get better and better control on $\lambda_0(C_{\eps,\kappa})$.

\begin{proof}
We take $1/3<\alpha_0<65/192$, fix $\tau=\alpha-1/16$ and write $\bar \eps = \min\{\eps_0,\cdots,\eps_{12}\}$.
From here on we only consider parameters $ \eps \leq \bar \eps$.
Note that $\bar \eps$ does not depend on $\alpha$.
In particular, for $\eps \leq \bar \eps$ we have uniformly controlled error terms in \cref{beta_asymp} as long as $m_{\eps,\kappa,i}$ and $n_{\eps,\kappa,i}$
meet the assumptions from \cref{beta_asymp}.
At this point, by abuse of notation, we 
forget about all our previous $\eps_l$ and will use a potentially different $\eps_l$ in the $l$-th step of the iteration argument below.

By \eqref{eq_extra_assum_gap} and our choices above we may apply \cref{L^2_lem_loc_scale} to find $\kappa_{\eps,1} \in [\kappa_0,\kappa_1]$  such that
\begin{equation} \label{eq_it_in_choice}
 \frac{\phi_0(x_0)  \lambda_0(C_{\eps,\kappa_{\eps,1}})}{\lambda_1(\Sigma)}
\frac{n_{\eps,\kappa_{\eps,1},1}}{m_{\eps,\kappa_{\eps,1},1} }
a_{\eps,\kappa_{\eps,1},0}
= 
(1+\theta)d_0 \eps \log(1/\eps),
\end{equation}
for some small $\theta \in (0,1/2)$ that we fix once and for all time.
Of course, we may also assume that $\kappa_{\eps,1}$ is chosen minimally in $[\kappa_0,\kappa_1]$ such that \eqref{eq_it_in_choice} holds.

Upon decreasing $\bar \eps$, we can assume that
\begin{equation} \label{eq_mass_comp}
m_{\eps,\kappa,1}^2 + n_{\eps,\kappa,1}^2 \geq 15/16
\end{equation}
for any $\kappa \in [\kappa_0,\kappa_1]$ if $\eps \leq \bar \eps$ thanks to \cref{first_upper_bd}, \cref{L^2_ground_mass_cyl}, \cref{L^2_surface_part}, and \eqref{eq_extra_assum_gap}.
Therefore,  
our choice of $\alpha_0$ and $\kappa_{\eps,1}$ guarantees that $m_{\eps,\kappa_{\eps,1},1} , n_{\eps,\kappa_{\eps,1},1}$ are such that we may apply \cref{beta_asymp} and \cref{L^2_lem_pre_bd} for $\eps \leq \bar \eps$.
Moreover, thanks to \eqref{eq_mass_comp} and our choice of $\alpha_0$ we also have that $|m_{\eps,\kappa_{\eps,1},1}| \geq \eps^{1/64}$ if $\eps \leq \bar \eps$.

We write
$$
c_{\eps,1} = \lambda_1(\Sigma) - \lambda_2(\Sigma_{\eps,\kappa_{\eps,1}})
$$
and assume that
\begin{equation} \label{eq_extra_assum_gap_2}
c_{\eps,1} \geq  \eps^{257/256}.
\end{equation}
By the preliminary bound in \cref{L^2_lem_pre_bd}, we have that
\begin{equation} \label{eq_it_in}
c_{\eps,1} \leq d_0 \eps \log(1/\eps),
\end{equation}
for $\eps \leq \bar \eps$
since $|m_{\eps,\kappa_{\eps,1},1}| \geq 2 \eps^{\tau}$ by our choices above.

We then find from \cref{comp_neu_bd} combined with \cref{beta_asymp} and \cref{thm_lower_neumann}
that
\begin{equation} \label{it_est}
\begin{split}
\lambda_2&(\Sigma_{\eps,\kappa_{1,\eps}}) 
+ c_{\eps,1}
=
\lambda_1(\Sigma)
\leq
\mu_1(\Sigma \setminus B_{\eps^k}) + O(\eps^k)
\\
\leq &
\lambda_2(\Sigma_{\eps,\kappa}) - (1+\theta)d_0 \eps \log(1/\eps) + 2 \frac{\lambda_0(C_{\eps,\kappa_{\eps,1}})^2 \phi_0(x_0)^2}{\lambda_1(\Sigma)^2}\frac{a_{\eps,\kappa_{\eps,1},0}^2}{c_{\eps,1}} 
+ 2c_{\eps,1}
+O(\eps^{129/128}).
\end{split}
\end{equation}
Thanks to \eqref{eq_it_in} this implies that
\begin{equation} \label{eq_it_est_2}
\begin{split}
-2  \frac{\lambda_0(C_{\eps,\kappa_{\eps,1}})^2 \phi_0(x_0)^2}{\lambda_1(\Sigma)^2} \frac{a_{\eps,\kappa_{\eps,1},0}^2}{c_{\eps,1}} 
&\leq
c_{\eps,1} - (1+\theta) d_0 \eps \log(1/\eps) +O(\eps^{129/128})
\\
&\leq
- \theta d_0 \eps \log(1/\eps) +O(\eps^{129/128})
\\
& \leq
- \frac{\theta}{2} d_0 \eps \log(1/\eps)
\end{split}
\end{equation}
as long as $\eps \leq \eps_1$, so that we can absorb the error term.
Therefore, we arrive at
$$
c_{\eps,1} \leq  \frac{\lambda_0(C_{\eps,\kappa_{\eps,1}})^2 \phi_0(x_0)^2}{\lambda_1(\Sigma)^2} \frac{4 a_{\eps,\kappa_{\eps,1},0}^2}{\theta d_0 \eps \log(1/\eps)}.
$$
Since $|a_{\eps,\kappa_{\eps,1},0}^2| \leq C \eps^{3\alpha +1}$
this is a significantly better bound on $\lambda_2(\Sigma_{\eps,\kappa_{\eps,1}})$ than our initial bound \eqref{eq_it_in}.
We now fix $1/3<\alpha<\alpha_0$, where $1/3<\alpha_0<65/192$ is sufficiently close to $1/3$ such that
\begin{equation} \label{eq_def_alpha0}
\frac{\lambda_0(C_{\eps,\kappa})^2 \phi_0(x_0)^2}{\lambda_1(\Sigma)^2} \frac{4 a_{\eps,\kappa,0}^2}{\theta d_0 \eps \log(1/\eps)}
\geq 
\eps^{257/256}
\end{equation}
for $\alpha_0$ and
for any $\kappa \in [\kappa_0,\kappa_1]$.

From \cref{L^2_main_decomp} we get that
\begin{equation} \label{eq_2_term}
\frac{n_{\eps,\kappa_{\eps,1},1}}{m_{\eps,\kappa_{\eps,1},1}}  a_{\eps,\kappa_{\eps,1},0}
=
- \frac{m_{\eps,\kappa_{\eps,1},2}}{n_{\eps,\kappa_{\eps,1},2}}  a_{\eps,\kappa_{\eps,1},0} + O(\eps^{65/64})
\end{equation}
since $|m_{\eps,\kappa_{\eps,1},1}| \geq \eps^{1/64}$.
Therefore, by the specific choice \eqref{eq_it_in_choice} combined with \eqref{eq_2_term}  and since $|n_{\eps,\kappa_{\eps,1},1}| \geq \eps^{1/32}$, we find from \cref{expan_cyl} (cf.\ \eqref{eq_expan_sec_eigen}) that
\begin{equation*} 
\begin{split}
\lambda_0(C_{\eps,\kappa_{\eps,1}}) 
&\geq 
\lambda_2(\Sigma_{\eps,\kappa_{\eps,1}})  - (1+\theta) d_0 \eps \log(1/\eps) + O(\eps^{129/128})
\\
&\geq
\lambda_1(\Sigma) - c_{\eps,1} -(1+\theta) d_0 \eps \log(1/\eps) + O(\eps^{129/128})
\\
&\geq
\lambda_1(\Sigma) - (1+\theta) d_0 \eps \log(1/\eps) -  \frac{\lambda_0(C_{\eps,\kappa_{\eps,1}})^2 \phi_0(x_0)^2}{\lambda_1(\Sigma)^2} \frac{8 a_{\eps,\kappa_{\eps,1},0}^2}{\theta d_0 \eps \log(1/\eps)}
\end{split}
\end{equation*}
for $\eps \leq \eps_1$ upon decreasing $\eps_1$ so that we can absorb the $O(\eps^{129/128})$ term into $2c_{\eps,1}$.
If \eqref{eq_extra_assum_gap_2} does not hold we obtain the exact same bound directly from \cref{L^2_lem_lambda_2_close}\footnote{In fact, this gives an even better bound thanks to our choice of $\alpha_0$, which avoids a case distinction.}
since $\alpha < \alpha_0$.
Therefore, we find that 
\begin{equation} \label{eq_bd_grd_1}
\lambda_0(C_{\eps,\kappa}) 
\geq 
\lambda_1(\Sigma) - (1+\theta) d_0 \eps \log(1/\eps) -  \frac{\lambda_0(C_{\eps,\kappa_{\eps,1}})^2 \phi_0(x_0)^2}{\lambda_1(\Sigma)^2} \frac{8 a_{\eps,\kappa_{\eps,1},0}^2}{\theta d_0 \eps \log(1/\eps)}
\end{equation}
for any $\kappa \geq \kappa_{\eps,1}$.

We now take the smallest $\kappa_{\eps,2} \in [\kappa_0,\kappa_1]$ given by \cref{L^2_lem_loc_scale} such that
\begin{equation} \label{eq_choice_2}
\frac{\lambda_0(C_{\eps,\kappa_{\eps,2}})\phi_0(x_0)}{\lambda_1(\Sigma)} \frac{n_{\eps,\kappa_{\eps,2},1}}{m_{\eps,\kappa_{\eps,2},1}} a_{\eps,\kappa_{\eps,2},0} =(1+\theta) \frac{1+\theta}{2}d_0 \eps \log(1/\eps).
\end{equation}
Note that we need to have $\kappa_{\eps,2} > \kappa_{\eps,1}$, in particular \eqref{eq_bd_grd_1} holds at $\kappa_{\eps,2}$.
Moreover, exactly as for $\kappa_{\eps,1}$, we see that $m_{\eps,\kappa_{\eps,2},1},n_{\eps,\kappa_{\eps,2},1}$ are such that we can invoke \cref{beta_asymp}.
We write
$$
c_{\eps,2}=\lambda_1(\Sigma) - \lambda_2(\Sigma_{\eps,\kappa_{\eps,2}})
$$
and assume again first that 
\begin{equation} \label{eq_gap_2}
c_{\eps,2} \geq \eps^{257/256}.
\end{equation}
In particular this allows us to invoke \cref{L^2_main_decomp} to obtain
\begin{equation} \label{eq_constraint}
\begin{split}
-\frac{\lambda_0(C_{\eps,\kappa_{\eps,2}})\phi_0(x_0)}{\lambda_1(\Sigma)} \frac{m_{\eps,\kappa_{\eps,2},2}}{n_{\eps,\kappa_{\eps,2},2}}  a_{\eps,\kappa_{\eps,2},0}
& = (1+\theta) \frac{1+\theta}{2}d_0 \eps \log(1/\eps) + O(\eps^{65/64})
\\
&\geq
\frac{(1+ \theta)}{2} d_0 \eps \log(1/\eps) +  \frac{\lambda_0(C_{\eps,\kappa_{\eps,1}})^2 \phi_0(x_0)^2}{\lambda_1(\Sigma)^2} \frac{8 a_{\eps,\kappa_{\eps,1},0}^2}{\theta d_0 \eps \log(1/\eps)}
\end{split}
\end{equation}
from \eqref{eq_choice_2} if $\eps \leq \eps_2$ such that the
$O(\eps^{65/64})$-term (which is uniform since we have sufficient control the size of $|m_{\eps,\kappa,1}|$ and $|n_{\eps,\kappa,1}|$) and the very last term\footnote{Note that $\eps^{-3\alpha-1} \lambda_0(C_{\eps,\kappa})^2 a_{\eps,\kappa,0}^2$ is uniformly bounded from above and below for $\kappa \in [\kappa_0,\kappa_1]$.}
 are both bounded by $\theta (1+\theta) d_0 \eps \log(1/\eps)/4$.

This in turn implies thanks to \cref{expan_cyl} and \eqref{eq_bd_grd_1} that
$$
\lambda_2(\Sigma_{\eps,\kappa_{\eps,2}}) \geq \lambda_1(\Sigma) -  \frac{1+\theta}{2} d_0 \eps \log(1/\eps),
$$
i.e.\
$$
c_{\eps,2} \leq \frac{1+\theta}{2} d_0 \eps \log(1/\eps)
$$
for $\eps \leq \eps_2$.

We now apply \cref{thm_lower_neumann}, \cref{comp_neu_bd}, and \cref{beta_asymp} in the exact same way as in \eqref{it_est} and \eqref{eq_it_est_2} to obtain
\begin{equation*}
\begin{split}
\lambda_2(\Sigma_{\eps,\kappa_{\eps,,2}}) + c_{\eps,2}
\leq 
\lambda_2(\Sigma_{\eps,\kappa_{\eps,2}}) - \frac{\theta}{2} \frac{1+\theta}{2} d_0 \eps \log(1/\eps) + 2  \frac{\lambda_0(C_{\eps,\kappa_{\eps,2}})^2 \phi_0(x_0)^2}{\lambda_1(\Sigma)^2} \frac{a_{\eps,\kappa_{\eps,2},0}^2}{c_{\eps,2}} + c_{\eps,2}
\end{split}
\end{equation*}
as long as $\eps \leq \eps_2$ (upon decreasing $\eps_2$) so that we can absorb the $O(\eps^{129/128})$ error term into  $\frac{\theta}{2} \frac{1+\theta}{2} d_0 \eps \log(1/\eps)$.
As before, this implies
\begin{equation} \label{eq_iter_2}
c_{\eps,2} 
\leq 
 \frac{\lambda_0(C_{\eps,\kappa_{\eps,2}})^2 \phi_0(x_0)^2}{\lambda_1(\Sigma)^2} \frac{8a_{\eps,\kappa_{\eps,2},0}^2}{\theta(1+\theta) d_0 \eps \log(1/\eps)}
\leq
 \frac{\lambda_0(C_{\eps,\kappa_{\eps,2}})^2 \phi_0(x_0)^2}{\lambda_1(\Sigma)^2}
\frac{8a_{\eps,\kappa_{\eps,2},0}^2}{\theta d_0 \eps \log(1/\eps)},
\end{equation}
which then gives for any $\kappa \in [\kappa_{\eps,2},\kappa_1]$ that
\begin{equation} \label{eq_bd_grd_2}
\lambda_0(C_{\eps,\kappa}) 
\geq 
\lambda_1(\Sigma) -  \frac{\lambda_0(C_{\eps,\kappa_{\eps,2}})^2 \phi_0(x_0)^2}{\lambda_1(\Sigma)^2} \frac{16a_{\eps,\kappa_{\eps,2},0}^2}{\theta d_0 \eps \log(1/\eps)} -  \frac{(1+\theta)^2}{2}d_0 \eps \log(1/\eps)
\end{equation}
for $\eps \leq \eps_2$ upon decreasing $\eps_2$
such that we can absorb the $O(\eps^{129/128})$ term into $2c_{\eps,2}$.
As above, if the bound \eqref{eq_gap_2} fails to be true, we get the exact same bound directly from \cref{L^2_lem_lambda_2_close}.
By our choice of $\alpha_0$ at \eqref{eq_def_alpha0}, the bound at \eqref{eq_iter_2} is worse that $\eps^{257/256}$.

We now run the very same argument once again with the smallest $\kappa_{\eps,3} \in [\kappa_{\eps,2},\kappa_1] $ such that
$$
\frac{\lambda_0(C_{\eps,\kappa_{\eps,3}})\phi_0(x_0)}{\lambda_1(\Sigma)} \frac{n_{\eps,\kappa_{\eps,3},1}}{m_{\eps,\kappa_{\eps,3},1}} a_{\eps,\kappa_{\eps,3},0} = (1+\theta) \frac{(1+\theta)^2}{4} d_0 \eps \log(1/\eps).
$$
Take $\eps_3>0$ such that we have for $\eps \leq \eps_3$ that
$$
 \frac{\lambda_0(C_{\eps,\kappa_{\eps,2}})^2 \phi_0(x_0)^2}{\lambda_1(\Sigma)^2} \frac{16 a_{\eps,\kappa_{\eps,2},0}^2}{\theta d_0 \eps \log(1/\eps)} \leq \theta  \frac{(1+\theta)^2}{4} d_0 \eps \log(1/\eps)
$$
and such that the $O(\eps^{129/128})$-term (which is still uniformly controlled by the very same argument invoked above) is bounded by the right hand side of the last inequality, as well.

This gives the improved bound 
$$
\lambda_0(C_{\eps,\kappa_{\eps,3}}) 
\geq 
\lambda_1(\Sigma) -  \frac{\lambda_0(C_{\eps,\kappa_{\eps,3}})^2 \phi_0(x_0)^2}{\lambda_1(\Sigma)^2} \frac{32a_{\eps,\kappa_{\eps,3},0}^2}{\theta d_0 \eps \log(1/\eps)} -  \frac{(1+\theta)^3}{4}d_0 \eps \log(1/\eps)
$$
if 
$\eps \leq \eps_3$, where we have once again by our choice of $\alpha_0$ that the bound on $c_{\eps,3}$ is worse than $\eps^{257/256}$.

If we can justify to iterate this argument $l$-times we will get the bound
$$
\lambda_0(C_{\eps,\kappa}) 
\geq
\lambda_1(\Sigma) -  \frac{\lambda_0(C_{\eps,\kappa_{\eps,l}})^2 \phi_0(x_0)^2}{\lambda_1(\Sigma)^2} \frac{4^l a_{\eps,\kappa_{\eps,l},0}^2}{\theta d_0 \eps \log(1/\eps)} - \frac{(1+\theta)^l}{2^{l-1}} d_0 \eps \log(1/\eps)
$$
for $\kappa \in [\kappa_{\eps,l},\kappa_1]$, where $\kappa_{\eps,l} \geq \kappa_{\eps,l-1}$ is the smallest parameter such that
\begin{equation} \label{eq_choice_mass}
\frac{\lambda_0(C_{\eps,\kappa_{\eps,l}})\phi_0(x_0)}{\lambda_1(\Sigma)} \frac{n_{\eps,\kappa_{\eps,l},1}}{m_{\eps,\kappa_{\eps,l},1}} a_{\eps,\kappa_{\eps,l},0}= (1+\theta) \frac{(1+\theta)^l}{2^l} d_0 \eps \log(1/\eps)
\end{equation}
as long as $\eps \leq \eps_l$.
Here, $\eps_l$ is such that
\begin{equation} \label{eq_scale_cond}
\frac{4^l a_{\eps,\kappa,0}^2}{ \theta d_0 \eps \log(1/\eps)} \leq \theta \frac{(1+\theta)^l}{2^l} d_0 \eps \log(1/\eps).
\end{equation}
for all $\kappa \in [\kappa_0,\kappa_1]$
and such that we can absorb the $O(\eps^{129/128})$-terms (which we need to make sure to be still uniform) into the
right hand side above, as well.
(Note that $\eps_{l} \leq \eps_{l-1}$ for any $l$).

Pick $\delta>0$ sufficiently small such that with $L=\lceil \log(\eps^{-\delta}) \rceil$ we have that
$$
\frac{(1+\theta)^L}{2^L} d_0 \eps \log(1/\eps) \geq \eps^{257/256}
$$
for $\eps \leq \eps_*$ with $\eps_*>0$ sufficiently small.
This implies that the $O(\eps^{129/128})$-term is uniform at any $\kappa_{\eps,l}$ for $l \leq L$ since $\alpha_0 < 65/192$.
Similarly, upon decreasing $\delta$ we can assume that 
\eqref{eq_scale_cond} holds for $l=L$ and hence for any $l \leq L$ for $\eps \leq \eps_*$.
Finally, notice that, upon decreasing $\delta$ once again, we have
$$
\lambda_0(C_{\eps,\kappa_{\eps,L}}) \geq \lambda_1(\Sigma) - C \eps^{3\alpha - c \delta} - C \eps^{1+c' \delta} \log(1/\eps)
=
\lambda_1(\Sigma)- o(\eps),
$$ 
at $\kappa_{\eps,L}$ since $\alpha>1/3$.
where the constants come from the base changes in logarithms,
while, at the same time we still have
$$
\frac{
n_{\eps,\kappa_{\eps,L},1} }{m_{\eps,\kappa_{\eps,L},1}} \to \infty,
$$
which clearly implies the assertion.
\end{proof}

\subsection{Proof of the main technical theorem} \label{sec_proofs_main}

Given \cref{prop_iter} the proof of \cref{thm_main_technical} is very short.

\begin{proof}[Proof of \cref{thm_main_technical}]
We assume that 
$$
\lambda_1(\Sigma_{\eps,\kappa}) \leq \lambda_1(\Sigma) - \eps^{k/4}
$$
for any $\kappa \in [\kappa_0,\kappa_1]$, since there is nothing to prove otherwise.
For $\eps>0$ small, take $\kappa_\eps$ given by \cref{L^2_lem_loc_scale} such that
$$
n_{\eps,\kappa_\eps,1}/m_{\eps,\kappa_\eps,1}=1.
$$
It follows from \cref{prop_iter} that 
$$
\lambda_1(\Sigma) \leq \lambda_0(C_{\eps,\kappa_\eps}) + o(\eps)
$$
as $\eps \to 0$.
Moreover, using $m_{\eps,\kappa,1}^2+n_{\eps,\kappa,1}^2=1 + o(1)$, we get from \cref{expan_cyl} that 
$$
\lambda_0(C_{\eps,\kappa_{\eps}}) \leq \lambda_1(\Sigma_{\eps,\kappa_\eps}) - C \eps^{3\alpha/2+1/2}.
$$
Combining these two estimates we arrive at
$$
\lambda_1(\Sigma) \leq \lambda_1(\Sigma_{\eps,\kappa_\eps}) + o(\eps).
$$
On the other hand, we have that
$$
\area(\Sigma_{\eps,\kappa_\eps}) \geq \area(\Sigma) - C \eps^{2k} + c \eps,
$$
where $C=C(\Sigma)$ and $c=c(\kappa_1)>0$.
In summary, we have that
$$
\lambda_1(\Sigma_{\eps,\kappa_\eps}) \area(\Sigma_{\eps,\kappa_\eps})
\geq \lambda_1(\Sigma)\area(\Sigma) + c \lambda_1(\Sigma_{\eps,\kappa_\eps}) \eps + o(\eps),
$$
which implies the theorem since $\lambda_1(\Sigma_{\eps,\kappa_\eps})c\geq c'>0$.
\end{proof}

\begin{rem}
The argument for cross caps is completely identical. 
The crucial properties of the metrics on the cylinder that we have used are as follows.
\\
\begin{inparaenum}[1)]
\item
The separation of two consecutive Dirichlet eigenvalues is on scale $\sim l \eps^{2 \alpha}$, which is much larger than
the scale on which the corresponding quasimodes fail to solve the eigenvalue equation, which is $\sim l \eps^{3\alpha/2+1/2}$.
(The results in \cref{sec_conc} rely on this.)
\\
\item
The first non-trivial Dirichlet and Neumann eigenvalues satisfy
$$
\lambda_0(C_{\eps,\kappa}) \leq \mu_1(C_{\eps,\kappa})
$$
(this was used in \cref{sec_asymp}).
By \cref{lem_basic_spec_comp} this holds for $M_{\eps,\kappa}$ as well.
\end{inparaenum}
\end{rem}

\section{Proofs of the main results} \label{sec_proofs_main_2}

We finally conclude \cref{thm_main} from \cref{thm_main_technical} by a smoothing argument.

\begin{proof}[Proof of \cref{thm_main}]
We take $\eps$ and $\kappa$ given by \cref{thm_main_technical} such that
$$
\lambda_1(\Sigma_{\eps,\kappa}) \area(\Sigma_{\eps,\kappa}) > \lambda_1(\Sigma) \area(\Sigma).
$$
For ease of notation, we simply call this surface $\Sigma'$ with metric $g'$.
In order to conclude \cref{thm_main}, we now construct a sequence of smooth metrics $g_j'$ on $\Sigma'$
such that 
$$
\lambda_1(\Sigma',g_j') \to \lambda_1(\Sigma',g')
$$ 
and
$$
\area(\Sigma',g_j') \to \area(\Sigma',g').
$$

The metric $g'$ is non-smooth near the conical singularities of $(\Sigma,g)$ and at the boundary of the attached cylinder along
$\partial B_{\eps^k}(x_i)$.
Let us first explain how to smooth out the singular set near the boundary of the cylinder.
We restrict to the case $x=x_0$, the case $x=x_1$ is completely analogous.
The argument below is similar to the last step in the proof of \cref{lem_pt_bd}, where we have been slightly less careful since we did
not need much regularity of the metric used there.

By scaling the canonical global coordinates on $C_{\eps,\kappa}$, we find a neighborhood $U= \Sph^1 \times [1,2)  \subset C_{\eps,\kappa}$ of the corresponding boundary component $\partial_1C_{\eps,\kappa}$, which comes with coordinates $(x,y) \in \IR/(2 \pi \IZ) \times [1,2)$
in which the metric is given by $(\kappa y)^{-2}(dx^2+dy^2)$.
We denote by $A_{r,R}$ the annulus $B_R(0) \setminus B_r(0) \subset \IR^2$, where $0<r<R$ .
The map $\Phi \colon U \to A_{1,2}$ given by $\Phi(x,y) = e^{-1} \exp(y+ix)$ defines a conformal diffeomorphism.
Similarly, we have fixed from the very beginning a neighborhood $V$ of $\partial B_{\eps^k}(x_0)$ in $\Sigma$ together with coordinates  mapping to $A_{\eps^k,2\eps^k}$.
We first dilate these coordinates by $\eps^{-k}$ and then compose them with the map $z \mapsto 1/z$ to obtain conformal coordinates $\Psi \colon B_{2\eps^k} \setminus B_{\eps^k} \to A_{1/2,1}$ that map $\partial B_{\eps^k}$ to $\partial B_1$.
Therefore, we can combine $\Phi$ and $\Psi$ to obtain coordinates $\Theta \colon W \to A_{1/2,2}$ on an open neighborhood of $\partial B_{\eps^k}$ endowed with coordinates  such that
$$
\Theta^* g_{\eps,\kappa} = f (dx^2+dy^2)
$$
for a positive bounded function $f$, which is smooth away from $\partial B_1$.
From here on, we can simply use a sequence of smooth positive functions $f_k$ that agree with $f$ near $\partial A_{1/2,2}$ and such that
$f_k \to f$ in $L^2(A_{1/2,2})$ and take the associated smooth metrics 
$$
g_k = f_k (dx^2+dy^2).
$$
Since $g_k$ agrees with $g_{\eps,\kappa}$ near $\partial W$ we obtain a sequence of metrics $g_{\eps,\kappa}^{(k)}$ on $\Sigma_{\eps,\kappa}$ such that $g_{\eps,\kappa}^{(k)} \to g_{\eps,\kappa}$ in the sense of \cite[Definition 4.1]{kokarev}.
In particular,
$$
\lambda_1(\Sigma_{\eps,\kappa}, g_{\eps,\kappa}^{(k)}) \to \lambda_1(\Sigma_{\eps,\kappa},g_{\eps,\kappa}).
$$
by the work of Kokarev \cite[Lemma 4.1]{kokarev}. 
By construction, we also have 
$$
\area(\Sigma_{\eps,\kappa}, g_{\eps,\kappa}^{(k)}) \to \area(\Sigma_{\eps,\kappa},g_{\eps,\kappa}).
$$
Therefore, for $k$ sufficiently large, the surface $(\Sigma_{\eps,\kappa},g_{\eps,\kappa}^{(k)})$ will still have 
$$
\lambda_1(\Sigma_{\eps,\kappa}, g_{\eps,\kappa}^{(k)}) \area(\Sigma_{\eps,\kappa}, g_{\eps,\kappa}^{(k)})  > \lambda_1(\Sigma)\area(\Sigma).
$$
A very similar argument applies to the conical singularities since the metric near these is given by $f(z)(|z|^k+o(|z|^k) |dz|^2$
with $f$ a smooth positive function.
\end{proof}

\begin{rem}
In the case of attaching $M_{\eps,\kappa}$, we also have to smooth out the singularities of the metric on $M_{\eps,\kappa}$.
This can be done in the very same way as above.
\end{rem}

\begin{proof}[Proof of \cref{thm_max}]
Since the maximizing metric $g$ in \cite[Theorem 2]{petrides} is given by 
$$
\frac{|\nabla \Phi|_h}{\lambda_1(\Sigma,g)} h,
$$
where $\Phi \colon (\Sigma,h) \to \mathbb{S}^N$ is a harmonic map and $h$ a smooth metric, the isolated conical singularities are exactly of the type covered by \cref{thm_main}.
Hence, a straightforward induction argument combining \cite[Theorem 2]{petrides} and \cref{thm_main} gives the existence of the maximizing metrics, smooth away from at most finitely many conical singularities for orientable surfaces of any genus.
Since the sharp bound for the normalized first eigenvalue of the real projective plane is $12 \pi$ (see \cite{li_yau}), the case of non-orientable surfaces
follows in the very same fashion by induction using also the existence result in the orientable case.
\end{proof}

\appendix

\section{Proof of \cref{lem_anne_quasimod}} \label{sec_anne}

\begin{proof}
We write $g=g_1+g_2$,
where
$$
g_1 = \sum_{\{l \colon \lambda_l(\Sigma_{\eps,\kappa}) < \lambda -s\}} \langle f , u_{\eps,\kappa,l} \rangle_{L^2(\Sigma_{\eps,\kappa})} u_{\eps,\kappa,l}
$$
and $g_2=g-g_1$.
Note that
$$
\int_{\Sigma_{\eps,\kappa}} \nabla g_i \cdot \nabla f = \int_{\Sigma_{\eps,\kappa}} |\nabla g_i|^2
$$
and
$$
\int_{\Sigma_{\eps,\kappa}} g_i f = \int_{\Sigma_{\eps,\kappa}} |g_i|^2.
$$
Therefore, we find from the assumption that
$$
\int_{\Sigma_{\eps,\kappa}} |\nabla g_i|^2 \leq \lambda \int_{\Sigma_{\eps,\kappa}} |g_i|^2 + \delta \|g_i\|_{W^{1,2}(\Sigma_{\eps,\kappa})},
$$
which implies that
$$
\|g_i\|_{W^{1,2}(\Sigma_{\eps,\kappa})}^2 \leq (\lambda+1) \|g_i \|_{L^2(\Sigma_{\eps,\kappa})}^2 + \delta \|g_i\|_{W^{1,2}(\Sigma_{\eps,\kappa})}.
$$
This in turn implies that
$$
\|g_i\|_{W^{1,2}(\Sigma_{\eps,\kappa})} \leq  (\lambda+1)\|g_i\|_{L^2(\Sigma_{\eps,\kappa})} + \delta.
$$
We now distinguish two cases.
Since we assume $s \leq 1$, the conclusion trivially holds if $\|g_i\|_{L^2(\Sigma_{\eps,\kappa})} \leq \delta$.
If $\|g_i\|_{L^2(\Sigma_{\eps,\kappa})} \geq \delta$, the previous computation implies that we have
$$
\|g_i\|_{W^{1,2}(\Sigma_{\eps,\kappa})} \leq  (\lambda+2)\|g_i\|_{L^2(\Sigma_{\eps,\kappa})} 
$$
Thus, testing against $g_i$, we find that
$$
s \|g_i\|_{L^2(\Sigma_{\eps,\kappa})}^2 \leq  (\lambda+2) \delta \|g_i\|_{L^2(\Sigma_{\eps,\kappa})},
$$
from which the lemma easily follows.
\end{proof}

\section{Green's functions} \label{sec_green}

\begin{lemma}
Let $(\Sigma,g)$ be a closed Riemannian surface with $\area(\Sigma,g)=1$ and $z \in \Sigma$, then there is a unique function $G(\cdot,z) \colon \Sigma \setminus \{z\} \to \IR$
such that 
\begin{itemize}
\item[(i)] $\Delta G(\cdot,z) = \delta_z - 1$ in the sense of distributions.
\item[(ii)] In conformal coordinates centered at $z$, such that $g(z) = g_{\it{eucl}}$ in these coordinates, we have that
$$
G(x,z) = \frac{1}{2\pi} \log\left(\frac{1}{|x-z|}\right) + \psi(x),
$$
where $\psi$ is a smooth function with $\psi(0)=0$.
\end{itemize}
\end{lemma}

\begin{proof}
We take conformal coordinates $(U,x)$ centered as $z$ as in the assertion.
Let $\eta \colon \Sigma \to [0,1]$ be a cut-off function that is $1$ near $z$ and has $\supp \eta \subset U$.
Consider the function $f \colon \Sigma \setminus \{z\} \to \IR$ given in $U$ by 
$$
f(x) = \eta(x) \frac{1}{2\pi} \log \left(\frac{1}{|x-z|}\right),
$$
where $|x-z|$ is the Euclidean distance in the coordinates $(U,x)$.
Since these coordinates are conformal and the Laplace operator is conformally covariant in dimension two, it is easy to see that
\begin{equation} \label{eq_green_ortho}
\Delta f
=
\delta_z + h,
\end{equation}
where 
$$
h=2 \nabla \eta \cdot \nabla \frac{1}{2\pi} \log \left(\frac{1}{|x-z|}\right) + \frac{1}{2\pi} \log \left(\frac{1}{|x-z|}\right) \Delta \eta.
$$
is a smooth function defined on all of $\Sigma$.
It follows from \eqref{eq_green_ortho} that
$$
\int_\Sigma h = -1.
$$
Therefore, since $\area(\Sigma,g)=1$, the function $h+1$ is orthogonal to the constants.
Since the constant function are exactly the kernel of $\Delta$ as $\Sigma$ is closed, we can find a smooth function $r \colon \Sigma \to \IR$ which is unique up to the addition of constants with
$$
\Delta r = h+1.
$$
Thus we have
$$
\Delta (f - r) = \delta_y-1.
$$
By adding a constant to $f-r$ we can now easily arrange to have (ii).
Uniqueness follows immediately from the maximum principle.
\end{proof}


\bibliographystyle{alpha}
\bibliography{mybibfile}

\end{document}